\documentclass[twoside,11pt]{article}

\usepackage{jmlr2e}

\usepackage{textcomp}

\usepackage[utf8]{inputenc} % allow utf-8 input
\usepackage[T1]{fontenc}    % use 8-bit T1 fonts
\usepackage{hyperref}       % hyperlinks
\usepackage{url}            % simple URL typesetting
\usepackage{booktabs}       % professional-quality tables
\usepackage{amsfonts}       % blackboard math symbols
\usepackage[english]{babel}
\usepackage{amsmath}
\usepackage{stmaryrd}
\usepackage{nicefrac}       % compact symbols for 1/2, etc.
\usepackage{microtype}      % microtypography
\usepackage{xcolor}         % colors
\usepackage{algpseudocode}
\usepackage{algorithm}
\usepackage{graphicx}
\usepackage{amssymb,times}
\usepackage{ulem}

\floatname{algorithm}{Algorithm}

\usepackage{lastpage}
            \jmlrheading{23}{2022}{1-\pageref{LastPage}}{1/22; Revised 7/22}{8/22}{22-0026}{El Mehdi Achour, François Malgouyres, and Franck Mamalet}

\newcommand{\Kk}{\mathcal{K}}
\newcommand{\noyau}{\mathbf{K}}

\newcommand{\Ss}{\mathcal{S}}
\newcommand{\Sn}{S_N}
\newcommand{\Pp}{\mathcal{P}}
\newcommand{\Qq}{\mathcal{Q}}
\newcommand{\Cc}{\mathcal{C}}
\newcommand{\Bb}{\mathcal{B}}

\newcommand{\rk}{\text{rk}}

\usepackage{bbding}
\newcommand{\Iro}{{\mathbf I_{r0}}}
\newcommand{\RR}{\mathbb{R}}
\newcommand{\NN}{\mathbb{N}}
\newcommand{\Kcal}{\Kk}
\newcommand{\KkS}{\RR^{MN^2 \times CS^2N^2}}    % version 2D
\newcommand{\Kbf}{\noyau}
\newcommand{\KbfS}{\RR^{M\times C\times k\times k}}    % version 2D
\newcommand{\Kbb}{\mathbb K}
\newcommand{\KbbO}{\mathbb K^\perp}
\newcommand{\Kbfname}{kernel tensor}
\newcommand{\Kcalname}{layer transform matrix}

\DeclareMathOperator{\err}{err}
\newcommand{\ERR}[2] {\err^{#1}_{#2}}
\DeclareMathOperator{\Id}{Id}
\DeclareMathOperator{\Vect}{Vect}

\DeclareMathOperator{\conv}{conv}
\DeclareMathOperator{\CONV}{{\bf conv}}
\newcommand{\Yes}{\textcolor{green}{\Checkmark}}
\newcommand{\No}{\textcolor{red}{\XSolidBrush}}

\newboolean{keepsupplementary}
\setboolean{keepsupplementary}{true}  

\newboolean{keepmainpart}
\setboolean{keepmainpart}{true}  

\newcommand{\Deellip}{
   \textit{DEEL.LIP}
}

\begin{document}
\title{Existence, Stability and Scalability of Orthogonal Convolutional Neural Networks}

\author{\name El Mehdi Achour \email el\_mehdi.achour@math.univ-toulouse.fr \\
       \addr Institut de Mathématiques de Toulouse ; UMR 5219\\
          Université de Toulouse ; CNRS \\
          UPS IMT F-31062 Toulouse Cedex 9, France\\
       \AND
       \name François Malgouyres \email francois.malgouyres@math.univ-toulouse.fr \\
       \addr Institut de Mathématiques de Toulouse ; UMR 5219\\
          Université de Toulouse ; CNRS \\
          UPS IMT F-31062 Toulouse Cedex 9, France\\
       \AND
       \name Franck Mamalet \email franck.mamalet@irt-saintexupery.com \\
       \addr Institut de Recherche Technologique Saint Exupéry, Toulouse, France}

\editor{Kilian Weinberger}
\ShortHeadings{Orthogonal Convolutional Neural Networks}{Achour, Malgouyres, and Mamalet}
% \author{El Mehdi Achour$^{a}$$^{*}$, Fran\c{c}ois Malgouyres$^{a}$, Franck Mamalet$^{b}$ \\
%         \small $^{a}$Institut de Math\'ematiques de Toulouse ; UMR 5219\\
%          Universit\'e de Toulouse ; CNRS \\
%           UPS IMT F-31062 Toulouse Cedex 9, France\\
%       \small $^{b}$Institut de Recherche Technologique Saint Exup\'ery, Toulouse, France \\\\
%       \small $^{*}$Corresponding author: El Mehdi Achour; \tt{El\_mehdi.achour@math.univ-toulouse.fr}
%       }

\maketitle

\ifkeepmainpart

\begin{abstract}%
Imposing orthogonality on the layers of neural networks is known to facilitate the learning by limiting the exploding/vanishing of the gradient; decorrelate the features; improve the robustness.  This paper studies the theoretical properties of orthogonal convolutional layers. \\
\indent We establish necessary and sufficient conditions on the layer architecture guaranteeing the existence of an orthogonal convolutional transform. The conditions prove that orthogonal convolutional transforms exist for almost all architectures used in practice for 'circular' padding.
We also exhibit limitations with 'valid' boundary conditions and 'same' boundary conditions with zero-padding. \\
\indent Recently, a regularization term imposing the orthogonality of convolutional layers has been proposed, and impressive empirical results have been obtained in different applications \citep{wang2020orthogonal}.
The second motivation of the present paper is to specify the theory behind this.
We make the link between this regularization term and orthogonality measures. In doing so, we show that this regularization strategy is stable with respect to numerical and optimization errors and that, in the presence of small errors and when the size of the signal/image is large, the convolutional layers remain close to isometric.

The theoretical results are confirmed with experiments and the landscape of the regularization term is studied. Experiments on real data sets show that when orthogonality is used to enforce robustness, the parameter multiplying the regularization term
can be used to tune a tradeoff between accuracy and orthogonality, for the benefit of both accuracy and robustness. \\
\indent Altogether, the study guarantees that the regularization proposed in \citet{wang2020orthogonal} is an efficient, flexible and stable numerical strategy to learn orthogonal convolutional layers.

\end{abstract}


\section{Introduction}\label{sec intro}

We first start by introducing the problem, related work and the context of this paper.
\subsection{On Orthogonal Convolutional Neural Networks}
 Orthogonality constraint has first been considered for fully connected neural networks \citep{arjovsky2016unitary}. For Convolutional Neural Networks~\citep{lecun1995convolutional,krizhevsky2012imagenet,xiang-nips-15}, the introduction of the orthogonality constraint is a way to improve the neural network in several regards. First, despite well-established solutions \citep{he2016deep,ioffe2015batch}, the training of very deep convolutional networks remains difficult. This is in particular due to vanishing/exploding gradient problems \citep{hochreiter1991untersuchungen,bengio1994learning}. As a result, the expressive capacity of convolutional layers is not fully exploited \citep{ioffe2015batch}. This can lead to lower performances on machine learning tasks. Also, the absence of constraint on the convolutional layer often leads to irregular predictions that are prone to adversarial attacks \citep{szegedy2013intriguing,nguyen2015deep}. Gradient vanishing/exploding avoidance, built-in robustness and better generalization capabilities are the main aims of the introduction of  Lipschitz \citep{szegedy2013intriguing,qian2018l2,gouk2021regularisation,tsuzuku2018lipschitz,sedghi2018singular} and orthogonality  constraints to convolutional layers \citep{xie2017all,cisse2017parseval,huang2018orthogonal,zhang2019approximated,li2019preventing,guo2019regularization,qi2020deep,wang2020orthogonal,trockman2021orthogonalizing,jia2019orthogonal,li2019efficient,huang2020controllable,jia2019orthogonal,bansal2018can,Xiao18CNN10000Layers}.  
 Orthogonal convolutional networks have been applied successfully in diverse applications, such as classification, segmentation, inpainting~\citep{wang2020orthogonal, Zhang20EEG,larrazabal2021orthogonal}, or recently in few-shot learning~\citep{osahor2022ortho}.
Orthogonality is also proposed for Generative Adversarial Networks (GAN)\citep{miyato2018spectral}, or even required for Wasserstein distance estimation, such as in Wasserstein-GAN \citep{arjovsky-bottou-2017,gulrajani2017improved}, and Optimal Transport based classifier~\citep{serrurier2021achieving}.

Orthogonal convolutional networks are made of several orthogonal convolutional layers. This means that, when expressing the computation performed by the layer as a matrix, the matrix is orthogonal. The term 'orthogonal' applies both to square and non-square matrices\footnote{The same property is sometimes called 'semi-orthogonal'.}. In the latter case, it includes two commonly distinguished but related notions: row-orthogonality and column-orthogonality. 
This article focuses on the theoretical properties of orthogonal convolutional layers. Furthermore, since deconvolution (also called transposed convolution) layers are defined using convolution layers, the results can also be applied to orthogonal deconvolution layers.
We will consider the architecture of a convolutional layer as characterized by $(M,C,k,S)$, where $M$  is the number of output channels, $C$ of input channels, convolution kernels are of size $k\times k$ and the stride parameter is $S$. Unless we specify otherwise, we consider convolutions with circular boundary conditions\footnote{Before computing a convolution the input channels are made periodic outside their genuine support.}.
Thus, applied on 
input channels of size $SN \times SN$, the  $M$ output channels are of size $N\times N$. We denote by $\Kbf\in\KbfS$ the \Kbfname~and by $\Kk\in\KkS$ the matrix that applies the convolutional layer of architecture $(M,C,k,S)$ to $C$ vectorized channels of size $SN\times SN$.

We will first answer the important questions:
\begin{itemize}
    \item {\bf Existence:} What is a necessary and sufficient condition on  $(M,C,k,S)$ and $N$ such that there exists an orthogonal convolutional layer (i.e. $\Kk$ orthogonal) for this architecture? How do the 'valid' and 'same' boundary conditions restrict the orthogonality existence?
\end{itemize}
Answers to these questions are respectively in Section \ref{sec:existence}, Theorem \ref{Prop existence cco} and Section \ref{other-padding-sec}, Proposition \ref{valid} and Proposition \ref{same}.
  
Besides, we will rely on recently published papers \citep{wang2020orthogonal,qi2020deep} which characterize orthogonal convolutional layers as the zero level set of a particular function that is called $L_{orth}$ in \citet{wang2020orthogonal}\footnote{The situation is more complex in \citep{wang2020orthogonal,qi2020deep}. One of the contributions of the present paper is to clarify the situation. We describe here the clarified statement.} (see Section \ref{sec:lorth} for details). Formally, $\Kk$ is orthogonal if and only if $L_{orth}(\Kbf)=0$. They use $L_{orth}$ as a regularization term and obtain impressive performances on several machine learning tasks \citep[see ][]{wang2020orthogonal}. The regularization is later successfully applied for medical image segmentation \citep{Zhang20EEG}, inpainting \citep{larrazabal2021orthogonal} and few-shot learning \citep{osahor2022ortho}.

In the present paper, we investigate the following theoretical questions:
\begin{itemize}
\item \textbf{Stability with regard to minimization errors:} Does $\Kk$ still have good `approximate orthogonality properties' when $L_{orth}(\Kbf)$ is small but non zero?
Without this guarantee, it could happen that $L_{orth}(\noyau) = 10^{-9}$ and $\|{\mathcal K}{\mathcal K}^T - Id \|_2 = 10^{9}$.
This would make the regularization with $L_{orth}$ useless, unless the algorithm reaches $L_{orth}(\noyau) = 0$.
\item \textbf{Scalability and stability with regard to N:} Remarking that, for a given \Kbfname~$\Kbf$, $L_{orth}(\noyau)$ is independent of $N$ but the \Kcalname~$\Kk$ depends on $N$: When $L_{orth}(\noyau)$ is small, does $\Kk$ remain approximately orthogonal and isometric when $N$ grows? If so, the regularization with $L_{orth}$ remains efficient even for very large $N$.
\item {\bf Optimization:} Does the landscape of $L_{orth}$ lend itself to global optimization?
\end{itemize}

We give a positive answer to these questions, thus showing theoretical bounds proving that the regularization with $L_{orth}$ is stable (see Theorem \ref{prop norme Frobenius}, Theorem \ref{Prop norme spectrale} and Section \ref{stable_expe_sec}), and can be used in most cases to ensure quasi-orthogonality of the convolutional layers (see Section \ref{landscape-sec} and Section \ref{m=CS2_expe_sec}).

We describe the related works in Section \ref{rel-work-sec} and give the main elements of context in Section \ref{context-sec}. The theorems constituting the main contributions of the article are in Section \ref{main-thm-sec}. Experiments illustrating the theorems, on the landscape of $L_{orth}$,  as well as experiments showing the benefits of approximate orthogonality on image classification problems are in Section \ref{exp-sec}. In particular, the latter shows that when orthogonality is used to enforce robustness, the regularization parameter $\lambda$ multiplying $L_{orth}(\Kbf)$ can be used to tune a tradeoff between accuracy and orthogonality, for the benefit of both accuracy and robustness. The code will be made available in the \Deellip\footnote{\url{https://github.com/deel-ai/deel-lip}} library.

For clarity, we only consider convolutional layers applied to images (2D)  in the introduction and the experiments. But we emphasize that the theorems in Section \ref{main-thm-sec} and their proofs are provided for both signals (1D) and images (2D).

\subsection{Related Work and Contributions}\label{rel-work-sec}

Orthogonal matrices form the Stiefel Manifold and were studied in \citet{edelman1998geometry}. In particular, the Stiefel Manifold is compact, smooth and of known dimension. It is made of several connected components. This can be a numerical issue since most algorithms have difficulty changing connected components during optimization. The Stiefel Manifold has many other nice properties that make it suitable for (local) Riemannian optimization \citep{lezcano2019cheap,li2019efficient}. Orthogonal convolutional layers are a subpart of this Stiefel Manifold. To the best of our knowledge, the understanding of orthogonal convolutional layers is weak. There is no paper focusing on the theoretical properties of orthogonal convolutional layers.

Many articles \citep{xu2012robustness,cisse2017parseval,sokolic2017robust,jia2019orthogonal,scaman2018lipschitz,gouk2021regularisation,farnia2018generalizable}
focus on Lipschitz and orthogonality constraints of the neural network layers from a statistical point of view, in particular in the context of adversarial attacks.

Many recent papers have investigated the numerical problem of optimizing a \Kbfname~$\Kbf$ under the constraint that $\Kk$ is orthogonal or approximately orthogonal. They also provide modeling arguments and experiments in favor of this constraint. We can distinguish two main strategies: {\bf kernel orthogonality} \citep{xie2017all,cisse2017parseval,huang2018orthogonal,zhang2019approximated,guo2019regularization,jia2019orthogonal,li2019efficient,huang2020controllable,jia2019orthogonal,bansal2018can,serrurier2021achieving} and {\bf convolutional layer orthogonality} \citep{li2019preventing,qi2020deep,wang2020orthogonal,trockman2021orthogonalizing,projUN}. The latter has been introduced more recently.

We denote the input of the layer by $X\in\RR^{C\times SN\times SN}$ and its output by $Y=\CONV(\Kbf,X) \in\RR^{M\times N\times N}$.
\begin{itemize}
    \item {\bf Kernel Orthogonality:} This class of methods views the convolution as a multiplication between a matrix $\overline{\Kbf} \in\RR^{M\times Ck^2}$ formed by reshaping the kernel tensor $\noyau$ 
    \citep[see, for instance,][for more details]{cisse2017parseval,wang2020orthogonal}
    and the matrix $U(X)\in\RR^{Ck^2 \times N^2}$ whose columns contain the concatenation of the $C$ vectorized patches of $X$ needed to compute the $M$ output channels at a given spatial position \citep[see ][]{heide2015fast,yanai2016efficient}. 
    We therefore have, $\Vect{(Y)} = \Vect{\left(\overline{\Kbf}U(X)\right)}$.
    The kernel orthogonality strategy enforces the orthogonality of the matrix $\overline{\Kbf}$.
    \item {\bf Convolutional Layer Orthogonality:} This class of methods connects the input and the output of the layer directly by writing $\Vect{(Y)} = \Kk \Vect{(X)}$ and enforces the orthogonality of $\Kk$. 
    The difficulty of this method is that the size of the matrix $\Kk \in \mathbb{R}^{MN^2 \times CS^2N^2}$ depends on $N$ and can be very large.
\end{itemize}

Kernel orthogonality provides a numerical strategy whose complexity is independent of $N$.
However, kernel orthogonality does not imply that $\Kk$ is orthogonal. In a nutshell, the problem is that the composition of an orthogonal embedding\footnote{Up to a re-scaling, when considering circular boundary conditions, the mapping $U$ is orthogonal.} and an orthogonal dimensionality reduction has no reason to be orthogonal. This phenomenon has been observed empirically in \citet{li2019preventing} and \citet{jia2019orthogonal}. The authors of \citet{wang2020orthogonal} and \citet{qi2020deep} also argue that, when $\Kcal$ has more columns than rows (row orthogonality), the orthogonality of 
$\overline{\Kbf}$ is necessary but not sufficient to guarantee $\Kk$ orthogonal. 
Kernel orthogonality and convolutional layer orthogonality are different, the latter better avoids gradient vanishing and feature correlation.

We can distinguish between two numerical ways of enforcing orthogonality during training:
\begin{itemize}
    \item {\bf Hard Orthogonality:} This method consists in keeping the matrix of interest orthogonal during the whole training process.
    This can be done either by optimizing on the Stiefel Manifold, or by considering a parameterization of a subset of orthogonal matrices \citep[e.g., ][]{li2019efficient,li2019preventing,trockman2021orthogonalizing,singla2021skew,huang2018orthogonal,zhang2019approximated,projUN}.
    Note that some hard convolutional layer orthogonality methods consider mappings of $\Kcal$, therefore resulting in convolutions with kernels of size larger than $k \times k$.
    \item {\bf Soft Orthogonality:} Another method to impose orthogonality of matrices during the optimization is to add a regularization of the type $\|WW^T-I\|^2$ to the loss of the specific task.
This regularization penalizes the matrices far from orthogonal \citep[e.g., ][]{bansal2018can,cisse2017parseval,qi2020deep,wang2020orthogonal,xie2017all,guo2019regularization,jia2019orthogonal,huang2020controllable}.
\end{itemize}

Note that, unlike Kernel Orthogonality, Convolutional Layer Orthogonality deals directly with $\Kk$, and thus has a complexity that generally depends on $N$. 
However, in the context of Soft Convolutional Layer Orthogonality, the authors of \citet{qi2020deep,wang2020orthogonal} introduce the regularizer $L_{orth}$ which is independent of $N$ (see Section \ref{sec:lorth} for details), as a surrogate to $\|\Kk\Kk^T -\Id_{MN^2}\|^2_F$ and $\|\Kk^T\Kk -\Id_{CS^2N^2}\|^2_F$.
In \citet{wang2020orthogonal}, orthogonal convolutional layers involving a stride are considered for the first time. 

~ \\
The present paper specifies the theory supporting the regularization with $L_{orth}$ and the construction of orthogonal convolutional layers. We give necessary and sufficient conditions on the architecture for the orthogonal convolutional layers to exist (see Theorem \ref{Prop existence cco}); we unify the $L_{orth}$ formulation for both Row-Orthogonality and Column-Orthogonality cases (see Definition \ref{Lorth-def}); and prove that the regularization with $L_{orth}$: 1/ is stable, i.e.  $L_{orth}(\Kbf)$ is small $\Longrightarrow~ \Kk\Kk^T - \Id_{MN^2}$ is small in various senses  (see Theorem \ref{prop norme Frobenius} and Theorem \ref{Prop norme spectrale}); 2/ leads to an orthogonality error that scales favorably when input signal size $N$ grows  (see Theorem \ref{Prop norme spectrale} and Section \ref{stable_expe_sec}).
We empirically show that, in most cases, the landscape of $L_{orth}$ is such that its minimization can be achieved by \textit{Adam}~\citep{Diederik14Adam}, a standard first-order optimizer (see Section \ref{landscape-sec}). We also identify and analyse the problematic cases  (see Section \ref{m=CS2_expe_sec}). We show numerically that approximate orthogonality is preserved when $N$ increases (see Section \ref{stable_expe_sec}).
Finally, we illustrate on Cifar10 and Imagenette data sets how the regularization parameter can be chosen to control the tradeoff between accuracy and orthogonality, for the benefit of both accuracy and robustness (see Section \ref{sec:dataset_exp}).

\subsection{Context}\label{context-sec}
In this section, we describe the context of the article by defining orthogonality, the regularization function $L_{orth}$ and the Frobenius and spectral norms of the orthogonality residuals, which are two measures of approximate orthogonality. We relate the latter to an approximate isometry property whose benefits are listed in Table \ref{ztzrt}. The main notations defined in this section are reminded in Table \ref{Notation_table}, in Appendix \ref{notation_sec}.

\subsubsection{Orthogonality}\label{sec orthogonality main part}

Given a kernel tensor $\noyau$, the convolutional layer transform matrix $\Kk$ can be written as:
\begin{align*}
\mathcal{K} = \left(
\begin{array}{c c c}
  \mathcal{M}(\noyau_{1,1}) & \ldots & \mathcal{M}(\noyau_{1,C}) \\
  \vdots & \vdots & \vdots \\
  \mathcal{M}(\noyau_{M,1}) &\ldots & \mathcal{M}(\noyau_{M,C})
\end{array}
\right) \in \mathbb{R}^{MN^2 \times CS^2N^2} \;,
\end{align*}
where $\mathcal{M}(\noyau_{i,j})$ is a matrix that computes a strided convolution for the kernel $\noyau_{i,j} = \noyau_{i,j,:,:}$, from the input channel $j$, to the output channel $i$ (See Appendix \ref{conv as matrix vec prod} for details). Notice that we use the 'matlab-colon-notation',
such that $\noyau_{i,j,:,:} = (\noyau_{i,j,m,n})_{0\leq m,n\leq k-1}\in\RR^{k\times k}$.

In order to define orthogonal matrices, we need to distinguish two cases:
\begin{itemize}
\item  {\bf Row case (RO case).} When the size of the input space of $\Kk\in\KkS$ is larger than the size of its output space, i.e. $M\leq CS^2$, $\Kk$ is orthogonal if and only if  its rows are  normalized and mutually orthogonal. Denoting the identity matrix $\Id_{MN^2}\in\RR^{MN^2\times MN^2}$, this is written
\begin{equation}\label{ortho_p}
\Kk\Kk^T =\Id_{MN^2}.
\end{equation}
In this case, the mapping $\Kk$ performs a dimensionality reduction.
\item {\bf Column case (CO case).} When $M \geq CS^2$, $\Kk$ is orthogonal if and only if its columns are  normalized and mutually orthogonal:
\begin{equation}\label{ortho_e}
\Kk^T\Kk =\Id_{CS^2N^2}.
\end{equation}
 In this case, the mapping $\Kk$ is an embedding.
\end{itemize}
Both the RO case and CO case are encountered in practice. When $M = CS^2$, the matrix $\Kk$ is square and if it is orthogonal then both \eqref{ortho_p} and \eqref{ortho_e} hold. The matrix $\Kk$ is then orthogonal in the usual sense and both $\Kk$ and $\Kk^T$ are isometric.

\subsubsection{The Function \texorpdfstring{$L_{orth}(\Kbf)$}{Lorth}}
\label{sec:lorth}

In this section, we define a variant of the function  $L_{orth} :\KbfS \longrightarrow \RR$ defined in \citet{wang2020orthogonal,qi2020deep}. The purpose of the proposed variant is to unify the properties of $L_{orth}$ in the RO case and CO case.

Reminding that $k \times k$ is the size of the convolution kernel, for any $h,g \in \mathbb{R}^{k\times k}$ and any $P\in\NN$,  we define $\conv(h,g,\text{padding zero} = P, \text{stride} = 1) \in \RR^{(2P+1)\times(2P+1)}$ as the convolution\footnote{As is common in machine learning, we do not flip $h$.} between  $h$ and the zero-padding of g  (see Figure \ref{fig:con_h_g}).
Formally, for all $ i, j \in \llbracket0,2P \rrbracket$,
\[ [\conv(h,g,\text{padding zero} = P, \text{stride} = 1)]_{i,j} =  \sum_{i'=0}^{k-1}\sum_{j'=0}^{k-1} h_{i',j'} \bar{g}_{i+i',j+j'},
\]
where $\bar{g} \in \mathbb{R}^{(k+2P)\times (k+2P)}$ is defined, for all $(i,j)\in\llbracket 0,k+2P-1 \rrbracket^2 $, by
\[\bar{g}_{i,j} = \left\{\begin{array}{ll}
g_{i-P,j-P} & \mbox{if } (i,j) \in \llbracket P,P+k-1 \rrbracket^2, \\
0 & \mbox{otherwise}.
\end{array}\right.
\]

\begin{figure}
    \centering
    \includegraphics[width=1.0\linewidth]{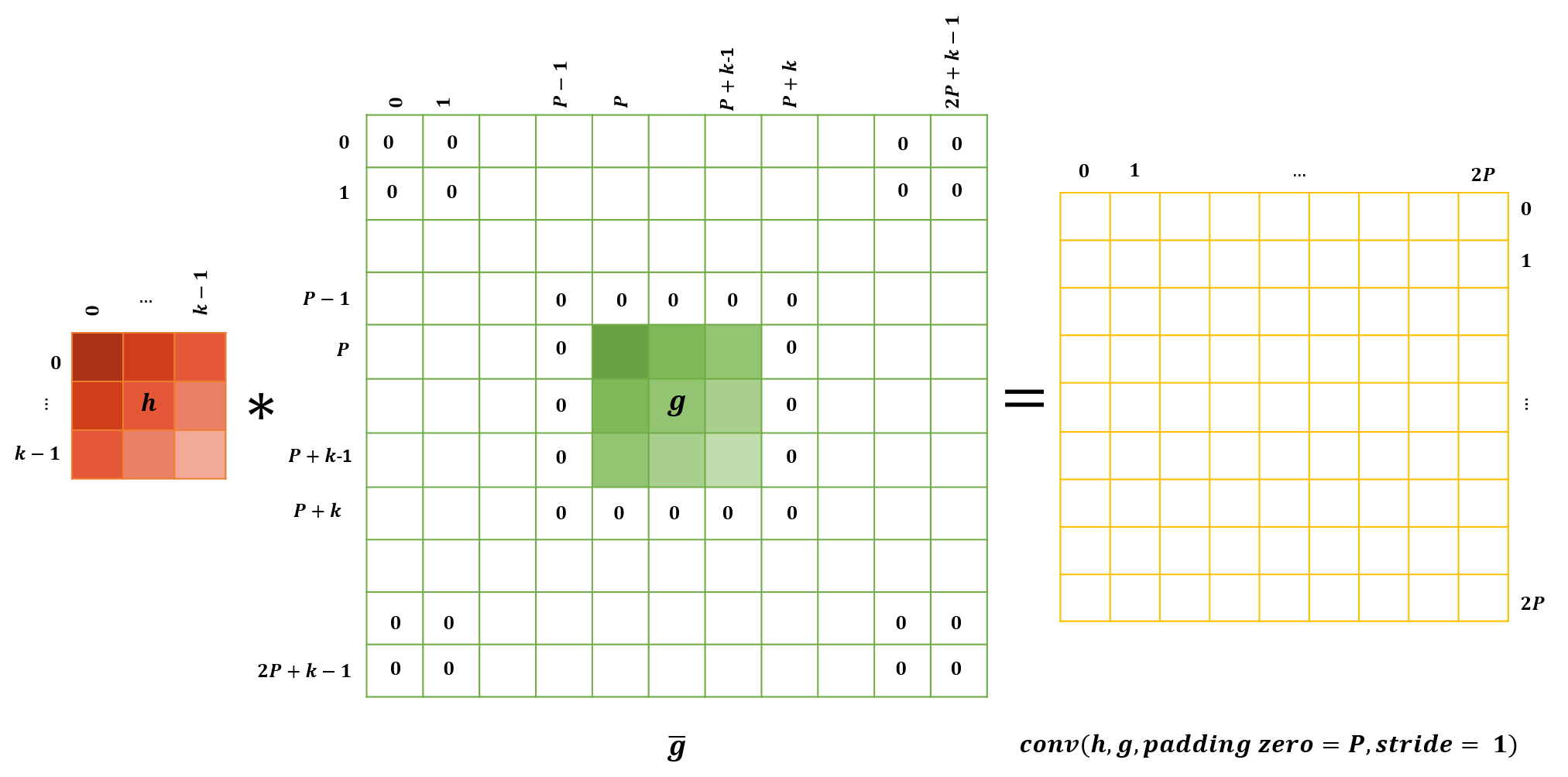}
    \caption{Illustration of $\conv(h,g,\text{padding zero} = P, \text{stride} = 1)$, in the 2D case. 
    }
\label{fig:con_h_g}
\end{figure}

We define $\conv(h,g,\text{padding zero} = P, \text{stride} = S) \in\RR^{(\lfloor 2 P/S \rfloor+1)\times(\lfloor 2 P/S \rfloor+1)}$, for all integer $S\geq 1$ and all $i,j \in  \llbracket 0,\lfloor 2 P/S \rfloor \rrbracket$, by 
\[[\conv(h,g,\text{padding zero} = P, \text{stride} = S)]_{i,j} =  [\conv(h,g,\text{padding zero} = P, \text{stride} = 1)]_{Si,Sj}.
\]

We denote (in bold) $\CONV(\noyau,\noyau,\text{padding zero} = P, \text{stride} = S) \in \mathbb{R}^{M \times M \times (\lfloor 2P/S \rfloor +1) \times (\lfloor 2P/S \rfloor +1)}$ the fourth-order tensor such that, for all $m,l \in \llbracket1,M \rrbracket$, 
\begin{multline*}
{\CONV(\noyau,\noyau,\text{padding zero} = P, \text{stride} = S)}_{m,l,:,:} \\
= \sum_{c=1}^C \conv(\noyau_{m,c},\noyau_{l,c},\text{padding zero} = P, \text{stride} = S) ,
\end{multline*}
where, for all $m\in\llbracket1,M \rrbracket$ and $c\in\llbracket1,C \rrbracket$, $\Kbf_{m,c}=\Kbf_{m,c,:,:}  \in\RR^{k\times k}$.\\

It has been noted in \citet{wang2020orthogonal} that, in the RO case, when $P = \left\lfloor \frac{k-1}{S} \right\rfloor S$,
\begin{align}\label{eq RO conv=Ir0}
  \Kk \quad \mbox{ orthogonal} \qquad \Longleftrightarrow \qquad \CONV(\noyau,\noyau,\text{padding zero} = P, \text{stride} = S) = \Iro ,  
\end{align}
where $\Iro \in \mathbb{R}^{M \times M \times (2 P/S +1)\times (2 P/S +1)}$ is the tensor whose entries are all zero except its central  $M \times M$ entry which is equal to an
identity matrix: $[\Iro]_{:,:, P/S, P/S } = Id_{M} $.

Therefore, denoting by $\|.\|_F$ the Euclidean norm in high-order tensor spaces, it is natural to define the following regularization penalty (we justify the CO case right after the definition).

\begin{definition}[$\mathbf{L_{orth}}$]\label{Lorth-def}
We denote by $P = \left\lfloor \frac{k-1}{S} \right\rfloor S$.
We define $L_{orth}: \KbfS\longrightarrow \RR_+$ as follows
\begin{itemize}
\item In the RO case, $M\leq CS^2$: 
\begin{equation}\label{Lorh_RO-eq}
    L_{orth}(\noyau) = {\| \CONV(\noyau,\noyau,\text{padding zero} = P, \text{stride} = S) - \Iro \|}_F^2~.
\end{equation}

\item In the CO case, $M\geq CS^2$: 
\[L_{orth}(\noyau) = {\| \CONV(\noyau,\noyau,\text{padding zero} = P, \text{stride} = S) - \Iro \|}_F^2 - (M-CS^2)~.
\]

\end{itemize}
\end{definition}
When $M=CS^2$, the two definitions trivially coincide. In the definition, the padding parameter $P$ is the largest multiple of $S$ strictly smaller than $k$. The difference with the definitions of $L_{orth}$ in \citet{wang2020orthogonal,qi2020deep} is in the CO case.
In this case with $S=1$, \citet{qi2020deep,wang2020orthogonal} use \eqref{Lorh_RO-eq} with $\noyau^T$ instead of $\Kbf$
.
For $S \geq 2$ in the CO case, we can not derive a simple equality as in \eqref{eq RO conv=Ir0}.
In \citet{wang2020orthogonal}, remarking that $\|\Kcal^T\Kcal - \Id_{CS^2N^2}\|_F^2  - \|\Kcal\Kcal^T - \Id_{MN^2}\|_F^2$ is a constant which only depends on the size of $\Kk$, the authors also argue that, whatever $S$,  one can also use \eqref{Lorh_RO-eq} in the CO case. 
We alter this 
in the CO case as in Definition \ref{Lorth-def} to obtain both in the RO case and the CO case:
\[\Kk \mbox{ orthogonal}  \qquad \Longleftrightarrow \qquad L_{orth}(\noyau) = 0.
\]
Once adapted to our notations, the authors in \citet{wang2020orthogonal,qi2020deep} propose to regularize convolutional layers parameterized by $(\noyau_l)_l$ by optimizing
\begin{equation}\label{model_regularise}
L_{task} + \lambda \sum_{l} L_{orth}(\noyau_l)
\end{equation}
where $L_{task}$ is the original objective function of a machine learning task. The function $ L_{orth}(\noyau)$ does not depend on $N$ and can be implemented in a few lines of code with Neural Network frameworks. Its gradient is then computed using automatic differentiation.

Of course, when doing so, even if the optimization is efficient, we expect $L_{orth}(\noyau_l)$ to be different from $0$ but 
less than $\varepsilon$,
for a small $\varepsilon$.
We investigate, in this article, whether, in this case, the transformation matrix $\Kk$, still satisfies useful orthogonality properties.
To quantify how much $\Kk$ deviates from being orthogonal, we define the approximate orthogonality criteria and approximate isometry property in the next section.
These notions allow to state
the stability and scalability theorems (Sections~\ref{sec:frobenius norm} and~\ref{sec:spectral norm}) and guarantee that the singular values remain close to $1$ when  $L_{orth}$ is small, even when $N$ is large.
This proves that the benefits related to the orthogonality of the layers, which are presented in Table \ref{ztzrt}, still hold.

\subsubsection{Approximate Orthogonality and Approximate Isometry Property}\label{ao-sec}

Perfect orthogonality is an idealization that never happens, due to floating-point arithmetic, and numerical and optimization errors. In order to measure how $\Kk$ deviates from being orthogonal, we define the {\bf orthogonality residual} by $\Kk\Kk^T - \Id_{MN^2}$, in the RO case, and $\Kk^T\Kk -\Id_{CS^2N^2}$, in the CO case.
Considering both the Frobenius norm $\|.\|_F$ of the orthogonality residual and its spectral norm  $\|.\|_2$, we have two criteria:
\begin{equation}\label{Frobenius_ortho_error}
\ERR{F}{N}(\Kbf) = \left\{\begin{array}{ll}
\|\Kk\Kk^T -\Id_{MN^2}\|_F & \mbox{, in the RO case,}\\
\|\Kk^T\Kk -\Id_{CS^2N^2}\|_F& \mbox{, in the CO case,}
\end{array}\right.
\end{equation}
and
\begin{equation}\label{spectral_ortho_error}
\ERR{s}{N}(\Kbf) = \left\{\begin{array}{ll}
\|\Kk\Kk^T -\Id_{MN^2}\|_2 & \mbox{, in the RO case,}\\
\|\Kk^T\Kk -\Id_{CS^2N^2}\|_2& \mbox{, in the CO case.}
\end{array}\right.
\end{equation}
When $M=CS^2$, the definitions in the RO case and the CO case coincide. 
The two criteria are of course related since for any matrix $A\in\RR^{a\times b}$, the Frobenius and spectral norms are such that
\begin{equation}\label{norm_equiv}
\|A\|_F \leq\sqrt{ \min(a,b)}\|A\|_2 \qquad \mbox{and}\qquad\|A\|_2 \leq  \|A\|_F~.
\end{equation}
However, the link is weak, when $\min(a,b)$ is large.

The regularization with $(\ERR{F}{N}(\Kbf))^2$ is a natural way to enforce soft-orthogonality of $\Kcal$. However, as mentioned in the introduction, it is not practical because the sizes of $\Kcal$ are too large. We will see in Theorem \ref{prop norme Frobenius} that $(\ERR{F}{N}(\Kbf))^2$ and $L_{orth}(\Kbf)$ differ by a multiplicative constant and it will make a clear connection between the regularization with $L_{orth}(\Kbf)$ and the regularization with $(\ERR{F}{N}(\Kbf))^2$. However, $\ERR{F}{N}(\Kbf)$ is difficult to interpret, this is why we consider $\ERR{s}{N}(\Kbf)$ which relates to the {\it approximate isometry property} of $\Kcal$ as we explain below. The latter has direct consequences on the properties of the layer
(see Table \ref{ztzrt}).

Indeed, in the applications, one key property of orthogonal operators is their connection to isometries. It is the property that prevents the gradient from exploding and vanishing \citep{cisse2017parseval,Xiao18CNN10000Layers,li2019efficient,huang2020controllable}. This property also enables to keep the examples well separated, which has an effect similar to the batch normalization \citep{qi2020deep,chai2020separating}, and to have a $1$-Lipschitz forward pass and therefore improves robustness \citep{wang2020orthogonal,cisse2017parseval,li2019preventing,trockman2021orthogonalizing,jia2019orthogonal}.

We denote the Euclidean norm of a vector by $\|.\|$. 
To clarify the connection between orthogonality and isometry, we define the `$\varepsilon$-Approximate Isometry Property' ($\varepsilon$-AIP).

\begin{definition}\label{e-aip-def}
A \Kcalname~$\Kk\in\KkS$ satisfies the $\varepsilon$-Approximate Isometry Property if and only if
\begin{itemize}
\item RO case, $M\leq CS^2$: 
\begin{equation*}%\label{aip_ro}
\left\{\begin{array}{ll}
\forall x\in\RR^{CS^2N^2}& \|\Kk x \|^2 \leq (1+\varepsilon)  \|x\|^2 \\
\forall y\in\RR^{MN^2}& (1-\varepsilon)  \|y\|^2  \leq \|\Kk^T y \|^2 \leq (1+\varepsilon)  \|y\|^2
\end{array}\right.
\end{equation*}
\item CO case, $M\geq CS^2$:
\begin{equation*}%\label{aip_co}
\left\{\begin{array}{ll}
\forall x\in\RR^{CS^2N^2}& (1-\varepsilon)  \|x\|^2  \leq \|\Kk x \|^2 \leq (1+\varepsilon)  \|x\|^2 \\
\forall y\in\RR^{MN^2}& \|\Kk^T y \|^2 \leq (1+\varepsilon)  \|y\|^2 
\end{array}\right.
\end{equation*}
\end{itemize}
\end{definition}

The following proposition makes the link between $\ERR{s}{N}(\Kbf)$ and AIP. It shows that minimizing $\ERR{s}{N}(\Kbf)$ enhances the AIP property.  
\begin{proposition}\label{lemme lipchitz norme spectrale}
    Let $N$ be such that $SN \geq k $. We have, both in the RO case and CO case,
  	\[ \Kk \mbox{ is } {\ERR{s}{N}(\Kbf)}\mbox{-AIP}.
	\]
\end{proposition}
This statement actually holds for any matrix (not only \Kcalname) and is already stated in \citet{bansal2018can,guo2019regularization}. For completeness, we provide proof, in Appendix \ref{proof lemme lipchitz norme spectrale}.

In Proposition \ref{lemme lipchitz norme spectrale} and in Theorem \ref{Prop existence cco} (see the next section), the condition $SN \geq k $ only states that the input 
width and height are larger than the size of the kernels. This is always the case in practice.

\begin{table}
\begin{center}
\begin{tabular}{llcccc}
%\hline
\toprule
& & \multicolumn{2}{c}{Forward pass} &  \multicolumn{2}{c}{Backward pass}\\
%\hline
\cline{3-6}
& & Lipschitz& Keep examples&  Prevent & Prevent  \\
& &Forward pass& separated & grad. expl.  & grad. vanish.  \\
\hline
Convolutional & $M<CS^2$ &  \Yes & \No&  \Yes &  \Yes   \\
\cline{2-6} 
layer  & $M> CS^2$ &  \Yes & \Yes& \Yes & \No  \\
\hline
Deconvolution &$M<CS^2$ &  \Yes &\Yes&  \Yes &\No\\
\cline{2-6} 

layer& $M>CS^2$ &  \Yes & \No&\Yes & \Yes\\
\hline
Conv. \& Deconv.  & $M=CS^2$ &  \Yes & \Yes& \Yes & \Yes  \\
%\hline
\bottomrule
\end{tabular}
\end{center}
\caption{Properties of a $\varepsilon$-AIP layers (when $\varepsilon \ll 1$), depending on whether $\Kbf$ defines a convolutional or deconvolutional layer. The red crosses indicate when the forward or backward pass 
performs a dimensionality reduction.\label{ztzrt}}
\end{table}

We summarize in Table \ref{ztzrt} the properties of $\varepsilon$-AIP layers when $\varepsilon$ is small, in the different possible scenarios.
We remind that a \Kbfname~$\Kbf$ can define a convolutional layer or a deconvolution layer. Deconvolution layers are, for instance, used to define layers of the decoder of an auto-encoder or variational auto-encoder~\citep{Kingma2014VAE}. In the convolutional case, $\Kcal$ is applied during the forward pass and $\Kcal^T$ is applied during the backward pass. In a deconvolution layer, $\Kcal^T$ is applied during the forward pass and $\Kcal$ during the backward pass. Depending on whether we have $M<CS^2$, $M>CS^2$ or $M=CS^2$, when $\Kk$ is $\varepsilon$-AIP with $\varepsilon<<1$, either $\Kcal^T$, $\Kcal$ or both preserve distances (see Table \ref{ztzrt}).

To complement Table \ref{ztzrt}, notice that in the RO case, if $\ERR{F}{N}(\Kbf) \leq \varepsilon$, then for any $i$, $j$ with $i\neq j$, we have $ |\Kk_{i,:} \Kk^T_{j,:}  |\leq  \varepsilon$, where $\Kk_{i,:}$ is the $i^{\mbox{th}}$ line of $\Kk$. In other words, when $\varepsilon$ is small, the features computed by $\Kk$ are mostly uncorrelated \citep{wang2020orthogonal}.

\section{Theoretical Analysis of Orthogonal Convolutional Layers}\label{main-thm-sec}

This section contains the theoretical contributions of the article.  In all the theorems in this section, the considered convolutional layers are either applied to a signal, when $d=1$, or an image, when $d=2$.

We remind that the architecture of the layer is characterized by $(M,C,k,S)$ where: $M$ is the number of output channels; $C$ is the number of input channels; $k\geq 1$ is an odd positive integer and the convolution kernels are of size $k$, when $d=1$, and $k\times k$, when $d=2$; the stride parameter is $S$.

We want to highlight that the theorems of Sections \ref{sec:existence}, \ref{sec:frobenius norm} and \ref{sec:spectral norm} are for convolution operators defined with circular boundary conditions (see Appendix \ref{conv as matrix vec prod} for details).
We point out in Section \ref{other-padding-sec} restrictions for the `valid' and `same' zero-padding boundary conditions (see Appendix \ref{Proof of proposition valid} and Appendix \ref{Proof of Proposition same} for details).

With circular boundary conditions, all input channels are of size $SN$, when $d=1$, $SN\times SN$, when $d=2$. The output channels are of size $N$ and $N\times N$, respectively when $d=1$ and $2$.
When $d=1$, the definitions of $L_{orth}$, $\ERR{F}{N}$ and $\ERR{s}{N}$ are in Appendix \ref{notation1D-sec}.

In Section \ref{sec:existence}, we state a theorem that provides the necessary and sufficient conditions on the architecture for an orthogonal convolutional layer to exist. In Section \ref{other-padding-sec}, we describe restrictions for the 'valid' and 'same' boundary conditions.
In Section \ref{sec:frobenius norm}, we state a theorem that provides a relation between the Frobenius norm of the orthogonality residual and the regularization penalty $L_{orth}$.
Finally, in Section \ref{sec:spectral norm}, we state a theorem that provides an upper bound of the spectral norm of the orthogonality residual using the regularization penalty $L_{orth}$.

\subsection{Existence of Orthogonal Convolutional Layers}
\label{sec:existence}
The next theorem gives a necessary and sufficient condition on the architecture 
$(M,C,k,S)$ and $N$ for an orthogonal convolutional layer transform
to exist. To simplify notations, we denote, for $d=1$ or $2$, the space of all the \Kbfname{s} by
\[\Kbb_d = \left\{\begin{array}{ll}
\RR^{M \times C \times k} & \mbox{when }d=1, \\
\RR^{M \times C \times k \times k} & \mbox{when }d=2 .
\end{array}\right.
\]

We also denote, for $d=1$ or $2$,
\[\KbbO_d = \{\Kbf \in\Kbb_d |~ \Kcal\mbox{ is orthogonal}\}.
\]
\begin{theorem}\label{Prop existence cco}
Let $N$ be such that $SN \geq k $ and $d=1$ or $2$. 
    \begin{itemize}
        \item RO case, i.e. $M \leq CS^d$:
        $~~\KbbO_d  \neq\emptyset ~~ \mbox{ if and only if }~~ M \leq Ck^d\;.
        $
        \item CO case, i.e. $M \geq CS^d$:
        $~~\KbbO_d  \neq \emptyset~~ \mbox{ if and only if }~~ S \leq k\;.
        $
    \end{itemize}
\end{theorem}
Theorem \ref{Prop existence cco} is proved in Appendix \ref{proof Prop existence cco}.
Again, the conditions coincide when $M=CS^d$. 

When $S \leq k$, which is by far the most common situation, there exist orthogonal convolutional layers in both the CO case and the RO case. Indeed, in the RO case, when $S \leq k$, we have $M \leq CS^d \leq Ck^d$.

However, skip-connection layers (also called shortcut connection) with stride in Resnet~\citep{he2016deep} for instance, usually have an architecture $(M,C,k,S)$ = $(2C,C,1,2)$, where $C$ is the number of input channels.
The kernels are of size $1\times 1$. In that case, $M\leq CS^d $ and $M>Ck^d $. Theorem \ref{Prop existence cco} says that there is no orthogonal convolutional layer for this type of layer. 

To conclude, the main consequence of Theorem \ref{Prop existence cco} is that, with circular boundary conditions and for most of the architecture used in practice (with an exception for the skip-connections with stride), there exist orthogonal convolutional layers.

%%%%%%%%%%%%%%%%%%%%%%%%%%%%%%%%%%%%%%%
\subsection{Restrictions due to Boundary Conditions}\label{other-padding-sec}
In Sections \ref{sec:existence}, \ref{sec:frobenius norm} and \ref{sec:spectral norm}, we consider convolutions defined with circular boundary conditions.
This choice is neither for technical reasons nor to enable the use of the Fourier basis.
We illustrate in the next two propositions that, for convolutions defined with the 'valid' condition, or the 'same' condition with zero-padding, hard-orthogonality is in many situations too restrictive. We consider in this section an unstrided convolution.

Before stating the next proposition, we remind that with the 'valid' boundary conditions, only the entries of the output such that the support of the translated kernel is entirely included in the input's domain are computed. The size of each
output channel is smaller than the size of the input channels. The formal definition of the 'valid' boundary conditions is at the beginning of Appendix \ref{Proof of proposition valid}.
\begin{proposition}\label{valid}
    Let $N \geq 2k-1$.
    With the 'valid' condition, there exists no orthogonal convolutional layer in the CO case.
\end{proposition}
This proposition holds in the 1D and 2D cases.
We give its proof only in the 1D case in Appendix \ref{Proof of proposition valid}.

Before stating the next proposition, we remind that to compute a convolution with the zero-padding 'same' boundary conditions, we first extend each input channel with zeros and then compute the convolutions such that each output channel has the same support as the input channels, before padding/extension. A formal definition of the zero-padding 'same' boundary condition is at the beginning of Appendix \ref{Proof of Proposition same}. 

Let ${(e_{i,j})}_{i=0..k-1,j=0..k-1}$ be the canonical basis of $\RR^{k \times k}$. For the zero-padding 'same', we have the following proposition.
\begin{proposition}\label{same}
    Let $N \geq k$.
    For $\noyau \in \RR^{M \times C \times k \times k}$, with the zero-padding 'same' and $S=1$, both in the RO case and CO case, if $\Kk$ is orthogonal then there exist $(\alpha_{m,c})_{m=1..M,c=1..C} \in \RR^{M \times C}$ such that for all $(m,c) \in \llbracket 1,M \rrbracket \times \llbracket 1,C \rrbracket$, $\noyau_{m,c} = \alpha_{m,c} e_{r,r}$, where $r$ satisfies $k=2r+1$. 
    As a consequence
    \begin{align*}
        \Kk = \left(
\begin{array}{c c c}
  \alpha_{1,1}Id_{N^2} & \ldots & \alpha_{1,C}Id_{N^2} \\
  \vdots & \vdots & \vdots \\
  \alpha_{M,1}Id_{N^2} &\ldots & \alpha_{M,C}Id_{N^2}
\end{array}
\right) \in \mathbb{R}^{MN^2 \times CN^2} \;.
    \end{align*}
\end{proposition}
This proposition holds in the 1D and 2D cases.
We give its proof only in the 1D case in Appendix \ref{Proof of Proposition same}.\\

To recapitulate, the results state that with padding 'valid', no orthogonal convolution can be built in the CO case and that for zero-padding 'same', the orthogonal convolution layers are trivial transformations.

Note that, the propositions do not exclude the existence of sufficiently expressive sets of 'approximately orthogonal' convolutional layers with these boundary conditions. A strategy based on soft-orthogonality may still enjoy most of the benefits of orthogonality for these boundary conditions.

%%%%%%%%%%%%%%%%%%%%%%%%%%%%%%%%%%%%%%%
\subsection{Frobenius Norm Stability}\label{sec:frobenius norm}

We recall that the motivation behind this is the following :
The authors of \citet{wang2020orthogonal,qi2020deep} argue that $L_{orth}(\noyau) = 0$ is equivalent to $\Kk$ being orthogonal. However, they do not provide stability guarantees. Without this guarantee, it could happen that $L_{orth}(\noyau) = 10^{-9}$ and $\|{\mathcal K}{\mathcal K}^T - Id \|_F = 10^{9}$. This would make the regularization with $L_{orth}$ useless, unless the algorithm reaches $L_{orth}(\noyau) = 0$.

The following theorem proves that this
cannot occur.
Therefore, if $L_{orth}(\noyau)$ is small, $\ERR{F}{N}(\noyau)$ is small at least for moderate signal sizes. Also, a corollary is that adding $L_{orth}$ as a penalty regularization is equivalent to adding the Frobenius norm of the orthogonality residual.

\begin{theorem}\label{prop norme Frobenius}
    Let $N$ be such that $SN \geq 2k-1 $ and $d=1$ or $2$. For a convolutional layer defined using circular boundary conditions, we have, both in the RO case and CO case,
        \[(\ERR{F}{N}(\Kbf))^2 = N^d L_{orth}(\noyau) \;,
        \]
        where $\ERR{F}{N}(\Kbf)$ is defined in \eqref{Frobenius_ortho_error}.
\end{theorem}
Theorem \ref{prop norme Frobenius} is proved, in Appendix \ref{proof prop norme Frobenius}.
We remind that $L_{orth}(\Kbf)$ is independent of $N$.
The theorem formalizes for circular boundary conditions and for both the CO case and the RO case, the reasoning leading to the regularization with $L_{orth}$ in \citet{wang2020orthogonal}.

Using Theorem \ref{prop norme Frobenius}, we find that \eqref{model_regularise} becomes
\[L_{task} + \sum_{l} \frac{\lambda}{N_l^d}  (\ERR{F}{N_l}(\noyau_l))^2.
\]
Once the parameter $\lambda$  is made dependent on the input size of layer $l$, the regularization term $\lambda L_{orth}$ is equal to the Frobenius norm of the orthogonality residual. This justifies the use of $L_{orth}$ as a regularizer.

We can also see from Theorem \ref{prop norme Frobenius} that, for both the RO case and the CO case, when $L_{orth}(\noyau) =0$, $\Kcal$ is orthogonal, independently of $N$.
This recovers the result stated in \citet{qi2020deep} for $S=1$, and the result stated in \citet{wang2020orthogonal} in the RO case for any $S$ .

Considering another signal size $N'$ and applying Theorem \ref{prop norme Frobenius} with the sizes $N$ and $N'$, we find
\[(\ERR{F}{N'}(\noyau))^2 = \frac{(N')^d}{N^d}(\ERR{F}{N}(\noyau))^2.
\]
To the best of our knowledge, this equality is new. This could be of importance in situations when $N$ varies. For instance when the neural network is learned on a data set containing signals/images of a given size, but the inference is done for signals/images of varying size \citep{ren2015faster,Shelhamer2017FCN,ji2019invariant}.

Finally, using \eqref{norm_equiv} and Proposition \ref{lemme lipchitz norme spectrale}, $\Kcal$ is $\epsilon$-AIP with $\epsilon$ scaling like the square root of the signal/image size. This might not be satisfactory. We exhibit in the next section a tighter bound on $\epsilon$, independent of the input size $N$.

%%%%%%%%%%%%%%%%%%%%%%%%%%%%%%%%%%%%%%%

\subsection{Spectral Norm Stability and Scalability}\label{sec:spectral norm}

We prove in Theorem \ref{Prop norme spectrale}  that $\ERR{s}{N}(\noyau)^2$ is sandwiched between two quantities proportional to $L_{orth}(\noyau)$. The multiplicative factors do not depend on $N$. Hence, when $L_{orth}(\noyau)$ is small, $\ERR{s}{N}(\noyau)^2$ is also small for all $N$. As a consequence, as long as  $L_{orth}(\noyau) \ll 1$ even if the algorithm does not reach $L_{orth}(\noyau)=0$, regularizing with $L_{orth}(\noyau)$ permits to construct nearly orthogonal and isometric convolutional layers independently of $N$.

Moreover, combined with Proposition \ref{lemme lipchitz norme spectrale} this ensures that, if $L_{orth}(\noyau)$ is small, $\Kk$ is $\varepsilon$-AIP with $\varepsilon$ small. Using Table \ref{ztzrt}, we see that this property leads to more robustness and avoids gradient vanishing/exploding. This is in line with the empirical results observed in \citet{wang2020orthogonal,qi2020deep}.

\begin{theorem}\label{Prop norme spectrale}
   Let $N$ be such that $SN \geq 2k-1 $ and $d=1$ or $2$. For a convolutional layer defined using circular boundary conditions, we have
        \[\alpha' ~L_{orth}(\noyau) ~ \leq ~(\ERR{s}{N}(\Kbf))^2 ~\leq ~\alpha ~L_{orth}(\noyau)
        \]
        where $\ERR{s}{N}(\Kbf)$ is defined in \eqref{spectral_ortho_error}, for $\alpha' = \frac{1}{\min(M,CS^2)}$ and
        \[ \alpha = \left\{\begin{array}{ll}
         \left(2\left\lfloor \frac{k-1}{S} \right\rfloor+1\right)^dM & \mbox{in the RO case }(M \leq CS^d), \\
         (2k-1)^dC   & \mbox{in the CO case }(M \geq CS^d).
        \end{array}\right.
        \]

\end{theorem}
Theorem \ref{Prop norme spectrale} is proved, in Appendix \ref{proof Prop norme spectrale}.
When $M=CS^d$, the two inequalities hold and it is possible to take the minimum of the two $\alpha$ values.

As we can see from Theorem \ref{Prop norme spectrale}, unlike with the Frobenius norm, the spectral norm of the orthogonality residual is upper-bounded by a quantity that does not depend on $N$. The lower-bound of Theorem \ref{Prop norme spectrale}  guarantees that the upper-bound is tight up to a multiplicative constant. However, we cannot expect much improvement in this regard since the multiplicative constant $\sqrt{\alpha}$ is usually moderately large\footnote{For usual architectures, $\sqrt{\alpha}$ is always smaller than $200$.}. For instance, with $(M,C,k,S) = (128,128,3,2)$, for images, $\sqrt{\alpha} \leq 34$.   
If as is common in practice the optimization algorithm reaches $L_{orth}(\Kbf) \leq 10^{-6}$, Theorem \ref{Prop norme spectrale} guarantees that, independently of $N$, \[\ERR{s}{N}(\Kbf)\leq \sqrt{\alpha L_{orth}(\noyau)}\leq 0.034.\]
Using Proposition \ref{lemme lipchitz norme spectrale}, independently of $N$, a convolutional layer defined with $\noyau$ is $\epsilon$-AIP, for $\epsilon \leq 0.034$, and we have, using Definition \ref{e-aip-def},
\[\left\{\begin{array}{ll}
\forall x\in\RR^{CS^2N^2}& \|\Kk x \| \leq \sqrt{1+\varepsilon}  \|x\| \leq 1.017 \|x\| \\
\forall y\in\RR^{MN^2}& 0.982\|y\|\leq \sqrt{1-\varepsilon}  \|y\|  \leq \|\Kk^T y \| \leq 1.017 \|y\|
\end{array}\right.
\]
The layer benefits from the properties described in Table \ref{ztzrt}.

The development done for the above example can be repeated as soon as  $L_{orth}(\Kbf) \ll 1$, both in the RO case and CO case.
Experiments that confirm this behavior are in Section \ref{exp-sec}.

\section{Experiments}\label{exp-sec}
Before illustrating the benefits of approximate orthogonality to robustly classify images in Section \ref{sec:dataset_exp}, we conduct  several synthetic experiments in Section \ref{synth_exp_sec}. The synthetic experiments empirically evaluate the landscape of $L_{orth}$ in Section \ref{landscape-sec}, \ref{m=CS2_expe_sec} and illustrate the theorems of Section \ref{main-thm-sec}, in Section \ref{stable_expe_sec}.

Both in Section \ref{synth_exp_sec} and Section \ref{sec:dataset_exp}, to evaluate how close $\Kcal$ is to being orthogonal, we compute some of its singular values $\sigma$ for different input sizes $SN\times SN$. When $S=1$, we compute all the singular values of $\Kk$ with the Algorithm~\ref{S=1_algo},  Appendix~\ref{sec:poweriter}, from~\citet{sedghi2018singular}. For convolutions with stride, $S>1$, there is no known practical algorithm to compute all the singular values and we simply apply the well-known power iteration algorithm associated with a spectral shift, to retrieve the smallest and largest singular values $(\sigma_{min},\sigma_{max})$ of $\Kcal$  (see Algorithm~\ref{anyS_algo} in Appendix~\ref{sec:poweriter}).
We remind that $\Kcal$ orthogonal is equivalent to $\sigma_{min}=\sigma_{max}=1$.

\subsection{Synthetic Experiments}\label{synth_exp_sec}
Section \ref{landscape-sec}, \ref{m=CS2_expe_sec} and  \ref{stable_expe_sec} report on results of the massive experiment that is described below.

 In order to avoid interaction with other objectives, we train a single 2D convolutional layer with circular padding. We explore all the architectures such that $\KbbO_2\neq \emptyset$, for $C\in\llbracket1,64 \rrbracket$ , $M\in\llbracket1,64 \rrbracket$, $S\in \{1,2,4\}$, and $k\in \{1,3,5,7\}$. This leads to $44924$  architectures for which an orthogonal convolutional layer exists (among $49152$ architectures in total). 

For each architecture, the model is trained using a \textit{Glorot uniform} initializer and an \textit{Adam} optimizer \citep{Diederik14Adam} with fixed learning rate\footnote{We do not report experiments for other tested learning rates $10^{-1},10^{-3},10^{-4},10^{-5}$ because they lead to the same conclusions.} 0.01 on a null loss ($L_{task}(X,Y,\Kbf) = 0$, for all input X, target Y, and kernel tensor $\Kbf$) and the $L_{orth}(\Kbf)$ regularization (see Definition \ref{Lorth-def}) during $3000$ steps\footnote{Increasing the number of steps leads to the same conclusions.}. 

We report below implementation details that have no influence on the results, since $L_{task} = 0$. 
No data are involved in the synthetic experiments and the input of the layer contains a null input of size $(C,64,64)$. Other input sizes from $8$ to $256$ were tested but not reported,  leading to the same conclusions. Batch size (which thus has no influence on the results) is set to one.

\subsubsection{Optimization Landscape}\label{landscape-sec}

\begin{figure}[ht]
    \centering
    \includegraphics[width=1.0\linewidth]{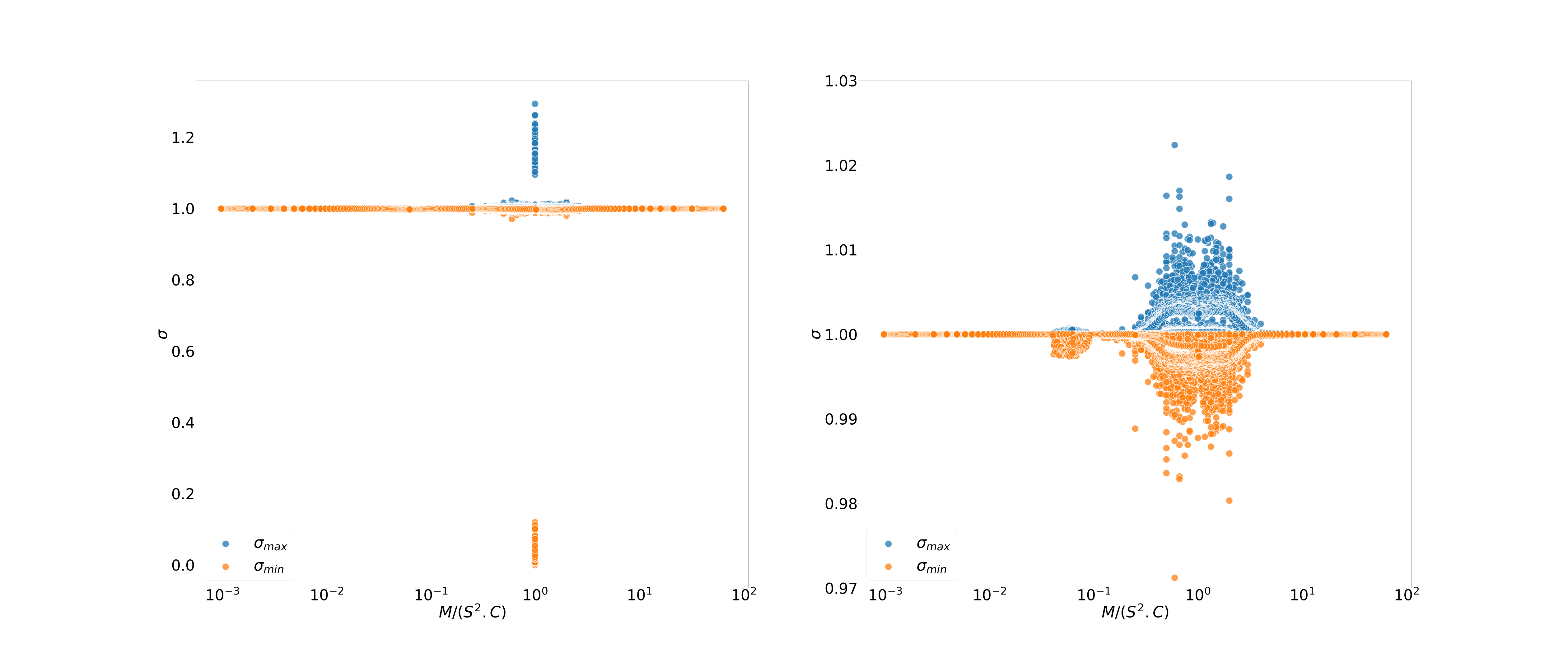}
    \caption{{\bf Optimization of $\mathbf{L_{orth}}$.} Minimization of $L_{orth}(\Kbf)$ for a kernel tensor of architecture $(M,C,k,S)$ among the $44924$ possible configurations satisfying $\KbbO_2\neq \emptyset$. For each architecture the resulting trained kernel 
    is represented by a blue dot and an orange dot corresponding respectively to its largest and smallest singular values $\sigma_{max}$ and $\sigma_{min}$. 
	The $x$-axis represents $M/CS^2$ in log scale.  (left)  All configurations; (right) All configurations for which $M\neq CS^2$. On the right final convolutions are nearly orthogonal ($\sigma_{max}=\sigma_{min}\approx 1$), but some configurations on the left (where $M=CS^2$) have $\sigma_{max}$ larger that one, and $\sigma_{min}$ close to zero.}
\label{fig:exist_and_optim}
\end{figure}

For each architecture, we plot on Figure \ref{fig:exist_and_optim} the values of $\sigma_{min}$ and $\sigma_{max}$ for the obtained $\Kcal$ and $SN\times SN=64 \times 64$. The experiment for a given architecture $(M,C,k,S)$ is represented by two points: $\sigma_{max}$, in blue, and $\sigma_{min}$, in orange.  For each point $(x,y)$ in Figure~\ref{fig:exist_and_optim}, the first coordinate $x$ corresponds to the ratio $\frac{M}{CS^2}$ of the considered architecture, and the second coordinate $y$ equals the singular value ($\sigma_{min}$ or $\sigma_{max}$) of the obtained $\Kk$. The points with $x \leq 1$ correspond to the artchitecture in the RO case ($\Kk$ is a fat matrix), and the others correspond to the architectures in the CO case ($\Kk$ is a tall matrix).
 
The right plot of Figure~\ref{fig:exist_and_optim} shows that all configurations where $M\neq CS^2$ are trained very accurately to near-perfect orthogonal convolutions. These configurations represent the vast majority of cases found in practice. However, the left plot of  Figure~\ref{fig:exist_and_optim} points out that some architectures, with $M= CS^2$, might not fully benefit of the regularization with $L_{orth}$. These architectures, corresponding to a square $\Kcal$, can mostly be found when $M=C$ and $S=1$, for instance in VGG \citep{simonyan2015deep} and Resnet \citep{he2016deep}.
We have conducted experiments that we do not report here in detail, and it seems that this is specific to the convolutional case. Fully-connected layers optimized to be orthogonal do not suffer from this phenomenon.

\subsubsection{Analysis of the \texorpdfstring{$M= CS^2$}{} Cases}\label{m=CS2_expe_sec}

Since we know that $\KbbO_2\neq \emptyset$, the explanation for the failure cases (when $\sigma_{max}$ or $\sigma_{min}$ significantly differ from $1$) is 
that the optimization was not successful.  
We tried many learning rate schemes and iteration numbers but obtained similar results\footnote{See the description of the experiments at the beginning of Section \ref{synth_exp_sec} and the related footnotes.}. 

To evaluate the proportion of successful optimizations when $M=CS^2$, we run 100 training experiments, with independent initialization, for each configuration when $M= CS^2$. In average, after convergence, we found  $\sigma_{min} \sim 1 \sim \sigma_{max} $ for $14\%$ of runs, proving that the minimizer can be reached. The explanation of this phenomenon and the evaluation of its impact on applications are open questions that we keep for future research.
A contribution of the article is to empirically identify these problematic cases.

\begin{figure}
    \centering
    \begin{minipage}{0.45\textwidth}
        \centering
        \includegraphics[width=1.0\textwidth]{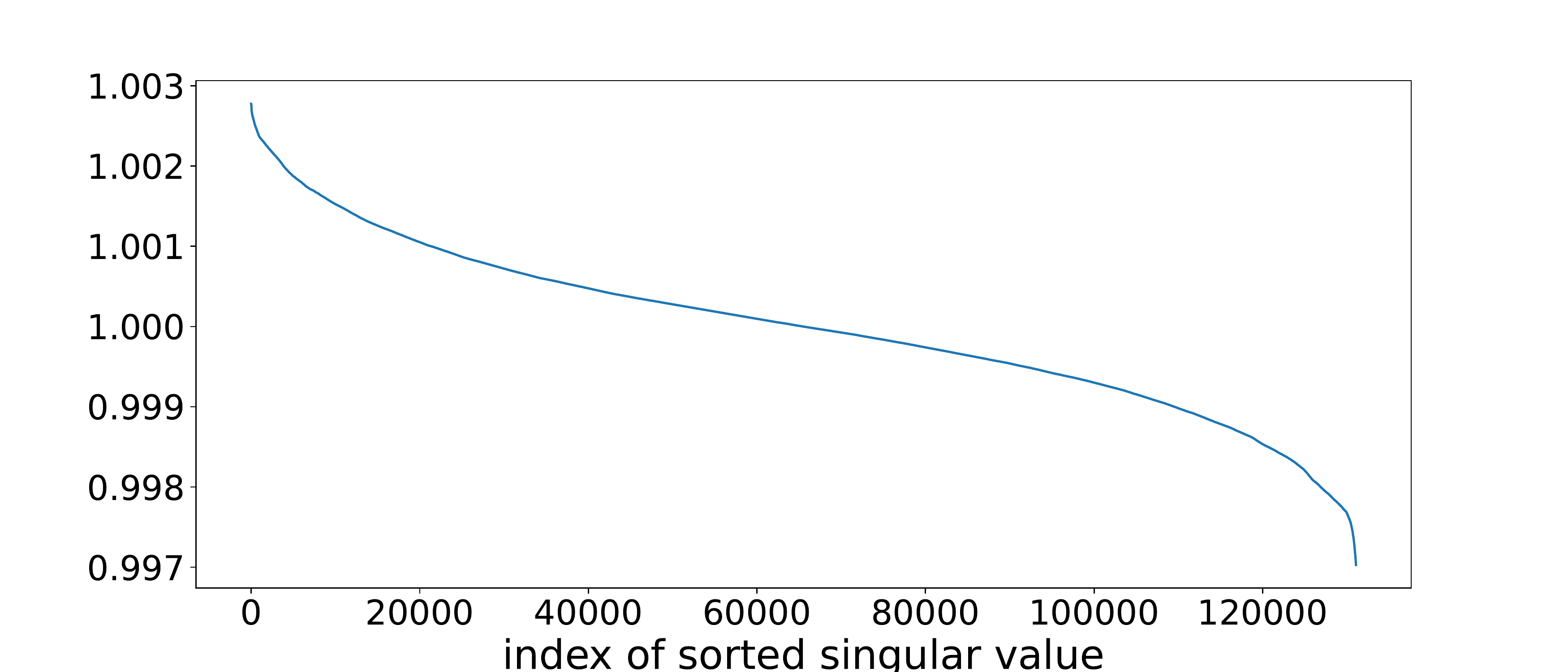} %
    \end{minipage}
    \begin{minipage}{0.45\textwidth}
        \centering
        \includegraphics[width=1.0\textwidth]{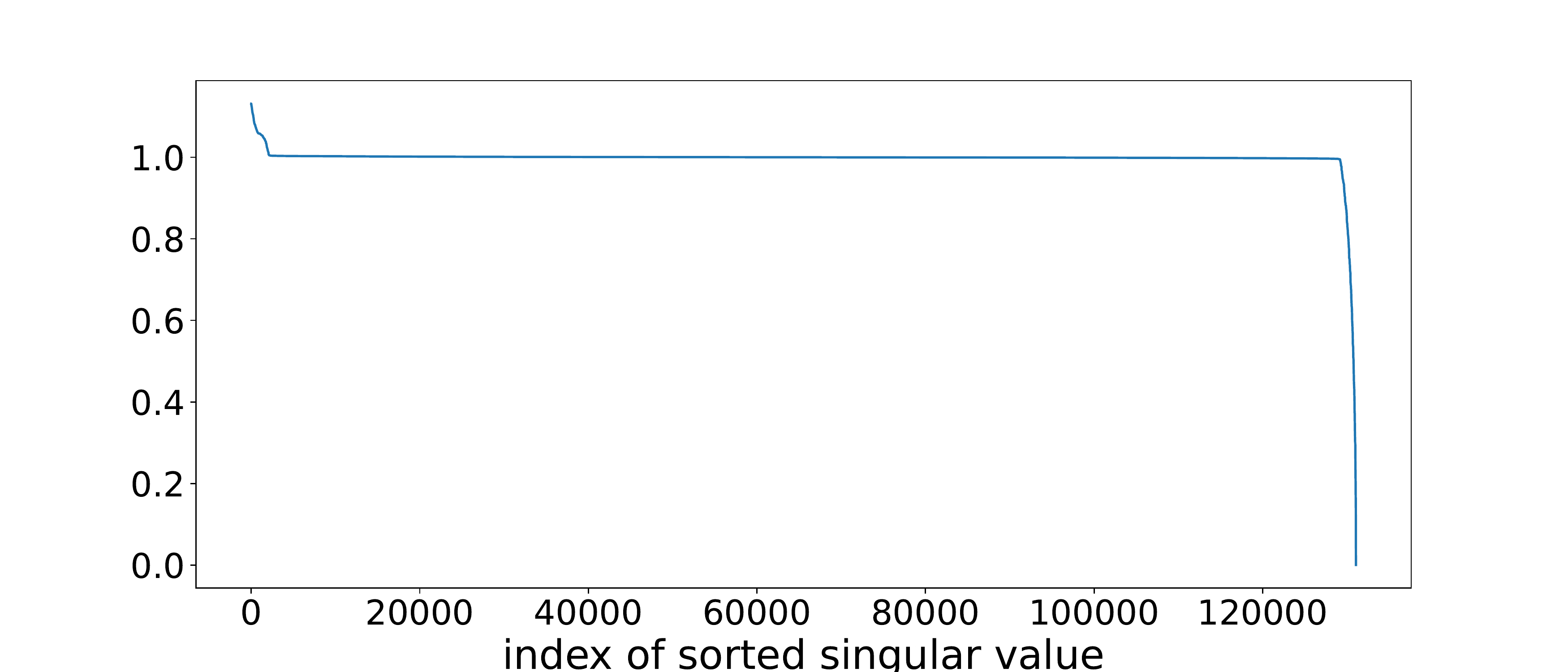} % 
    \end{minipage}\hfill
    \caption{{\bf Singular values of} $\mathbf{\Kcal}$, when $C=M$ and $S=1$. Optimization is (Left)  successful, $L_{orth}(\Kbf)  \ll 1$, (Right) Unsuccessful, $L_{orth}(\Kbf) \geq 0.1$. On the left, we see that orthogonal convolutions can be reached even when $C=M$ and $S=1$. On the right we see that, even for unsuccessful optimization, most of the singular values are very close to one. 
    }

\label{fig:optim_isometries}
\end{figure}

We display on Figure \ref{fig:optim_isometries} the singular values of $\Kcal$ defined  for $S=1$ and $N\times N=64 \times 64$ for two experiments where $M=C$. In the experiment on the left, the optimization is successful and the singular values are very accurately concentrated around $1$. On the right, we see that only a few of the singular values significantly differ from $1$.

Figure \ref{fig:optim_isometries} shows that even if $\sigma_{min}$ and $\sigma_{max}$ are not close to 1, as shown in Figure \ref{fig:exist_and_optim}, most of the singular values are close to 1. This probably explains why the landscape problem does not alter the performance on real data sets in \citet{wang2020orthogonal} and \citet{qi2020deep}. Notice that \citet{wang2020orthogonal} display a curve similar to Figure \ref{fig:optim_isometries} when used for a real data set.

\subsubsection{Stability of \texorpdfstring{$(\sigma_{min},\sigma_{max})$}{} when \texorpdfstring{$N$}{} Varies} \label{stable_expe_sec}
\begin{figure}[ht]
    \centering
    
    \begin{minipage}{0.49\textwidth}
        \centering
        \includegraphics[width=1.0\textwidth]{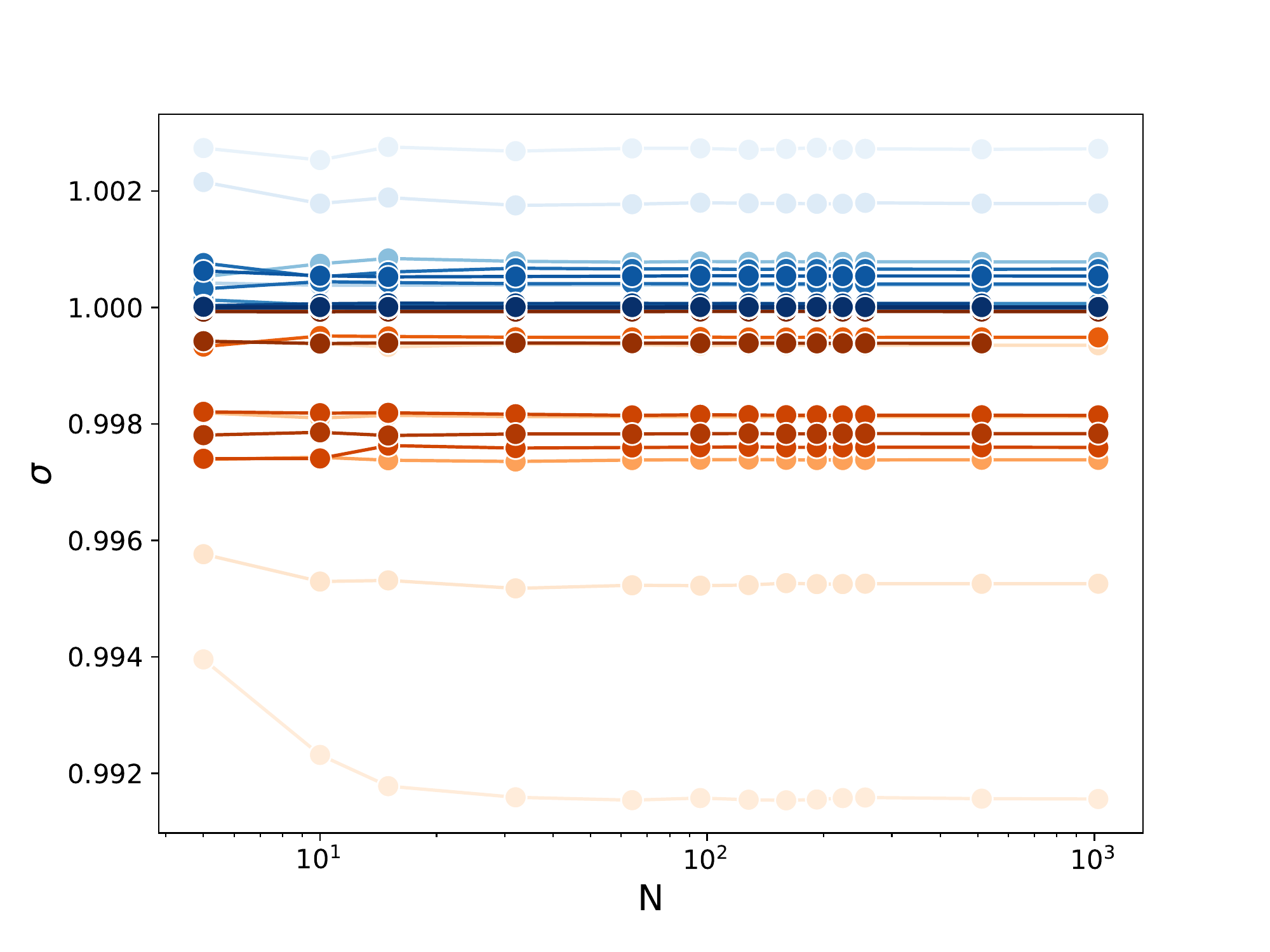} % first figure itself
        %\caption*{(a)}
    \end{minipage}\hfill
    \begin{minipage}{0.49\textwidth}
        \centering
        \includegraphics[width=1.0\textwidth]{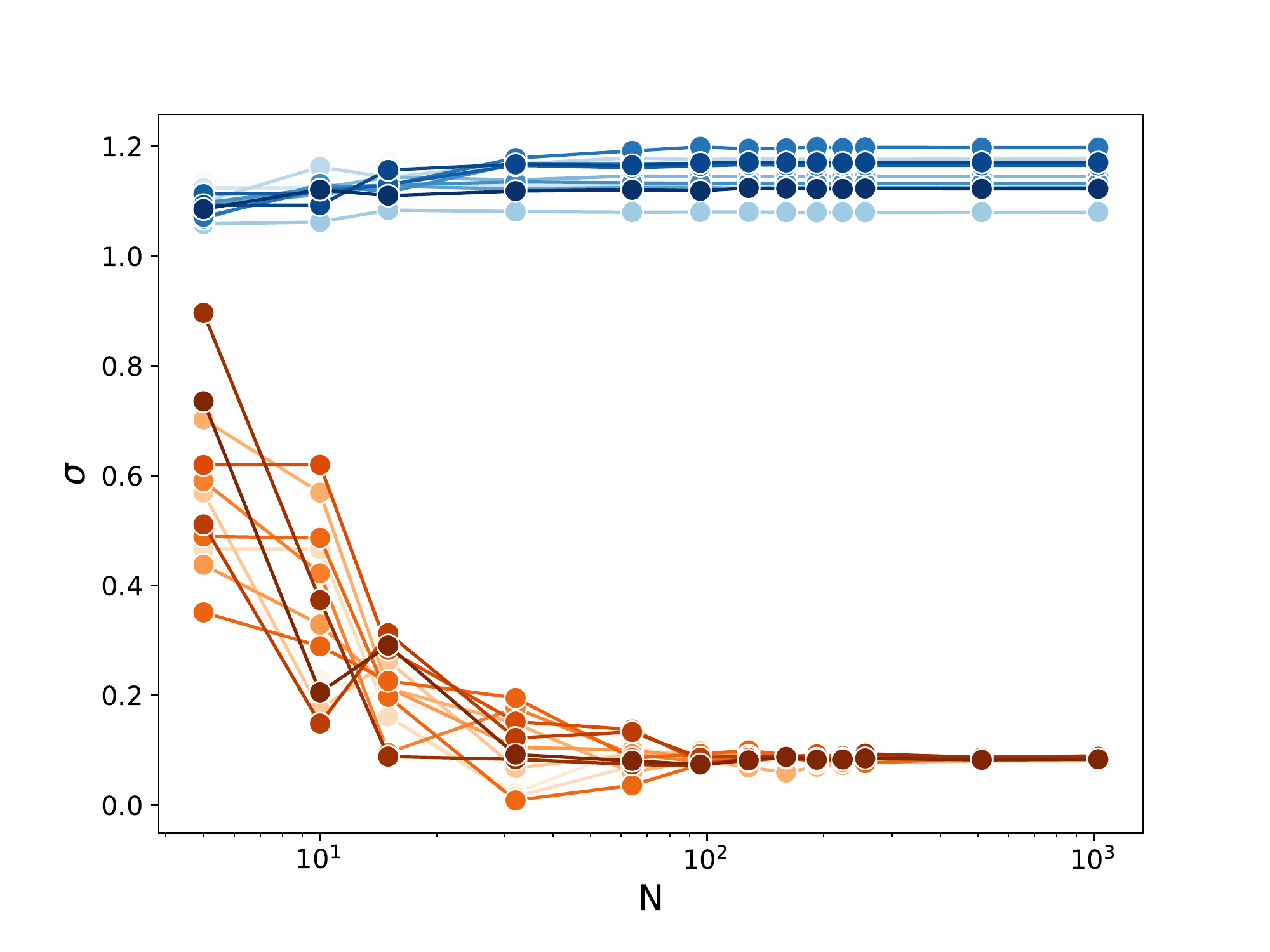} % second figure itself
        %\caption*{(b)}
    \end{minipage}
    \caption{{\bf Evolution of $\mathbf{\sigma_{min}}$ and $\mathbf{\sigma_{max}}$ according to input image size} (x-axis: $N$ in log-scale). Each line (transparency) represents the singular values $\sigma_{max}$ (in blue) and $\sigma_{min}$ (in orange) of $\Kcal$ for different $N$ and a fixed $\Kbf$. (Left) $\Kbf$ is such that $L_{orth}(\Kbf)\ll 1$, singular values remain close to one whatever $N$.  (Right) $\Kbf$ is such that $L_{orth}(\Kbf)\not\ll 1$, the largest (resp. smallest) singular values  increase (resp. decrease) when N grows.}
\label{fig:out_stats_increasingN}
\end{figure}

In this experiment, we evaluate how the singular values $\sigma_{min}$ and $\sigma_{max}$ of $\Kcal$ vary when the parameter $N$ defining the size $SN\times SN$ of the  input channels varies, for $\Kbf$ fixed. This is important for applications \citep{Shelhamer2017FCN,ji2019invariant,ren2015faster} using fully convolutional networks, or for transfer learning using pre-learnt convolutional feature extractor.

To do so, we randomly select $50$ experiments for which the optimization was successful ($L_{orth}(\Kbf) \leq 0.001$) and $50$ experiments for which it was unsuccessful ($L_{orth}(\Kbf) \geq 0.02$). They are respectively used to construct the figures on the left and the right side of  Figure \ref{fig:out_stats_increasingN}. For a given $\Kbf$, we display  the singular values $\sigma_{min}$ and $\sigma_{ max}$ of $\Kcal$ for $N\in \{5,12,15,32,64,128,256,512,1024\}$, as orange and blue dots. The dots corresponding to the same $\Kbf$ are linked by a line. 

We see, on the left of  Figure \ref{fig:out_stats_increasingN}, that for successful experiments ($L_{orth} (\Kbf) \ll 1$), the singular values are very stable when $N$ varies. This corresponds to the behavior described in Theorem \ref{Prop norme spectrale} and Proposition \ref{lemme lipchitz norme spectrale}. We also point out, on the right of  Figure \ref{fig:out_stats_increasingN}, that for  unsuccessful optimization ($L_{orth} \not\ll 1$), $\sigma_{min}$ (resp. $\sigma_{max}$) values decrease (resp. increase) rapidly when N increases.
 
\subsection{Data sets Experiments}\label{sec:dataset_exp}
In this section we compare, on Cifar10 and Imagenette data sets, the performance, robustness, spectral properties and processing time of three networks: standard convolutional neural networks with unconstrained convolutions called {\it Conv2D}, the same network architectures with convolutions regularized with $L_{orth}$ and the same network architectures with convolutions constrained with a method that we call {\it Cayley}, a hard convolutional layer orthogonality method\footnote{Our experiments complement the comparison of $L_{orth}$ with kernel orthogonality methods in \citet{wang2020orthogonal}.} based on the Cayley transform \citep{trockman2021orthogonalizing}. The latter builds convolutions parameterized by $k\times k$ parameters but, because a mapping is applied to obtain orthogonality, the convolution kernels are of size $N\times N$. In comparison,  $L_{orth}$ regularization provides convolutions kernels of size $k\times k$, as is standard. The methods are therefore not expected to provide the same results which makes the comparison a bit complicated. This comparison is also somewhat unfair since the regularization with $L_{orth}$ enjoys a parameter $\lambda$. We show results for a wide range of $\lambda$ but assume, when interpreting the results, that an optimal $\lambda$ is chosen, for instance using cross-validation.

The design of the experiments aims at simultaneously obtaining good accuracy and robustness. Therefore, for the purpose of robustness, we only use isometric activations and nearly orthogonal convolutional layers. We cannot expect, with this robustness constraint, to obtain clean accuracies as good as those reported in \citet{wang2020orthogonal}.

On Cifar10, we use a VGG-like architecture~\citep{simonyan2015deep} with nine convolutional layers and a single dense output layer with ten $1$-Lipschitz neurons. In all experiments, for a fair comparison, we use invertible downsampling emulation as in~\citet{trockman2021orthogonalizing}. In order to avoid problematic configurations described in Section~\ref{m=CS2_expe_sec}, we alternate channels numbers $C,C+2,C$ within each VGG block (see Table~~\ref{tab:NN_archi_cifar10}). 

The network is trained during 400 epochs with a batch size of 128, using cross-entropy loss with temperature, Adam optimizer~\citep{Diederik14Adam} with a decreasing learning rate, and standard data augmentation. A full description of hyperparameters is given in Appendix~\ref{Cifar_appendix}. The optimized network achieves $91\%$ accuracy on the Cifar10 test set, for the $Conv2D$ classical network.

As already mentioned, three configurations are compared:  $Conv2D$: classical convolutions (i.e. no regularization), $Cayley$: convolutions constrained by Cayley method \citep{trockman2021orthogonalizing}, and $L_{orth}$: the regularization with $L_{orth}$. For the $L_{orth}$ regularization, we investigate the properties of the solution obtained when minimizing \eqref{model_regularise} for $\lambda \in \{ 10, 1, 10^{-1}, 10^{-2}, 10^{-3}, 10^{-4}, 10^{-5}\}$.

 After training, $\sigma_{max}$ and $\sigma_{min}$ values are computed for each convolutional layer using the method described in Appendix \ref{sec:poweriter}. Each configuration is learnt 10 times to provide mean and standard deviation for the following metrics:
\begin{itemize}
    \item \textit{Acc. clean}: Classical accuracy on a clean test set
    \item $\Sigma _{max}= \max_{l} (\sigma _{max}(\Kcal _l))$:  the largest singular value among all the convolutional layers's singular values.
    \item $\Sigma _{min}= \min_{l} (\sigma _{min}(\Kcal _l))$: the smallest  singular value among all the convolutional layers's singular values.
    \item $E_{lip}$
    :  Empirical local Lipschitz constants of the network computed using the PGD-like method proposed by~\citet{yang2020_closer}.
    \item $E_{rob}$
    : The empirical robustness accuracy, i.e. the proportion of test samples on which a vanilla Projected Gradient Descent (PGD) attack~\citep{madry2018towards} failed (for a robustness radius $\epsilon=36/255$). PGD attack is applied with 10 iterations and a  factor $\alpha=\epsilon/4.0$ .
    \item $T_{epoch}$: the average epoch processing time.
\end{itemize}

\begin{figure}[ht]
    \centering
    
    \begin{minipage}{0.49\textwidth}
        \centering
        \includegraphics[width=1.0\textwidth]{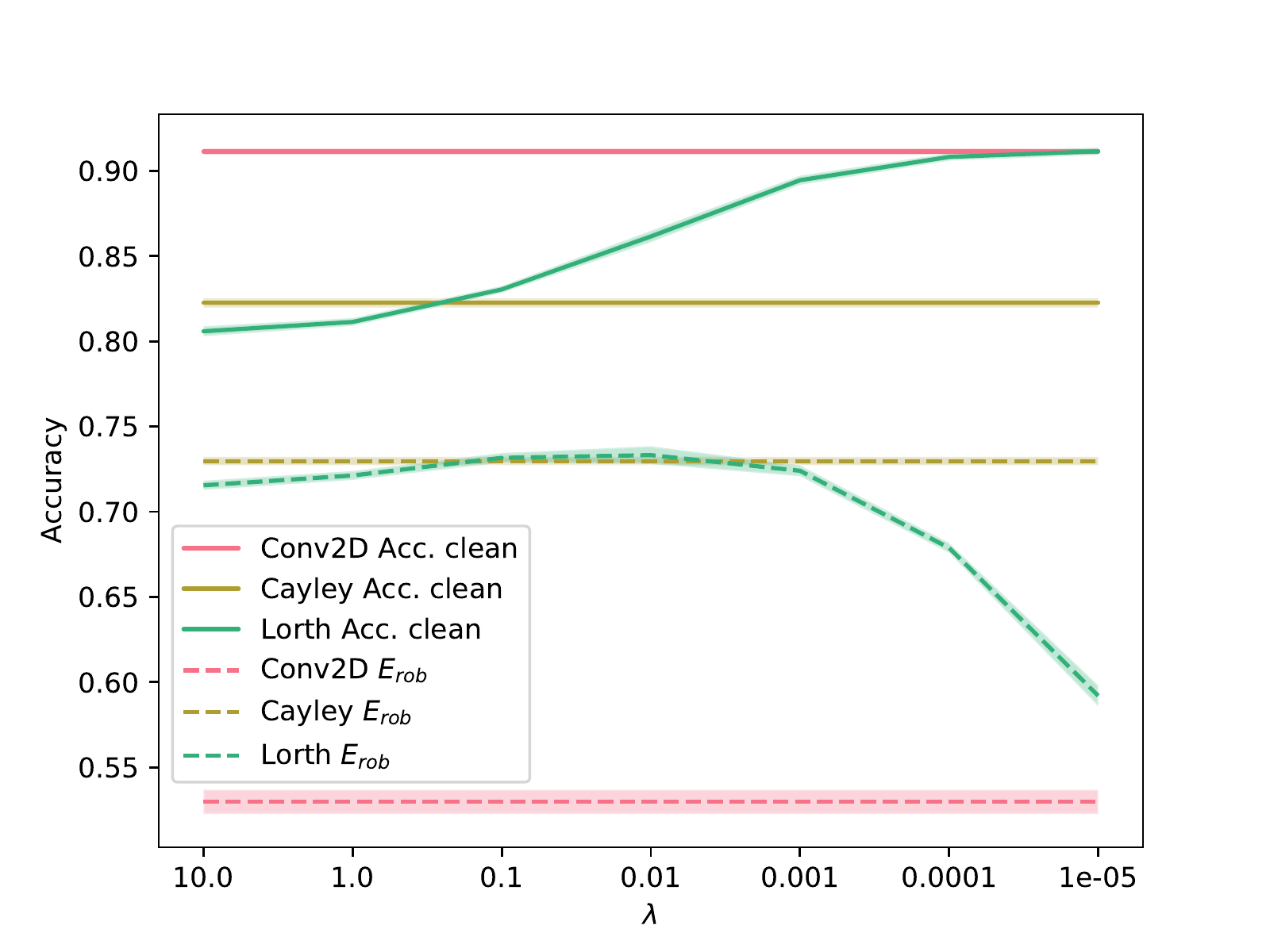} 
    \end{minipage}\hfill
    \begin{minipage}{0.49\textwidth}
        \centering
        \includegraphics[width=1.0\textwidth]{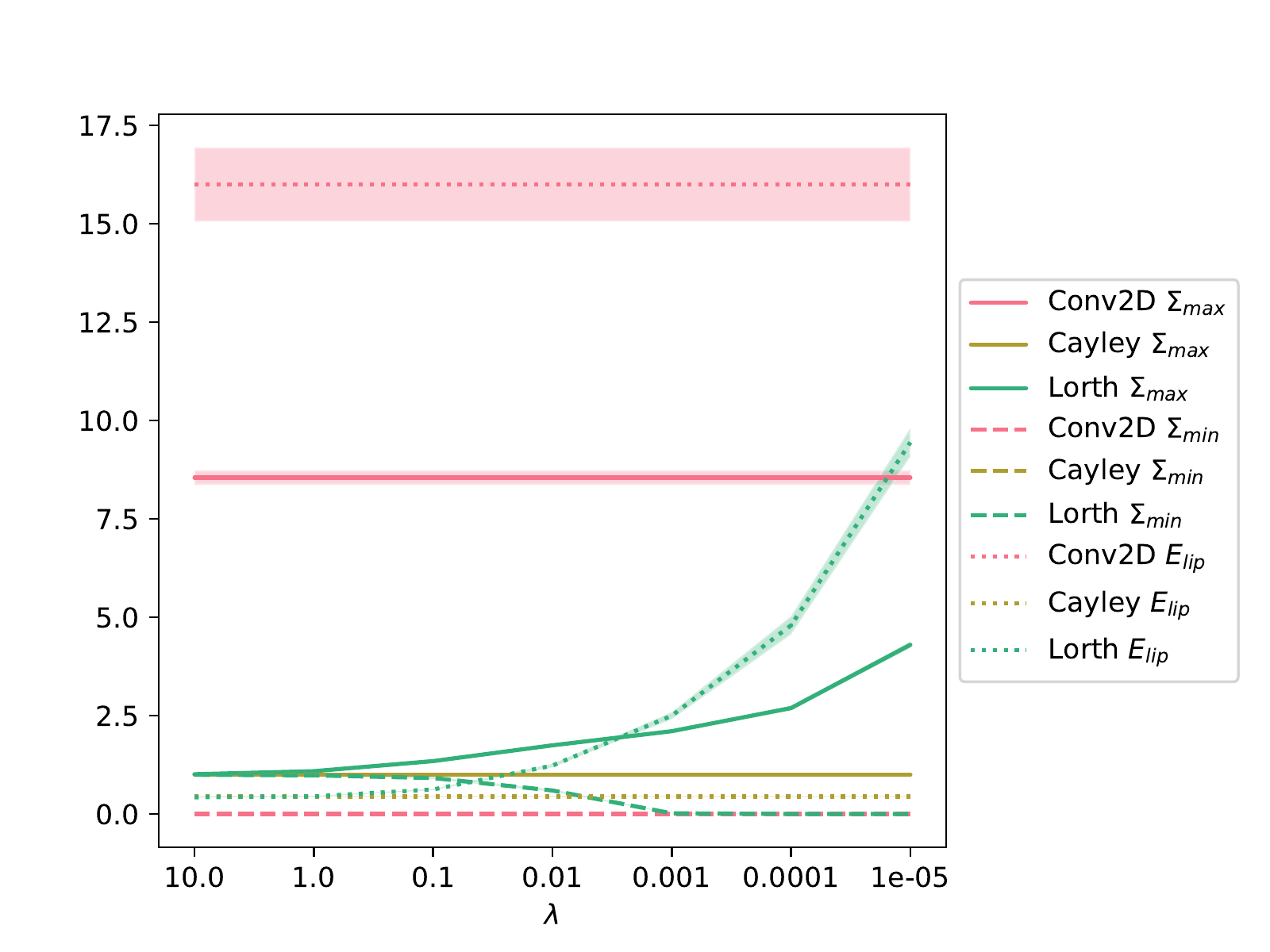} 
    \end{minipage}
    \caption{{\bf Cifar10}: Mean evolution of metrics (and error band in shadow) according to $\lambda$ parameter for $L_{orth}$ method, and comparison with $Conv2D$ and $Cayley$ configurations (constant values): (Left) Clean accuracy and empirical robustness for $\epsilon=36/255$, (Right) $\Sigma _{max}$, $\Sigma _{min}$ and $E_{lip}$ metrics.  The parameter $\lambda$ permits tune a tradeoff between accuracy and orthogonality, for the benefit of a better accuracy and a better robustness. }
\label{fig:cifar10_expe_metrics}
\end{figure}

Figure~\ref{fig:cifar10_expe_metrics} shows that the  regularization parameter $\lambda$, in \eqref{model_regularise}, provides a way to tune a tradeoff between robustness ($E_{rob}$) and clean accuracy ($Acc. clean$), by controlling the singular values of the layers ($\Sigma _{max}$ and $\Sigma _{min}$). On the contrary $Cayley$ or $Conv2D$ each provide a single tradeoff (shown with constant value in figures).  
The configurations $\lambda=10^{-1} \text{ and } 10^{-2}$  achieve better clean accuracy and similar empirical robustness performances as the $Cayley$ method.  Furthermore, their empirical Lipschitz constants are very close to one. Finally, error bands for $Cayley$ and $L_{orth}$ methods are very narrow. 

Processing time $T_{epoch}$ for regularizing with $L_{orth}$ is only $5\%$ slower than the reference network $Conv2D$, but $2.2$ times faster than the one for the $Cayley$ method. It is not reported here in detail but the convergence speeds, in number of epochs, are similar. 
Moreover, $L_{orth}$ provides classical convolution at inference. On the contrary, the $Cayley$ method provides orthogonal convolutions of size $N\times N$ obtained using a mapping that involves Fourier transforms, which leads to higher computational complexity even at inference. The change of support can also explain the slight difference in $Acc. clean$ between the $Cayley$ method and the strong regularization $\lambda= 10$ for $L_{orth}$ method.

\begin{figure}[ht]
    \centering
    
    \begin{minipage}{0.49\textwidth}
        \centering
        \includegraphics[width=1.0\textwidth]{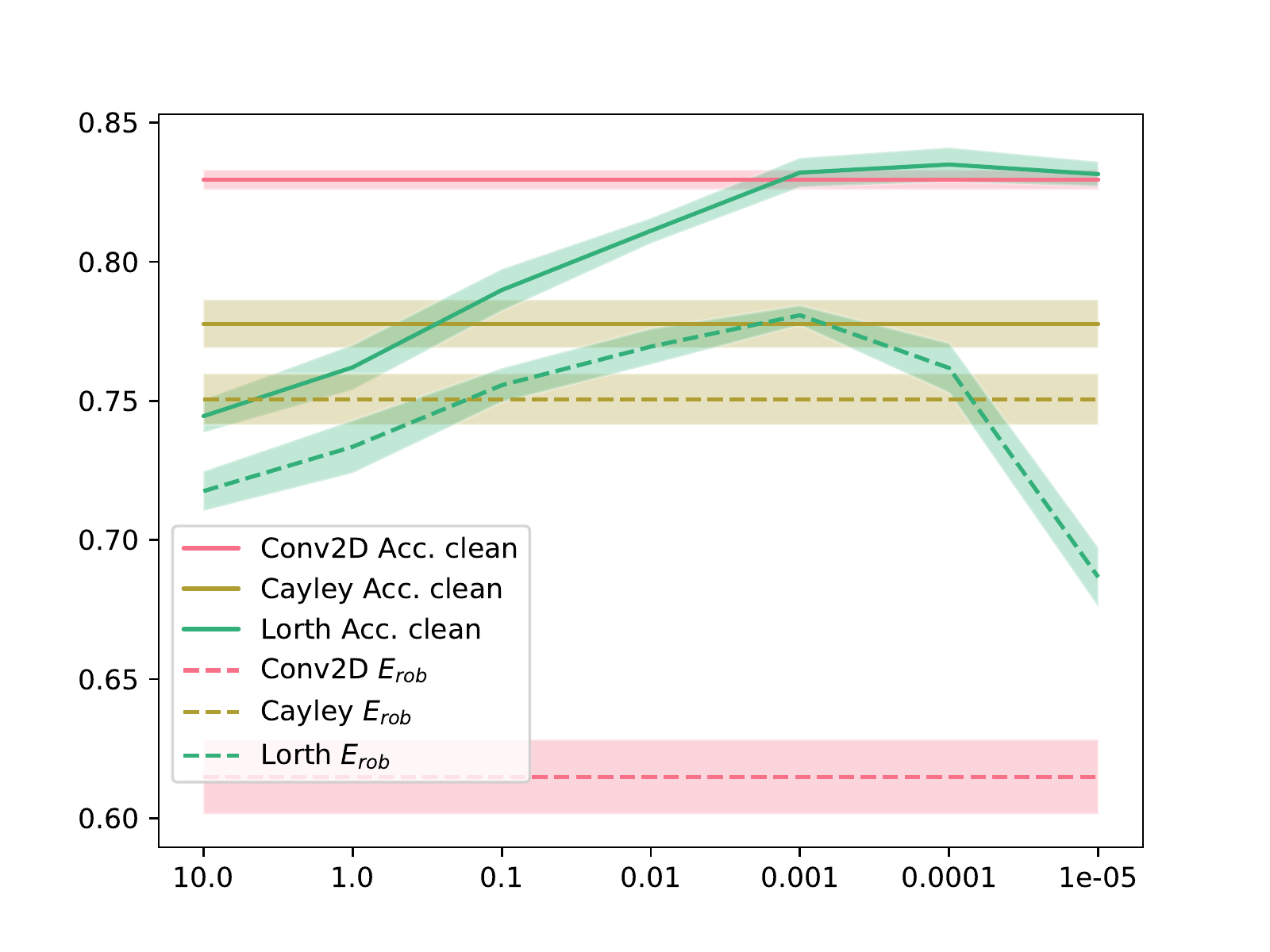} % 
    \end{minipage}\hfill
    \begin{minipage}{0.49\textwidth}
        \centering
        \includegraphics[width=1.0\textwidth]{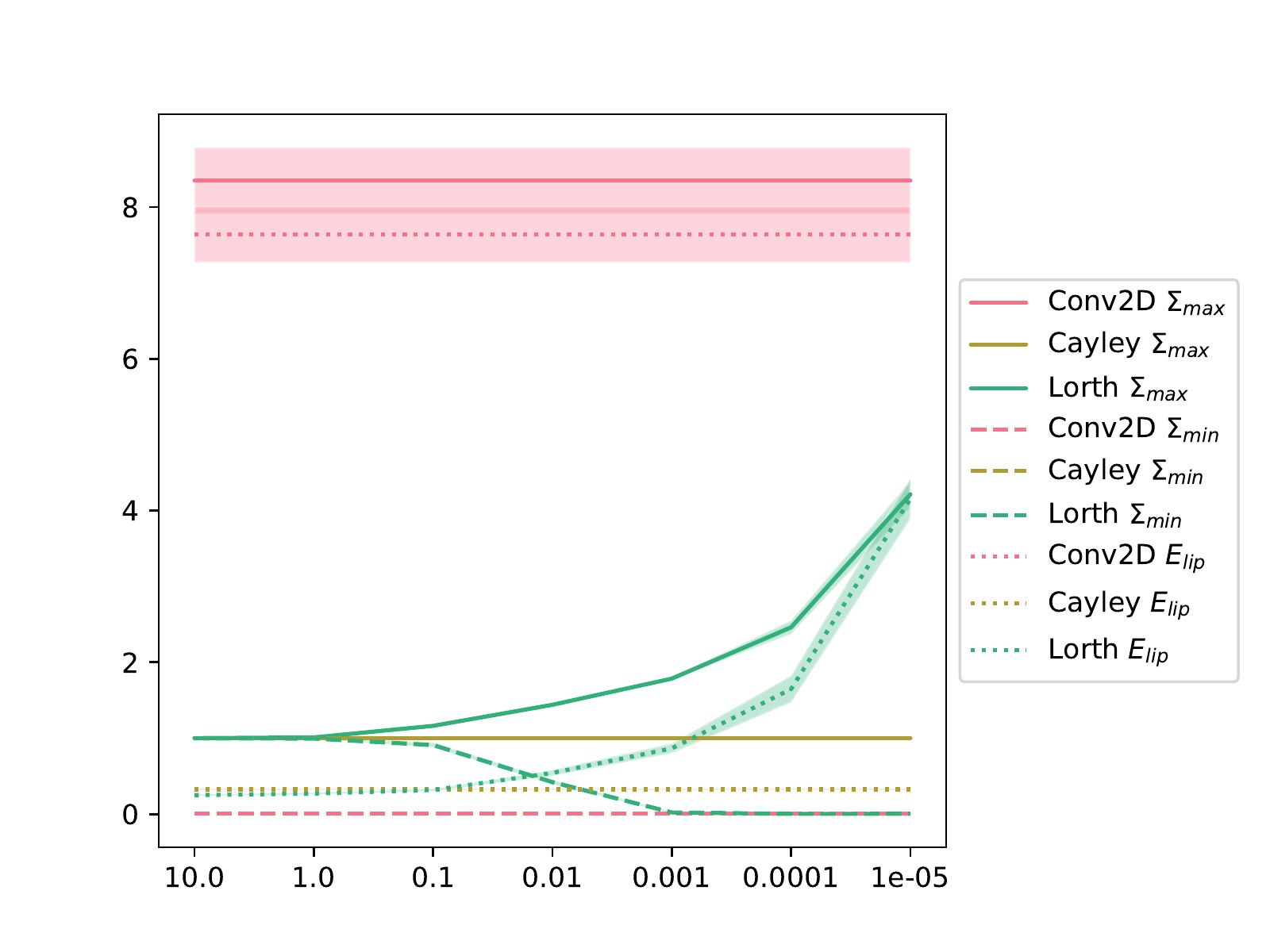} 
    \end{minipage}
    \caption{{\bf Imagenette}: Mean evolution of metrics (and error band in shadow) according to $\lambda$ parameter for $L_{orth}$ method, and comparison with $Conv2D$ and $Cayley$ configurations (constant values): (Left) Clean accuracy and empirical robustness for $\epsilon=36/255$, (Right) $\Sigma _{max}$, $\Sigma _{min}$ and $E_{lip}$ metrics. The parameter $\lambda$ permits to tune a tradeoff between accuracy and orthogonality, for the benefit of a better accuracy and a better robustness. }
\label{fig:imagenette_expe_metrics}
\end{figure}

Figure~\ref{fig:imagenette_expe_metrics} presents 
the same experiments on the Imagenette data set ~\citep{imagenetteDataset}. The latter is a 10-class subset of Imagenet data set~\citep{imagenet_cvpr09} with $160\times 160$ images. The architecture is also a VGG-like one but with 15 convolutional layers.
We trained $10$ times during $400$ epochs with a batch size of $64$, using the same loss and optimizer as for the Cifar10 experiments. The architecture and hyperparameters are described in Section~\ref{Imagenette_appendix}. The average performance of the unconstrained $Conv2D$ networks is about $83\%$.

Interestingly, even if the image size is larger on the Imagenette data set, $L_{orth}$ regularization shows the same profile as for Cifar10, when $\lambda$ decreases, ranging from strong orthogonality to high clean accuracy (equivalent to Conv2D configuration), with the best compromise for $\lambda=10^{-3}$. Notice that for $\lambda$ decreasing from $10$ to $0.01$ the loss of orthogonality permits to obtain a significantly better accuracy but does not significantly favor attacks. Altogether, this leads to an increase in the empirical robustness accuracy. The phenomenon is present but less visible on Cifar10.

Besides, because $L_{orth}$ does not depend on the size parameter $N$ of the input channels, the processing time for the $L_{orth}$ regularization is only $1.1$ times slower than for the non-constrained convolution $Conv2D$. In comparison, the $Cayley$ method is $6.5$ slower than $Conv2D$. 

\section{Conclusion}

This paper provides a necessary and sufficient condition on the architecture for the existence of an orthogonal convolutional layer with circular padding.
The conditions prove that orthogonal convolutional layers exist for most relevant architectures.
We show that the situation is less favorable with `valid' and `same' zero-paddings.
We also prove that the minimization of the surrogate $L_{orth}$ enables constructing orthogonal convolutional layers in a stable manner, that also scales well with the input size parameter $N$.
The experiments confirm that this is practically the case for most of the configurations, except when $M=CS^2$ for which interrogations remain.

Altogether, the study guarantees that the regularization with $L_{orth}$ is an efficient, stable numerical strategy to learn orthogonal convolutional layers.
It can safely be used even when the signal/image size is very large.
The regularization parameter $\lambda$ is chosen depending on the tradeoff we want between accuracy and orthogonality, for the benefit of both accuracy and robustness.

Let us mention three open questions related to this article. First, a better understanding of the landscape problem as well as solutions to this problem when $M=CS^2$ could be useful. Also, as initiated in ~\citet{Kim21LipTransformer,Fei2022OVIT}, the extension of Lipschitz and orthogonal constraints and regularization to the attention-based networks is a natural and relevant open question. Finally, a clean adaptation of the regularization with $L_{orth}$ for the 'valid' and 'same' boundary conditions is needed. 
As shown in Section \ref{other-padding-sec}, approximate orthogonality seems to be key with these boundary conditions.

%\if@submission \@submission false
%\section*{Acknowledgements}
\acks{
Our work has benefited from the AI Interdisciplinary Institute ANITI. ANITI is funded by the French ”Investing for
the Future – PIA3” program under the Grant agreement n°ANR-19-PI3A-0004. 
The authors gratefully acknowledge the support of the DEEL project.\footnote{https://www.deel.ai/}
}

\fi %keepmainpart
%%%%%%%%%%%%%%%%%%%%%%%%%%%%%%%%%%%%%%%%%%%%%%%%%%%%%%%%%%%%

\ifkeepsupplementary

%\newpage

\appendix

\section{Notation and Definitions}
In this section we specify some notation and definitions.

\subsection{Notation} \label{notation_sec}
We summarize the notations specific to our problem in Table \ref{Notation_table}. We then describe mathematical notations, their adaptation to our context and notations for the canonical bases of matrix spaces that appear in the proofs.
\begin{table}[ht]
\begin{center}
\begin{tabular}{l l l}\hline
    notation & domain/type & description \\\hline
    $M$ &$\NN$ & number of output channels \\\hline
    $C$ &$\NN$ & number of input channels \\\hline
    $k$ & $\NN$, odd & the convolution kernel is of support $k$ for signals, $k\times k$ for images\\\hline
    $d$ & \{1,2\} & $1$ when the layer applies to signals; $2$ for images \\\hline
    $S$ & $\NN$ & stride/sampling parameter \\ \hline
    $N$ & $\NN$ & input channels are of size $SN$ for signals, $SN\times SN$ for images \\
    & & output channels are of size $N$ for signals, $N\times N$ for images \\\hline
     ${\mathbb K}_d$ & vector space & equal to $\KbfS$ for images or $\RR^{M\times C \times k}$ for signals\\\hline
    $\Kbf$ & ${\mathbb K}_d$  & kernel tensor that contains all weights defining the layer \\\hline
    $\Kcal$ & $\KkS$  & the matrix that applies the convolutional layer defined by $\Kbf$ \\
    & or $\RR^{ MN \times CSN}$ & to inputs of size defined by $N$\\\hline
    $\KbbO_d$ & subset of $\Kbb_d$ & kernel tensors $\Kbf$ such that $\Kcal$ is orthogonal\\\hline
     $\Kbf_{i,j}$ & $\RR^{k\times k}$ or $\RR^k$  & weights of the convolution from input channel $j$ to output channel $i$\\\hline
    $\mathcal{M}(\Kbf_{i,j})$ & $\RR^{N^2 \times S^2N^2}$  & matrix that applies the strided convolution defined by $\Kbf_{i,j}$ \\
     &or $\RR^{N \times SN}$  & to inputs of size defined by $N$ \\\hline
    $L_{orth}(\Kbf)$ & $\RR_+$ & regularization applied to $\Kbf$ and enforcing orthogonality of $\Kcal$, \\
    & & see Definition \ref{Lorth-def}\\\hline
    $\ERR{F}{N}(\Kbf) $& $\RR_+$ & measures, in Frobenius norm, how matrix $\Kcal$ for the signal size $N$\\
    & &deviates from being orthogonal,  \\\hline
    $\ERR{s}{N}(\Kbf) $& $\RR_+$ & same as above for the spectral norm \\\hline
   $\Iro$ & tensor & tensor appearing in the definition of $L_{orth}$ 
\\\hline\end{tabular}
\end{center}
\caption{Summary of the main notations.}\label{Notation_table}
\end{table}

%\subsection{Standard math definition}
The floor of a real number will be denoted by $\lfloor . \rfloor$.
For two integers $a$ and $b$, $\llbracket a,b \rrbracket$ denotes the set of integers $n$ such that $a \leq n \leq b$.
We also denote by $a\%b$ the rest of the euclidean division of $a$ by $b$, and
$\llbracket a,b \rrbracket \%n = \{x\%n | x \in \llbracket a,b \rrbracket \} $.
We denote by $\delta_{i=j}$, the Kronecker symbol, which is equal to $1$ if $i=j$, and $0$ if $i \neq j$.

We denote by $0_s$ the null vector of $\RR^s$.
For a matrix $A \in \RR^{m \times n}$, $\sigma_{max}(A)$ denotes the largest singular value of $A$ and ${\|A\|}_2= \sigma_{max}(A)$ is its spectral norm.
We also have $\|A\|_1 = \max_{0 \leq j \leq n-1 } \sum_{i=0}^{m-1} |A_{i,j}|$ and $\|A\|_{\infty} = \max_{0 \leq i \leq m-1 } \sum_{j=0}^{n-1} |A_{i,j}|$.
We denote by $\Id_n \in \RR^{n \times n}$ the identity matrix of size $n$. We use 'Matlab colon notation' as index of matrices and tensors. For instance, $A_{i,:}$ is the $i^{\mbox{th}}$ line of $A$.

Recall that $\|.\|_F$ denotes the norm which, to any tensor of order larger than or equal to 2, associates the square root of the sum of the squares of all its elements (e.g., for a matrix it corresponds to the Frobenius norm).

Recall that $S$ is the stride parameter, $k=2r+1$ is the size of the 1D kernels.
$SN$ is the size of the input channels and $N$ is the size of the output channels. \\
For a vector space $\mathcal{E}$, we denote by $\mathcal{B}(\mathcal{E})$ its canonical basis.
We set 
\begin{align}\label{def bases canoniques}
 \begin{cases}
     {(e_i)}_{i=0..k-1} = \Bb(\RR^k) \\ {(f_i)}_{i=0..SN-1} = \Bb(\RR^{SN}) \\
{(E_{a,b})}_{a=0..N-1,b=0..SN-1} = \Bb(\RR^{N \times SN}) \\
{(\overline{E}_{a,b})}_{a=0..SN-1,b=0..N-1} = \Bb(\RR^{SN \times N}) \\
{(F_{a,b})}_{a=0..SN-1,b=0..SN-1} = \Bb(\RR^{SN \times SN}) \\
{(G_{a,b})}_{a=0..N-1,b=0..N-1} = \Bb(\RR^{N \times N}) \;.
\end{cases}
\end{align}
Note that the indices start at 0, thus we have for example $e_0=\begin{bmatrix}
 1 \\ 0_{k-1}
\end{bmatrix}$, $e_{k-1} = \begin{bmatrix}
0_{k-1} \\ 1 
\end{bmatrix}$, and for all $i \in \llbracket 1,k-2 \rrbracket$, $e_i = \begin{bmatrix}
0_i \\ 1 \\ 0_{k-i-1}
\end{bmatrix}$.\\
To simplify the calculations, the definitions are extended for $ a, b $ outside the usual intervals, it is done by periodization.
Hence, for all $a,b \in \mathbb{Z}$, denoting by $\hat{a} = a\%SN $, $\Tilde{a} = a\%N$, and similarly $\hat{b} = b\%SN $, $\Tilde{b} = b\%N$,
we set 
\begin{align}\label{def bases pour Z}
    \begin{cases}
         e_a = e_{a\%k}, \qquad 
    f_a = f_{\hat{a}}  \\ 
    E_{a,b} = E_{\Tilde{a},\hat{b}} , \qquad
    \overline{E}_{a,b} =  \overline{E}_{\hat{a},\Tilde{b}} , \qquad 
    F_{a,b} = F_{\hat{a},\hat{b}} , \qquad
    G_{a,b} = G_{\Tilde{a},\Tilde{b}}.
    \end{cases}
\end{align}
Therefore, for all $a,b,c,d \in \mathbb{Z}$, we have
\begin{align}\label{produit matrice elementaires}
    \begin{cases}
         E_{a,b}F_{c,d} = \delta_{\hat{b}=\hat{c}} E_{a,d} , \qquad 
    E_{a,b}\overline{E}_{c,d} = \delta_{\hat{b}=\hat{c}} G_{a,d} \\
    \overline{E}_{a,b}E_{c,d} = \delta_{\Tilde{b}=\Tilde{c}} F_{a,d} , \qquad
    F_{a,b}\overline{E}_{c,d} = \delta_{\hat{b}=\hat{c}} \overline{E}_{a,d}.
    \end{cases}
\end{align}
Note also that
\begin{align}\label{E ab^T = bar E b,a}
    E_{a,b}^T = \overline{E}_{b,a} \;.
\end{align}

\subsection{Corresponding 1D Definitions}\label{notation1D-sec}
In this section, we give the definitions for signals (1D case), of the objects defined in the introduction for images (2D case).

\subsubsection{Orthogonality}\label{sec orthogonality appendix}
As in Section \ref{sec orthogonality main part}, we denote by $\noyau \in \RR^{M \times C \times k}$ the kernel tensor and $\Kk \in \RR^{MN \times CSN}$ the matrix that applies the convolutional layer of architecture $(M,C,k,S)$ to $C$ vectorized channels of size $SN$. Note that, in the 1D case, we need to compare $M$ with $CS$ instead of $CS^2$.\\
\textbf{RO case:} When $M \leq CS$, $\Kk$ is orthogonal if and only if $ \Kk \Kk^T = Id_{MN} \;. $ \\
\textbf{CO case:} When $M \geq CS$, $\Kk$ is orthogonal if and only if $ \Kk^T \Kk = Id_{CSN} \;. $

\subsubsection{The Function \texorpdfstring{$L_{orth}$}{}}\label{sec 1D L_orth}

We define $L_{orth}$ similarly to the 2D case (see Section \ref{sec:lorth} and Figure \ref{fig:con_h_g}).
Formally, for $P \in \mathbb{N}$, and $h,g \in \RR^k$,
we define 
\begin{align}\label{conv in R^2P+1}
    \conv(h,g,\text{padding zero} = P, \text{stride} = 1) \in \RR^{2P+1}
\end{align}
such that for all $ i \in \llbracket0,2P \rrbracket$,
\begin{align}\label{conv h g P 1 formule}
    [\conv(h,g,\text{padding zero} = P, \text{stride} = 1)]_{i} =  \sum_{i'=0}^{k-1} h_{i'} \Bar{g}_{i'+i} \;,
\end{align}
where $\Bar{g}$ is defined for $i \in \llbracket 0,2P+k-1 \rrbracket$ as follows
\begin{align}\label{def g bar}
    \bar{g}_{i} = \left\{\begin{array}{ll}
    g_{i-P} & \mbox{if } i \in \llbracket P,P+k-1 \rrbracket, \\
    0 & \mbox{otherwise}.
    \end{array}\right.
\end{align}
Note that, for $P' \leq P$, we have, for all $i \in \llbracket0,2P' \rrbracket$,
\begin{multline}\label{lien conv P et P'}
[\conv(h,g,\text{padding zero} = P', \text{stride} = 1)]_{i} \\
= [\conv(h,g,\text{padding zero} = P, \text{stride} = 1)]_{i+P-P'} .
\end{multline}

The strided version will be denoted by $\conv(h,g,\text{padding zero} = P, \text{stride} = S) \in \RR^{\lfloor 2 P/S \rfloor + 1}$ and is defined as follows:
For all $i \in  \llbracket 0,\lfloor 2 P/S \rfloor \rrbracket$
\begin{align}\label{def conv S}
    {\left[\conv(h,g,\text{padding zero} = P, \text{stride} = S)\right]}_{i} =  {\left[\conv(h,g,\text{padding zero} = P, \text{stride} = 1)\right]}_{Si}.
\end{align}
Finally, reminding that for all $m\in\llbracket1,M \rrbracket$ and $c\in\llbracket1,C \rrbracket$, $\Kbf_{m,c}\in\RR^k$, we denote by $$\CONV(\noyau,\noyau,\text{padding zero} = P, \text{stride} = S) \in \mathbb{R}^{M \times M \times (\lfloor 2P/S \rfloor +1)}$$ the third-order tensor such that, for all $m,l \in \llbracket1,M \rrbracket$,
\begin{multline}\label{def CONV}
{\CONV(\noyau,\noyau,\text{padding zero} = P, \text{stride} = S)}_{m,l,:} \\
= \sum_{c=1}^C \conv(\noyau_{m,c},\noyau_{l,c},\text{padding zero} = P, \text{stride} = S) .
\end{multline}
From now on, we take $P = \left\lfloor \frac{k-1}{S} \right\rfloor S$ and $\Iro \in \mathbb{R}^{M \times M \times (2 P/S +1)}$  the tensor whose entries are all zero except its central  $M \times M$ entry which is equal to an
identity matrix: $[\Iro]_{:,:, P/S } = Id_{M} $.
Put differently, we have for all $m,l \in \llbracket 1,M \rrbracket$,
\begin{align}\label{I_r0 bis}
    {\left[\Iro\right]}_{m,l,:} = \delta_{m=l} \begin{bmatrix}
    0_{P/S} \\ 1 \\ 0_{P/S} 
    \end{bmatrix} \;.
\end{align}
And $L_{orth}$ for 1D convolutions is defined as follows:
\begin{itemize}
\item In the RO case: 
\[L_{orth}(\noyau) = {\| \CONV(\noyau,\noyau,\text{padding zero} = P, \text{stride} = S) - \Iro \|}_F^2~.
\]
\item In the CO case: 
\[L_{orth}(\noyau) = {\| \CONV(\noyau,\noyau,\text{padding zero} = P, \text{stride} = S) - \Iro \|}_F^2 - (M-CS)~.
\]
\end{itemize}

\subsubsection{Measures of Deviation from Orthogonality}\label{sec 1D errors}

The orthogonality errors are defined by
\[\ERR{F}{N}(\Kbf) = \left\{\begin{array}{ll}
\|\Kk\Kk^T -\Id_{MN}\|_F & \mbox{, in the RO case,}\\
\|\Kk^T\Kk -\Id_{CSN}\|_F& \mbox{, in the CO case,}
\end{array}\right.
\]
and
\[\ERR{s}{N}(\Kbf) = \left\{\begin{array}{ll}
\|\Kk\Kk^T -\Id_{MN}\|_2 & \mbox{, in the RO case,}\\
\|\Kk^T\Kk -\Id_{CSN}\|_2& \mbox{, in the CO case.}
\end{array}\right.
\]

\section{The Convolutional Layer as a Matrix-Vector Product}\label{conv as matrix vec prod}
In this section, we write the convolutional layer as a matrix-vector product. 
In other words, we explicit $\Kk$ and the ingredients composing it.
The notation and preliminary results are useful in the proofs.
Note that the results are already known and can be found for example in \citet{sedghi2018singular}.
\subsection{1D Case}
%\subsubsubsection{Preliminaries}
We denote by $\Sn \in \mathbb{R}^{N \times SN}$ the sampling matrix (i.e., for $x = (x_0,\ldots,x_{SN-1})^T \in \mathbb{R}^{SN}$, we have for all $m \in \llbracket0,N-1\rrbracket$, $(S_Nx)_m = x_{Sm}$). \\
Put differently, we have 
\begin{align}\label{def S_N}
    S_N = \sum_{i=0}^{N-1} E_{i,Si} \;.
\end{align}
Also, note that, using \eqref{produit matrice elementaires} and \eqref{E ab^T = bar E b,a}, we have $S_N S_N^T = Id_N$
and 
\begin{align}\label{calcul S_N^T S_N}
    S_N^T S_N = \sum_{i=0}^{N-1} F_{Si,Si} \;.
\end{align}
For a vector
$x = (x_0,\ldots,x_{n-1})^T \in \mathbb{R}^n $, we denote by $C(x) \in \mathbb{R}^{n \times n}$ the circulant matrix defined by
\begin{align}\label{def mat circulante}
    C(x) = \left( \begin{array}{c c c c c}
         x_0 & x_{n-1} & \cdots & x_{2} & x_{1}  \\
         x_{1} & x_0 & x_{n-1} &  & x_{2} \\
         \vdots & x_{1} & x_0 & \ddots & \vdots \\
         x_{n-2} &  & \ddots & \ddots & x_{n-1} \\
         x_{n-1} & x_{n-2} & \cdots & x_{1} & x_0
    \end{array} \right) \;.
\end{align}
In other words, for $x \in \mathbb{R}^{n}$ and $X \in \mathbb{R}^{n \times n}$,
we have
\begin{align}\label{cond mat circulante}
    X = C(x) \iff \forall m,l \in \llbracket 0,n-1 \rrbracket , \  X_{m,l} = x_{(m-l)\%n} \;.
\end{align}
The notation for the circulant matrix $C(.)$ should not be confused with the number of the input channels $C$.
We also denote by $\Tilde{x} \in \RR^n$  the vector such that for all $i \in \llbracket0,n-1 \rrbracket$ , $\Tilde{x}_i = x_{(-i)\%n}$.
Again, the notation $\Tilde{x}$, for $x \in \RR^n$, should not be confused with $\Tilde{a}$, for $a \in \mathbb{Z}$.
We have 
\begin{align}\label{transpose circulant matrix}
    C(x)^T = C(\Tilde{x}) \;.
\end{align}
Also, for $x,y \in \RR^n$, we have 
\begin{align}\label{product circulant matrices}
    C(x)C(y) = C(x*y),
\end{align}
where $x*y \in \RR^n$, is such that for all $j \in \llbracket0,n-1\rrbracket$, 
\begin{align}\label{conv circulaire formule}
    [x*y]_j = \sum_{i=0}^{n-1}x_i y_{(j-i)\%n}.
\end{align}
$x*y$ is extended by $n$-periodicity.
Note that here $x*y$ denotes the classical convolution as defined in math (i.e. by flipping the second argument).
Note also that $x*y = y*x$ and therefore
\begin{align}\label{commutativity circulant matrices}
    C(x)C(y) = C(y)C(x) \;.
\end{align}
Throughout the article, the size of a filter is smaller than the size of the signal ($k=2r+1 \leq SN$).
For $n \geq k$, we introduce an embedding $P_n$ which associates to each
$h = (h_0,\ldots,h_{2r})^T \in \mathbb{R}^k$ the corresponding vector
\begin{align*}
    P_n(h) = (h_r,\ldots,h_1,h_0,0,\ldots,0,h_{2r},\ldots,h_{r+1})^T \in \mathbb{R}^n \;.
\end{align*}
Setting $[P_n(h)]_{i} = [P_n(h)]_{i\%n}$ for all $i \in \mathbb{Z}$, we have the following formula for
$P_n$: for $i \in \llbracket -r,-r+n-1 \rrbracket$,
\begin{align}\label{eq def P_n}
{[P_n(h)]}_i = \left\{ \begin{array}{l l}
h_{r-i} & \text{ if } i \in \llbracket-r,r \rrbracket \\
0 & \text{ otherwise. } 
\end{array} \right.
\end{align}

\textbf{Single-channel case:}
Let $x=(x_0,\ldots,x_{SN-1})^T \in \mathbb{R}^{SN}$ be a 1D signal.
We denote by $\text{Circular\_Conv}(h,x, \text{stride} = 1)$ the result of the circular convolution\footnote{as defined in machine learning (we do not flip h).} of $x$ with the kernel $h=(h_0,\ldots,h_{2r})^T \in \mathbb{R}^k$.
We have
\begin{align*}
    \text{Circular\_Conv}(h,x, \text{stride} = 1)
    &= \left( \sum_{i'=0}^{k-1} h_{i'} x_{(i'+i-r)\%SN}\right)_{i=0..SN-1} \;.
\end{align*}
Written as a matrix-vector product, this becomes
\begin{align*}
    & \left( \begin{array}{c c c c c c}
       h_0  & \cdots & h_{2r} & 0 & \cdots &0 \\
        0 & \ddots & \ddots & \ddots & \ddots & \vdots  \\
        \vdots & \ddots & \ddots & \ddots & \ddots & 0 \\
        0& \cdots & 0 & h_0  & \cdots & h_{2r} \\
    \end{array} \right)
    \left( \begin{array}{c}
         x_{SN-r} \\ \vdots \\ x_{SN-1} \\ x_0 \\ \vdots \\ x_{SN-1} \\ x_0 \\ \vdots \\ x_{r-1}
    \end{array} \right) \in \mathbb{R}^{SN} \\ 
    &= \left( \begin{array}{c c c c c c c c c c}
        h_r & h_{r+1} & \cdots & h_{2r} & 0 & \cdots & 0 & h_0 & \cdots & h_{r-1} \\
        h_{r-1} & \ddots & \ddots & \ddots & \ddots & \ddots & \ddots & \ddots & \ddots & \vdots  \\
        \vdots & \ddots & \ddots & \ddots & \ddots & \ddots & \ddots & \ddots & \ddots & h_0 \\
        h_0 & \ddots & \ddots & \ddots & \ddots & \ddots & \ddots & \ddots & \ddots & 0 \\
        0 & \ddots & \ddots & \ddots & \ddots & \ddots & \ddots & \ddots & \ddots & \vdots \\
        \vdots & \ddots & \ddots & \ddots & \ddots & \ddots & \ddots & \ddots & \ddots & 0 \\
        0 & \ddots & \ddots & \ddots & \ddots & \ddots & \ddots & \ddots & \ddots & h_{2r} \\
        h_{2r} & \ddots & \ddots & \ddots & \ddots & \ddots & \ddots & \ddots & \ddots & \vdots \\
        \vdots & \ddots & \ddots & \ddots & \ddots & \ddots & \ddots & \ddots & \ddots & h_{r+1} \\
        h_{r+1} & \cdots & h_{2r} & 0 & \cdots & 0 & h_0 & \cdots & h_{r-1} & h_r \\
    \end{array} \right) x \\
    &= C(P_{SN}(h)) x \;.
\end{align*}
The strided convolution is
\begin{align}\label{circular conv 1D 1 channel}
    \text{Circular\_Conv}(h,x,\text{stride}=S) = S_NC(P_{SN}(h)) x \in \mathbb{R}^N \;.
\end{align}
Notice that $S_NC(P_{SN}(h)) \in \RR^{N \times SN}$.

\textbf{Multi-channel convolution:}
Let $X \in \mathbb{R}^{C \times SN}$ be a multi-channel 1D signal.
We denote by $\text{Circular\_Conv}(\noyau,X,\text{stride} = S)$ the result of the strided circular convolutional layer of kernel $\noyau \in \mathbb{R}^{M \times C \times k}$ applied to $X$.
Using \eqref{circular conv 1D 1 channel} for all the input-output channel correspondences, we have
$
    Y = \text{Circular\_Conv}(\noyau,X,\text{stride}=S) \in \mathbb{R}^{M \times N}
$
if and only if
\begin{align*}
    \Vect(Y) &= 
    \left(
\begin{array}{c c c}
  S_NC(P_{SN}(\noyau_{1,1})) & \ldots & S_NC(P_{SN}(\noyau_{1,C})) \\
  \vdots & \vdots & \vdots \\
  S_NC(P_{SN}(\noyau_{M,1})) &\ldots & S_NC(P_{SN}(\noyau_{M,C}))
\end{array}
\right) \Vect(X) \;,
\end{align*}
where $\noyau_{i,j} = \noyau_{i,j,:} \in \mathbb{R}^k$.
Therefore,
\begin{align}\label{def mathcal K 1D}
\mathcal{K} = \left(
\begin{array}{c c c}
  S_NC(P_{SN}(\noyau_{1,1})) & \ldots & S_NC(P_{SN}(\noyau_{1,C})) \\
  \vdots & \vdots & \vdots \\
  S_NC(P_{SN}(\noyau_{M,1})) &\ldots & S_NC(P_{SN}(\noyau_{M,C}))
\end{array}
\right) \in \mathbb{R}^{MN \times CSN}
\end{align}
is the \Kcalname~associated to kernel $\noyau$.

\subsection{2D Case}
Notice that, since they are very similar, the proofs and notation are detailed in the 1D case, but we only provide a sketch of the proof and the main equations in 2D.
In order to distinguish between the 1D and 2D versions of $C(.)$, $P_n$ and $S_N$, we use calligraphic symbols in the 2D case.
%\subsubsection{Preliminaries}
We denote by $\Ss_N \in \RR^{N^2 \times S^2N^2}$ the sampling matrix in the 2D case (i.e., for a matrix $x \in \RR^{SN \times SN}$, if we denote by $z \in \RR^{N \times N}$, such that for all $i,j \in \llbracket0,N-1\rrbracket$, $z_{i,j} = x_{Si,Sj}$, then $\Vect(z) = \Ss_N \Vect(x)$).\\
For a matrix
$x \in \mathbb{R}^{n \times n} $, we denote by $\Cc(x) \in \mathbb{R}^{n^2 \times n^2}$ the doubly-block circulant matrix defined by
\begin{align*}%\label{def mat doubly-block circulante}
    \Cc(x) = \left( \begin{array}{c c c c c}
         C(x_{0,:}) & C(x_{n-1,:}) & \cdots & C(x_{2,:}) & C(x_{1,:})  \\
         C(x_{1,:}) & C(x_{0,:}) & C(x_{n-1,:}) &  & C(x_{2,:}) \\
         \vdots & C(x_{1,:}) & C(x_{0,:}) & \ddots & \vdots \\
         C(x_{n-2,:}) &  & \ddots & \ddots & C(x_{n-1,:}) \\
         C(x_{n-1,:}) & C(x_{n-2,:}) & \cdots & C(x_{1,:}) & C(x_{0,:})
    \end{array} \right) \;.
\end{align*}

For $n \geq k=2r+1$, we introduce the operator $\Pp_n$ which associates to a matrix
$h \in \mathbb{R}^{k \times k}$ the corresponding matrix
\begin{align*}
    \Pp_n(h) = \left(\begin{array}{c c c c c c c c c}
     h_{r,r} & \cdots & h_{r,0} & 0 & \cdots & 0 & h_{r,2r} & \cdots & h_{r,r+1}  \\
     \vdots & \vdots & \vdots & \vdots & \vdots & \vdots & \vdots & \vdots & \vdots \\
     h_{0,r} & \cdots & h_{0,0} & 0 & \cdots & 0 & h_{0,2r} & \cdots & h_{0,r+1} \\
     0 & \dots & 0 & 0 & \dots & 0 & 0 & \dots & 0 \\
     \vdots & \vdots & \vdots & \vdots & \vdots & \vdots & \vdots & \vdots & \vdots \\
     0 & \dots & 0 & 0 & \dots & 0 & 0 & \dots & 0 \\
     h_{2r,r} & \cdots & h_{2r,0} & 0 & \cdots & 0 & h_{2r,2r} & \cdots & h_{2r,r+1}  \\
     \vdots & \vdots & \vdots & \vdots & \vdots & \vdots & \vdots & \vdots & \vdots \\
     h_{r+1,r} & \cdots & h_{r+1,0} & 0 & \cdots & 0 & h_{r+1,2r} & \cdots & h_{r+1,r+1}  
\end{array} \right) \in \mathbb{R}^{n \times n} \;.
\end{align*}
Setting $[\Pp_n(h)]_{i,j} = [\Pp_n(h)]_{i\%n,j\%n}$ for all $i,j \in \mathbb{Z}$, we have the following formula for
$\Pp_n$: for $(i,j) \in \llbracket -r,-r+n-1 \rrbracket^2$,
\begin{align*}%\label{eq def Pp_n}
{[\Pp_n(h)]}_{i,j} = \left\{ \begin{array}{l l}
h_{r-i,r-j} & \text{ if } (i,j) \in \llbracket-r,r \rrbracket^2 \\
0 & \text{ otherwise. }
\end{array} \right.
\end{align*}

\textbf{Single-channel case:}
Let $x \in \mathbb{R}^{SN \times SN}$ be a 2D image.
We denote by $\text{Circular\_Conv}(h,x, \text{stride} = 1)$ the result of the circular convolution of $x$ with the kernel $h \in \mathbb{R}^{k \times k}$.
As in the 1D case, we have
\begin{align*}
    y = \text{Circular\_Conv}(h,x, \text{stride} = 1) \iff \Vect(y) = \Cc(\Pp_{SN}(h)) \Vect(x)
\end{align*}
and the strided circular convolution
\begin{align*}
    y = \text{Circular\_Conv}(h,x, \text{stride} = S) \iff \Vect(y) = \Ss_N \Cc(\Pp_{SN}(h)) \Vect(x) \;.
\end{align*}
Notice that $\Ss_N \Cc(\Pp_{SN}(h)) \in \RR^{N^2 \times S^2N^2}$.

\textbf{Multi-channel convolution :}
Let $X \in \mathbb{R}^{C \times SN \times SN}$ be a multi-channel 2D image.
We denote by $\text{Circular\_Conv}(\noyau,X,\text{stride} = S)$ the result of the strided circular convolutional layer of kernel $\noyau \in \mathbb{R}^{M \times C \times k \times k}$ applied to $X$.
We have
$
    Y = \text{Circular\_Conv}(\noyau,X,\text{stride}=S) \in \mathbb{R}^{M \times N \times N}
$
if and only if
\begin{align*}
    \Vect(Y) &= 
    \left(
\begin{array}{c c c}
  \Ss_N \Cc(\Pp_{SN}(\noyau_{1,1})) & \ldots & \Ss_N \Cc(\Pp_{SN}(\noyau_{1,C})) \\
  \vdots & \vdots & \vdots \\
  \Ss_N \Cc(\Pp_{SN}(\noyau_{M,1})) &\ldots & \Ss_N \Cc(\Pp_{SN}(\noyau_{M,C}))
\end{array}
\right) \Vect(X) \;,
\end{align*}
where $\noyau_{i,j} = \noyau_{i,j,:,:} \in \mathbb{R}^{k \times k}$.
Therefore,
\begin{align*}%\label{def mathcal K 2D}
\mathcal{K} = \left(
\begin{array}{c c c}
  \Ss_N \Cc(\Pp_{SN}(\noyau_{1,1})) & \ldots & \Ss_N \Cc(\Pp_{SN}(\noyau_{1,C})) \\
  \vdots & \vdots & \vdots \\
  \Ss_N \Cc(\Pp_{SN}(\noyau_{M,1})) &\ldots & \Ss_N \Cc(\Pp_{SN}(\noyau_{M,C}))
\end{array}
\right) \in \mathbb{R}^{MN^2 \times CS^2N^2}
\end{align*}
is the \Kcalname~associated to kernel $\noyau$.

%%%%%%%%%%%%%%%%%%%%%%%%%%%%%%%%%%%%%%%

\section{Proof of Theorem \ref{Prop existence cco}}\label{proof Prop existence cco}

As the proofs are very similar in the 1D and 2D cases, we give the full proof in the 1D case, in Section \ref{thm4_1D_proof_sec}, and we only give a sketch of the proof in the 2D case, in Section \ref{proof2D_exist_sec}.

We first prove the result in the RO case, then in the CO case. In each case, we prove separately the statement when an orthogonal convolutional layer exists and when no orthogonal convolutional layer exists. When the architecture places us in the former case, to prove an orthogonal convolutional layer exists, we exhibit an explicit kernel tensor $\Kbf$ and do the calculations to prove that $\Kk$ is orthogonal. The calculations are based on Lemma \ref{simplif 1}, in the RO case, and Lemma \ref{simplif 2}, in the CO case. Lemma \ref{simplif 0} synthesize the result of a calculation that is used to prove both Lemma \ref{simplif 1} and Lemma \ref{simplif 2}.  When on the contrary the architecture is such that there does not exist any orthogonal convolutional layer, we prove that for all $\Kbf$ the architecture condition implies $\rk(\Kk \Kk^T) <\rk(Id_{MN})$, in the RO case, and $\rk(\Kk^T\Kk) < \rk(Id_{CSN})$, in the CO case. This proves that no orthogonal convolutional layer exists.

\subsection{Proof of Theorem \ref{Prop existence cco}, for 1D Convolutional Layers}\label{thm4_1D_proof_sec}

We start by stating and proving three intermediate lemmas.
Recall that $k=2r+1$ and from \eqref{def bases canoniques}, that ${(e_i)}_{i=0..k-1} = \Bb(\RR^k)$ and ${(E_{a,b})}_{a=0..N-1,b=0..SN-1} = \Bb(\RR^{N \times SN})$.
\begin{lemma}\label{simplif 0}
    Let $j \in \llbracket0,k-1 \rrbracket$. We have
    \begin{align*}
        S_N C\left(P_{SN}\left(e_j\right)\right) 
        &= \sum_{i=0}^{N-1} E_{i,Si+j-r} \;.
    \end{align*}
\end{lemma}
\begin{proof}
Let $j \in \llbracket0, k-1 \rrbracket$.
Using \eqref{eq def P_n}, \eqref{def bases canoniques}, \eqref{def bases pour Z} and \eqref{def mat circulante}, we have 
\begin{align*}
        C(P_{SN}(e_j)) = C(f_{r-j}) 
        = \sum_{i=0}^{SN-1} F_{i,i-(r-j)} 
        = \sum_{i=0}^{SN-1} F_{i,i+j-r} \;.
\end{align*}
Using \eqref{def S_N} and \eqref{produit matrice elementaires}, we have
\begin{align*}
    S_N C\left(P_{SN}\left(e_j\right)\right) 
    = \left(\sum_{i=0}^{N-1} E_{i,Si}\right)\left(\sum_{i'=0}^{SN-1} F_{i',i'+j-r}\right)
    = \sum_{i=0}^{N-1} E_{i,Si+j-r} \;.
\end{align*}
\end{proof}

\begin{lemma}\label{simplif 1}
    Let $k_S = \min(k,S)$ and $j,l \in \llbracket0,k_S-1 \rrbracket$.
    We have
\begin{align*}
    S_N C(P_{SN}(e_j)) C(P_{SN}(e_l))^T S_N^T = \delta_{j=l} Id_N \;.
\end{align*}
\end{lemma}
\begin{proof}
Let $j,l \in \llbracket 0,k_S-1 \rrbracket$.
Since $k_S \leq k$, using Lemma \ref{simplif 0} and \eqref{E ab^T = bar E b,a},
\begin{align}
    S_N C\left(P_{SN}\left(e_j\right)\right) C\left(P_{SN}\left(e_l\right)\right)^T S_N^T
     &= \left(\sum_{i=0}^{N-1} E_{i,Si+j-r}\right)
    \left(\sum_{i'=0}^{N-1} {E}_{i',Si'+l-r}\right)^T \nonumber \\
    &= \left(\sum_{i=0}^{N-1} E_{i,Si+j-r}\right)
    \left(\sum_{i'=0}^{N-1} \overline{E}_{Si'+l-r,i'}\right) \;. \label{intermediaire calcul 1}
\end{align}
We know from \eqref{produit matrice elementaires} that $E_{i,Si+j-r} \overline{E}_{Si'+l-r,i'} = \delta_{\widehat{Si+j-r} = \widehat{Si'+l-r}} G_{i,i'}$.
But for $i,i' \in \llbracket0,N-1 \rrbracket$ and $j,l \in \llbracket0,k_S-1 \rrbracket$, since $k_S \leq S$,
we have 
$$ -r \leq Si+j-r \leq S(N-1) +k_S-1 -r \leq SN-1 -r .$$
Similarly, $Si'+l-r \in \llbracket-r,SN-1-r \rrbracket$.
Therefore, ${Si+j-r}$ and  ${Si'+l-r}$ lie in the same interval of size $SN$, hence
\begin{align*}
    \widehat{Si+j-r} = \widehat{Si'+l-r} \iff {Si+j-r} = {Si'+l-r} 
     \iff Si+j=Si'+l \;.
\end{align*}
If $Si+j=Si'+l$, then $$ |S(i-i')| = |j-l| < k_S \leq S.$$
Since $|i-i'| \in \mathbb{N}$, the latter inequality implies $i=i'$ and, as a consequence, $j=l$.
Finally,
\begin{align*}
    \widehat{Si+j-r} = \widehat{Si'+l-r} \iff i=i' \text{ and } j=l \;.
\end{align*}
Hence, using \eqref{produit matrice elementaires}, the equality \eqref{intermediaire calcul 1} becomes
\begin{align*}
    S_N C(P_{SN}(e_j)) C(P_{SN}(e_l))^T S_N^T &= \delta_{j=l} \sum_{i=0}^{N-1} G_{i,i}
    = \delta_{j=l} Id_N \;.
\end{align*}
\end{proof}
\begin{lemma}\label{simplif 2}
Let $S \leq k $.
We have
\begin{align*}
    \sum_{z=0}^{S-1} C(P_{SN}(e_z))^T S_N^T S_N C(P_{SN}(e_z)) = Id_{SN} \;.
\end{align*}
\end{lemma}
\begin{proof}
Let $z \in \llbracket0,S-1 \rrbracket$. Since $S \leq k$, we have $z \in \llbracket0,k-1 \rrbracket$. Hence using Lemma \ref{simplif 0}, then \eqref{E ab^T = bar E b,a} and \eqref{produit matrice elementaires}, we have 
\begin{align*}
    C\left(P_{SN}\left(e_z\right)\right)^T S_N^T S_N C\left(P_{SN}\left(e_z\right)\right) &= \left(\sum_{i=0}^{N-1} E_{i,Si+z-r}\right)^T
    \left(\sum_{i'=0}^{N-1} E_{i',Si'+z-r}\right) \\
    &= \left(\sum_{i=0}^{N-1} \overline{E}_{Si+z-r,i}\right)
    \left(\sum_{i'=0}^{N-1} E_{i',Si'+z-r}\right) \\
    &= \sum_{i=0}^{N-1} F_{Si+z-r,Si+z-r} \;.
\end{align*}
Hence
\begin{align*}
    \sum_{z=0}^{S-1} C(P_{SN}(e_z))^T S_N^T S_N C(P_{SN}(e_z)) = \sum_{z=0}^{S-1} \sum_{i=0}^{N-1} F_{Si+z-r,Si+z-r} \;.
\end{align*}
But, for $z \in \llbracket0,S-1 \rrbracket$ and $i \in \llbracket 0,N-1 \rrbracket$, $Si+z-r$ traverses $\llbracket -r,SN-1-r \rrbracket$.
Therefore, using \eqref{def bases pour Z}
\begin{align*}
    \sum_{z=0}^{S-1} C(P_{SN}(e_z))^T S_N^T S_N C(P_{SN}(e_z)) = \sum_{i=-r}^{SN-1-r} F_{i,i} = \sum_{i=0}^{SN-1} F_{i,i} = Id_{SN} \;.
\end{align*}
\end{proof}
\begin{proof}[Proof of Theorem \ref{Prop existence cco}]
Let $N$ be a positive integer such that $ SN \geq k$.\\
We start by proving the theorem in the RO case.\\
\textbf{Suppose $CS \geq M$ and $M \leq Ck$:}\\
Let us exhibit $\noyau \in \RR^{M \times C \times k}$ such that $\Kk \Kk^T = Id_{MN}$.\\
Let $k_S = \min(k,S)$.
Since $M \leq CS$ and $M \leq Ck$, we have $1 \leq M \leq Ck_S$.
Therefore, there exist a unique couple $(i_{max}, j_{max}) \in \llbracket 0,k_S-1 \rrbracket \times \llbracket 1,C \rrbracket$ such that $M = i_{max}C + j_{max}$.
We define the  \Kbfname~  $\noyau \in \mathbb{R}^{M \times C \times k}$ as follows:
For all $(i,j) \in \llbracket 0,k_S-1 \rrbracket \times \llbracket 1,C \rrbracket$ such that $iC+j \leq M$, we set $\noyau_{iC+j,j} = e_i$, and $\noyau_{u,v} = 0$ for all the other indices. Put differently, if we write $\Kbf$ as a 3rd order tensor (where the rows represent the first dimension, the columns the second one, and the $\Kbf_{i,j} \in \mathbb{R}^k$ are in the third dimension) we have :
\begin{align*} 
\noyau 
    = \begin{bmatrix}
    \noyau_{1,1} & \cdots & \noyau_{1,C}  \\
         \vdots & \ddots & \vdots \\
         \noyau_{C,1} & \cdots & \noyau_{C,C} \\
    \noyau_{C+1,1} & \cdots & \noyau_{C+1,C}  \\
         \vdots & \ddots & \vdots \\
         \noyau_{2C,1} & \cdots & \noyau_{2C,C} \\
         & \vdots & \\
    \noyau_{i_{max}C+1,1} & \cdots & \noyau_{i_{max}C+1,C} \\ \vdots & \ddots & \vdots
    \end{bmatrix}
    = \begin{bmatrix}
    e_0 & &  \\ 0& \ddots &0 \\  & & e_0 \\
    e_1 & &  \\ 0& \ddots &0 \\  & & e_1 \\
    & \vdots & \\
    e_{i_{max}} & &  \\ 0& \ddots &0 
    \end{bmatrix}
    \in \mathbb{R}^{M \times C \times k} \;,
\end{align*}
where $e_{i_{max}}$ appears $j_{max}$ times.
Therefore, using \eqref{def mathcal K 1D}, we have
\begin{align*}
    \Kk &= \begin{bmatrix}
    S_N C(P_{SN}(e_0)) & &  \\ 0& \ddots &0 \\  & & S_N C(P_{SN}(e_0)) \\
    S_N C(P_{SN}(e_1)) & &  \\ 0& \ddots &0 \\  & & S_N C(P_{SN}(e_1)) \\
    & \vdots & \\
    S_N C(P_{SN}(e_{i_{max}})) & &  \\ 0& \ddots &0 
    \end{bmatrix} \in \mathbb{R}^{MN \times CSN} \;,
\end{align*}
where $S_N C(P_{SN}(e_{i_{max}}))$ appears $j_{max}$ times.
We have $\Kk = D_{1:MN,:}$, where we set
\begin{align*}
    D &= \begin{bmatrix}
    S_N C(P_{SN}(e_0)) & &  \\ 0& \ddots &0 \\  & & S_N C(P_{SN}(e_0)) \\
    S_N C(P_{SN}(e_1)) & &  \\ 0& \ddots &0 \\  & & S_N C(P_{SN}(e_1)) \\
    & \vdots & \\
    S_N C(P_{SN}(e_{k_S-1})) & &  \\ 0& \ddots &0 \\ & & S_N C(P_{SN}(e_{k_S-1}))
    \end{bmatrix} \in \mathbb{R}^{k_SCN \times CSN} \;.
\end{align*}
But, for $j,l \in \llbracket 0,k_S-1 \rrbracket$, the $(j,l)$-th block of size $(CN,CN)$ of $DD^T$ is :
\begin{align*}
    \begin{bmatrix}
    S_N C(P_{SN}(e_j)) & &  \\ 0& \ddots &0 \\  & & S_N C(P_{SN}(e_j))
    \end{bmatrix}
    \begin{bmatrix}
    C(P_{SN}(e_l))^T S_N^T & &  \\ 0& \ddots &0 \\  & & C(P_{SN}(e_l))^T S_N^T
    \end{bmatrix} \;,
\end{align*}
which is equal to
\begin{align*}
    \begin{bmatrix}
    S_N C(P_{SN}(e_j)) C(P_{SN}(e_l))^T S_N^T & &  \\ 0& \ddots &0 \\  & & S_N C(P_{SN}(e_j)) C(P_{SN}(e_l))^T S_N^T
    \end{bmatrix}.
\end{align*}
Using Lemma \ref{simplif 1}, this is equal to $\delta_{j=l} Id_{CN}$.
Hence, $DD^T = Id_{k_SCN}$, and therefore,
\begin{align*}
    \Kk \Kk^T = D_{1:MN,:} (D_{1:MN,:})^T = (DD^T)_{1:MN,1:MN} = Id_{MN} \;.
\end{align*}
This proves the first implication in the RO case, i.e., if $M \leq Ck$, then $\KbbO_1 \neq \emptyset$.\\

\textbf{Suppose $CS \geq M$ and $M > Ck$:}\\
We need to prove that for all $\noyau \in \RR^{M \times C \times k}$, we have $\Kk \Kk^T \neq Id_{MN}$.\\ 
Since for all $(i,j) \in \llbracket1,M \rrbracket \times \llbracket 1,C \rrbracket$, each of the $ N $ rows of $ S_N C (P_{SN} (\noyau_{i, j})) $ has at most $ k $ non-zero elements, the number of non-zero columns of $ S_N C (P_{SN} (\noyau_{i, j})) $ is less than or equal to $ kN $.
Also, for all $ i, i ' \in \llbracket1,M \rrbracket$, the columns of $ S_N C (P_{SN} (\noyau_{i, j})) $ which can be non-zero are the same as those of $ S_N C (P_ {SN} (\noyau_{i ', j})) $.
Hence, we have for all $ j $, the number of non-zero columns of $\begin{bmatrix}
S_N C(P_{SN}(\noyau_{1,j})) \\ \vdots \\ S_N C(P_{SN}(\noyau_{M,j}))
\end{bmatrix} $ is less than or equal to $kN$.
Therefore, the number of non-zero columns of $ \Kk $ is less than or equal to $ CkN $.
Hence, since $Ck <M$, we have $\rk(\Kk \Kk^T) \leq \rk(\Kk) \leq CkN < MN = \rk(Id_{MN})$.
Therefore, for all $\noyau \in \mathbb{R}^{M \times C \times k}$, we have 
$\Kk \Kk^T \neq Id_{MN} \;.$ \\
This proves that if $CS \geq M$ and $M>Ck$, then $\KbbO_1 = \emptyset$.
This concludes the proof in the RO case.\\

\textbf{Suppose $M \geq CS$ and $S \leq k$:}\\
Let us exhibit $\noyau \in \RR^{M \times C \times k}$ such that $ \Kk^T \Kk = Id_{CSN}$.\\
For all $(i,j) \in \llbracket 0,S-1 \rrbracket \times \llbracket 1,C \rrbracket$, we set $\noyau_{iC+j,j} = e_i$, and $\noyau_{u,v} = 0$  for all the other indices. Put differently, if we write $\Kbf$ as a 3rd order tensor, we have
\begin{align*}
    \noyau
    =
    \begin{bmatrix}
    \noyau_{1,1} & \cdots & \noyau_{1,C}  \\
         \vdots & \ddots & \vdots \\
         \noyau_{C,1} & \cdots & \noyau_{C,C} \\
    \noyau_{C+1,1} & \cdots & \noyau_{C+1,C}  \\
         \vdots & \ddots & \vdots \\
         \noyau_{2C,1} & \cdots & \noyau_{2C,C} \\
         & \vdots & \\
    \noyau_{(S-1)C+1,1} & \cdots & \noyau_{(S-1)C+1,C} \\
    \vdots & \ddots & \vdots \\
    \noyau_{CS,1} & \cdots & \noyau_{CS,C} \\
    \noyau_{CS+1,1} & \cdots & \noyau_{CS+1,C} \\
    \vdots & \vdots & \vdots \\
    \noyau_{M,1} & \cdots & \noyau_{M,C} \\
    \end{bmatrix}
    = \begin{bmatrix}
    e_0 & &  \\ 0 & \ddots & 0 \\  & & e_0 \\
    e_1 & &  \\ 0 & \ddots & 0 \\  & & e_1 \\
    & \vdots & \\
    e_{S-1} & & \\ 0 & \ddots & 0 \\  & & e_{S-1} \\ 
    & & \\
    &  O & \\
    & & 
    \end{bmatrix} \in \mathbb{R}^{M \times C \times k} \;,
\end{align*}
where $O=0_{(M-CS) \times C \times k}$ denotes the null tensor.
Therefore, using \eqref{def mathcal K 1D}, we have
\begin{align*}
    \Kk &= \begin{bmatrix}
    S_N C(P_{SN}(e_0)) & &  \\ 0 & \ddots & 0 \\  & & S_N C(P_{SN}(e_0)) \\
    S_N C(P_{SN}(e_1)) & &  \\ 0 & \ddots & 0 \\  & & S_N C(P_{SN}(e_1)) \\
    & \vdots & \\
    S_N C(P_{SN}(e_{S-1})) & &  \\ 0 & \ddots & 0 \\  & & S_N C(P_{SN}(e_{S-1})) \\
    & \mathcal{O}
    \end{bmatrix} \in \mathbb{R}^{MN \times CSN} \;,
\end{align*}
where $\mathcal{O}=0_{(MN-CSN) \times CSN}$ denotes the null matrix.
Hence, $\Kk^T \Kk$ equals
\begin{align*}
    \begin{bmatrix}
    \sum_{z=0}^{S-1} C(P_{SN}(e_z))^T S_N^T S_N C(P_{SN}(e_z)) & & 0 \\ & \ddots & \\ 0 & & \sum_{z=0}^{S-1} C(P_{SN}(e_z))^T S_N^T S_N C(P_{SN}(e_z)) 
    \end{bmatrix} \;.
\end{align*}
Using Lemma \ref{simplif 2}, we obtain $\Kk^T \Kk = Id_{CSN} \;.$ \\
This proves that in the CO case, if $S \leq k$, then $\KbbO_1 \neq \emptyset$.\\

\textbf{Suppose $M \geq CS$ and $S > k$:}\\
We need to prove that for all $\noyau \in \RR^{M \times C \times k}$, we have $\Kk^T \Kk \neq Id_{CSN}$.\\ 
Following the same reasoning as in the case $ CS \geq M $ and $ M > Ck $, we have that the number of non-zero columns of $ \Kk $ is less than or equal to $ CkN $.
So, since $ k <S $, we have $\rk(\Kk^T\Kk) \leq \rk(\Kk) \leq CkN < CSN = \rk(Id_{CSN})$.
Therefore, for all $ \noyau \in \mathbb{R}^{M \times C \times k} $, we have
$\Kk^T \Kk \neq Id_{CSN} \;.$ \\
This proves that in the CO case, if $k<S$, then $\KbbO_1 = \emptyset$.
This concludes the proof.
\end{proof}

\subsection{Sketch of the Proof of Theorem \ref{Prop existence cco}, for 2D Convolutional Layers}\label{proof2D_exist_sec}
We first set $(e_{i,j})_{i=0..k-1,j=0..k-1} = \Bb(\RR^{k \times k})$.
As in the 1D case, we have the following two lemmas
\begin{lemma}\label{simplif 1 2D}
    Let $k_S = \min(k,S)$ and $j,j',l,l' \in \llbracket0,k_S-1 \rrbracket$.
    We have
\begin{align*}
    \Ss_N \Cc(\Pp_{SN}({e}_{j,j'})) \Cc(\Pp_{SN}({e}_{l,l'}))^T \Ss_N^T = \delta_{j=l}\delta_{j'=l'} Id_{N^2} \;.
\end{align*}
\end{lemma}
\begin{lemma}
Let $S \leq k $.
We have
\begin{align*}
    \sum_{z=0}^{S-1} \sum_{z'=0}^{S-1} \Cc(\Pp_{SN}({e}_{z,z'}))^T \Ss_N^T \Ss_N \Cc(\Pp_{SN}({e}_{z,z'})) = Id_{S^2N^2} \;.
\end{align*}
\end{lemma}

\textbf{For $CS^2 \geq M$ and $M \leq Ck^2$:}\\
We set $\overline{e}_{i+kj} = e_{i,j}$ for $i,j \in \llbracket0,k-1\rrbracket$.\\
Let $i_{max},j_{max} \in \llbracket 0,k_S^2-1 \rrbracket \times \llbracket 1,C \rrbracket$ such that $i_{max}C + j_{max} = M$.
We set 
\begin{align*} 
\noyau 
    = \begin{bmatrix}
    \noyau_{1,1} & \cdots & \noyau_{1,C}  \\
         \vdots & \ddots & \vdots \\
         \noyau_{C,1} & \cdots & \noyau_{C,C} \\
    \noyau_{C+1,1} & \cdots & \noyau_{C+1,C}  \\
         \vdots & \ddots & \vdots \\
         \noyau_{2C,1} & \cdots & \noyau_{2C,C} \\
         & \vdots & \\
    \noyau_{i_{max}C+1,1} & \cdots & \noyau_{i_{max}C+1,C} \\ \vdots & \ddots & \vdots
    \end{bmatrix}
    = \begin{bmatrix}
    \overline{e}_0 & &  \\ 0& \ddots &0 \\  & & \overline{e}_0 \\
    \overline{e}_1 & &  \\ 0& \ddots &0 \\  & & \overline{e}_1 \\
    & \vdots & \\
    \overline{e}_{i_{max}} & &  \\ 0& \ddots &0 
    \end{bmatrix}
    \in \mathbb{R}^{M \times C \times k \times k} \;,
\end{align*}
where $\overline{e}_{i_{max}}$ appears $j_{max}$ times.
Then we proceed as in the 1D case.\\

\textbf{For $CS^2 \geq M$ and $M > Ck^2$:}\\
Using the same argument as in 1D, we can conclude that the number of non-zero columns of $\Kk$ is less than or equal to $Ck^2N^2$.
Hence, $\rk(\Kk) \leq Ck^2N^2 < MN^2$. Therefore, for all $\noyau \in \mathbb{R}^{M \times C \times k \times k}$, we have 
$\Kk \Kk^T \neq Id_{MN^2} \;.$

\textbf{For $M \geq CS^2$ and $S \leq k$:}\\
Denoting by $O \in \mathbb{R}^{(M-CS^2) \times C \times k \times k}$ the null 4th order tensor of size $(M-CS^2) \times C \times k \times k$,
we set
\begin{align*}
    \noyau 
    =
    \begin{bmatrix}
    \noyau_{1,1} & \cdots & \noyau_{1,C}  \\
         \vdots & \ddots & \vdots \\
         \noyau_{C,1} & \cdots & \noyau_{C,C} \\
    \noyau_{C+1,1} & \cdots & \noyau_{C+1,C}  \\
         \vdots & \ddots & \vdots \\
         \noyau_{2C,1} & \cdots & \noyau_{2C,C} \\
         & \vdots & \\
    \noyau_{C(S^2-1)+1,1} & \cdots & \noyau_{C(S^2-1)+1,C} \\
    \vdots & \ddots & \vdots \\
    \noyau_{CS^2,1} & \cdots & \noyau_{CS^2,C} \\
    \noyau_{CS^2+1,1} & \cdots & \noyau_{CS^2+1,C} \\
    \vdots & \vdots & \vdots \\
    \noyau_{M,1} & \cdots & \noyau_{M,C} \\
    \end{bmatrix}
    = \begin{bmatrix}
    {e}_{0,0} & &  \\ 0 & \ddots & 0  \\  & & {e}_{0,0} \\
    {e}_{1,0} & &  \\ 0 & \ddots & 0 \\  & & {e}_{1,0} \\
    & \vdots & \\
    {e}_{S-1,S-1} & & \\ 0 & \ddots & 0 \\  & & {e}_{S-1,S-1} \\
    & & \\
    &  O & \\
    & & 
    \end{bmatrix} \in \mathbb{R}^{M \times C \times k \times k} \;.
\end{align*}
Then we proceed as in the 1D case.\\

\textbf{For $M \geq CS^2$ and $S > k$:}\\
By the same reasoning as in the 1D case, we have that the number of non-zero columns of $ \Kk $ is less than or equal to $ Ck^2N^2 $.
So, since $ k<S $, we have $ \rk(\Kk) \leq Ck^2N^2 < CS^2N^2 $.
Therefore, for all $ \noyau \in \mathbb{R}^{M \times C \times k \times k} $, we have
$\Kk^T \Kk \neq Id_{CS^2N^2} \;.$

\section{Restrictions due to Boundary Conditions}
In this section, we prove the theorems related to 'valid' and 'same' boundary conditions.

\subsection{Proof of Proposition \ref{valid}}\label{Proof of proposition valid}
\begin{proof}
For a single-channel convolution of kernel $h \in \RR^k$ with 'valid' padding, the matrix applying the transformation on a signal $x \in \RR^N$ has the following form:
\begin{align*}
    A_N(h) :=\left( \begin{array}{c c c c c c}
       h_0  & \cdots & h_{2r} & & &0 \\
        & \ddots & \ddots & \ddots &  \\
        & & \ddots & \ddots & \ddots & \\
        0& & & h_0  & \cdots & h_{2r}
        \end{array}
        \right) \in \RR^{(N-k+1) \times N} \;.
\end{align*}
Hence, for $\noyau \in \RR^{M \times C \times k}$, the \Kcalname~is:
\begin{align*}
    \Kk = \left(
\begin{array}{c c c}
  A_N(\noyau_{1,1}) & \ldots & A_N(\noyau_{1,C}) \\
  \vdots & \vdots & \vdots \\
  A_N(\noyau_{M,1}) &\ldots & A_N(\noyau_{M,C})
\end{array}
\right) \in \mathbb{R}^{M(N-k+1) \times CN} \;.
\end{align*}
Let us focus on the columns corresponding to the first input channel.
To simplify the notation, for $m \in \llbracket 1,M \rrbracket$ we denote by $a^{(m)} := \noyau_{m,1} \in \RR^k$.
By contradiction, suppose that $\Kk^T \Kk = Id_{CN}$. In particular,
for the first block matrix of size $M(N-k+1) \times N$ of $\Kk$ (i.e., corresponding to the first input channel),
its first column, last column and column of index $2r$ are of norm 1.
Since $N \geq 2k-1$, we have 
\begin{align*}
        \sum_{m=1}^M \left(a^{(m)}_0\right)^2 = 1 , \qquad 
        \sum_{m=1}^M \left(a^{(m)}_{2r}\right)^2 = 1 \qquad \text{and} \qquad
        \sum_{i=0}^{2r} \sum_{m=1}^M \left(a^{(m)}_{i}\right)^2  = 1 \;.
\end{align*}
This is impossible.
Therefore, for all $\noyau \in \RR^{M \times C \times k}$, we have $\Kk^T \Kk \neq Id_{CN}$.\\
\end{proof}

\subsection{Proof of Proposition \ref{same}} \label{Proof of Proposition same}
\begin{proof}
For a single-channel convolution of kernel $h \in \RR^k$ with zero-padding 'same', the matrix applying the transformation on a signal $x \in \RR^N$ has the following form:
\begin{align*}
    A_N(h) :=\left( \begin{array}{c c c c c c}
       h_r  & \cdots & h_{2r} & 0 & \cdots & 0 \\
        \vdots & \ddots & \ddots & \ddots & \ddots & \vdots  \\
        h_0 & \ddots & \ddots & \ddots & \ddots & 0 \\
        0& \ddots & \ddots & \ddots  & \ddots & h_{2r} \\
        \vdots & \ddots & \ddots & \ddots  & \ddots & \vdots \\
        0 & \cdots & 0 & h_0 & \cdots & h_r
        \end{array}
        \right) \in \RR^{N \times N} \;.
\end{align*}
Hence, for $\noyau \in \RR^{M \times C \times k}$, the matrix that applies the convolutional layer is :
\begin{align*}
    \Kk = \left(
\begin{array}{c c c}
  A_N(\noyau_{1,1}) & \ldots & A_N(\noyau_{1,C}) \\
  \vdots & \vdots & \vdots \\
  A_N(\noyau_{M,1}) &\ldots & A_N(\noyau_{M,C})
\end{array}
\right) \in \mathbb{R}^{MN \times CN} \;.
\end{align*}
\textbf{Suppose $M \leq C$ (RO case):} If $\Kk$ is orthogonal, then $\Kk \Kk^T = Id_{MN}$.
Let us fix $m \in \llbracket1,M\rrbracket $.
Since $\Kk \Kk^T = Id_{MN}$, the first row, the last row and the row of index $r$ of the $m$-th block matrix of size $N \times CN$ of $\Kk$ are of norm equal to 1, i.e.
\begin{align*}
        \|\Kk_{(m-1)N,:}\|_2^2 = 1 , \qquad 
        \|\Kk_{mN-1,:}\|_2^2 = 1 \qquad \text{and} \qquad
        \|\Kk_{(m-1)N+r,:}\|_2^2 = 1 \;.
\end{align*}
To simplify the notation, for $c \in \llbracket 1,C \rrbracket$, we denote by $a^{(c)} := \noyau_{m,c} \in \RR^k$.
Since $N \geq k$, the previous equations are equivalent to
\begin{align*}
         \sum_{i=r}^{2r} \sum_{c=1}^C \left(a^{(c)}_i\right)^2 = 1 , \qquad 
        \sum_{i=0}^{r} \sum_{c=1}^C \left(a^{(c)}_{i}\right)^2 = 1 \qquad \text{and} \qquad
        \sum_{i=0}^{2r} \sum_{c=1}^C \left(a^{(c)}_{i}\right)^2  = 1 \;.
\end{align*}
Substracting the first equality from the third one, and the second equality from the third one, we obtain
\begin{align*}
         \sum_{i=0}^{r-1} \sum_{c=1}^C \left(a^{(c)}_i\right)^2 = 0 , \qquad 
        \sum_{i=r+1}^{2r} \sum_{c=1}^C \left(a^{(c)}_{i}\right)^2 = 0 \qquad \text{and} \qquad
        \sum_{i=0}^{2r} \sum_{c=1}^C \left(a^{(c)}_{i}\right)^2  = 1 \;.
\end{align*}
This implies that for all $c \in \llbracket 1,C \rrbracket$, for all $i \in \llbracket 0,2r \rrbracket \setminus \{r\}$, $a^{(c)}_{i} = 0$.\\
As a conclusion, for any $m \in \llbracket 1,M \rrbracket$, any $c \in \llbracket 1,C \rrbracket$, and any $i \in \llbracket 0,2r \rrbracket \setminus \{r\}$,
$$\noyau_{m,c,i} = 0.$$
This proves the result in the RO case.\\
The proof of the CO case is similar, and we have the same conclusion.
\end{proof}

\section{Proof of Theorem \ref{prop norme Frobenius}}\label{proof prop norme Frobenius}
As in Section \ref{proof Prop existence cco}, we give the full proof in the 1D case and a sketch of proof in the 2D case.

In the RO case, the proof is based on calculations in which we carefully detail the structure of the matrix $ \Kk \Kk^T - Id_{MN} $ and identify its constituent with those of $L_{orth}(\Kbf)$. The main lemma describing the structure of $ \Kk \Kk^T - Id_{MN} $ is  Lemma \ref{simplif Kk Kk^T - Id_MN}. It is deduced from Lemma \ref{S_N C(x) S_N^T = C(S_Nx)} and Lemma \ref{S_N C C S_N^T = C Q_N conv} which focus on submatrices of $ \Kk \Kk^T - Id_{MN} $.

The result in the CO case is obtained from the result in the RO case and a known relation between $ \|\Kk \Kk^T - Id_{MN}\|_F^2 $ and 
$\|\Kk^T \Kk - Id_{CSN}\|_F^2 $, see for instance Lemma 1 in \citet{wang2020orthogonal}.

\subsection{Proof of Theorem \ref{prop norme Frobenius}, in the 1D Case}
Before proving Theorem \ref{prop norme Frobenius}, we first present three intermediate lemmas.
\begin{lemma}\label{S_N C(x) S_N^T = C(S_Nx)}
    Let $x \in \mathbb{R}^{SN}$.
    We have 
    $$ S_N C(x) S_N^T = C(S_N x) \;. $$
\end{lemma}
\begin{proof}
Let $x \in \mathbb{R}^{SN}$, $X = C(x)$ and
$Y = S_N X S_N^T \in \mathbb{R}^{N \times N}$.
The matrix $Y$ is formed by sampling $X$, i.e., for all $m,n \in \llbracket0,N-1\rrbracket$,
$$ Y_{m,n} = X_{Sm,Sn} .$$
Hence, using \eqref{cond mat circulante}, $Y_{m,n} = x_{(Sm-Sn)\%SN } = x_{S((m-n)\%N)}$.
Setting $y = S_N x$, we have $y_l = x_{Sl}$ for all $l \in \llbracket0,N-1\rrbracket$.
Therefore,
$Y_{m,n} = y_{(m-n)\%N}$, and using \eqref{cond mat circulante}, we obtain $Y=C(y)$.
Hence, from the definitions of $Y$, $X$ and $y$ we conclude that 
\begin{align*}
    S_N C(x) S_N^T = C(S_N x) \;.
\end{align*}
This completes the proof of the lemma.
\end{proof}

For $N$ such that $SN \geq {2k-1}$, and $P = \left\lfloor \frac{k-1}{S} \right\rfloor S$, we introduce the operator $Q_{S,N}$ which associates to a vector $x = (x_0, \ldots , x_{2 \frac{P}{S}})^T \in \mathbb{R}^{2 \frac{P}{S}+1}$, the vector
\begin{align}\label{def Q_N}
    Q_{S,N}(x) = (x_{\frac{P}{S}}, \ldots, x_{2 \frac{P}{S}}, 0,\ldots, 0, x_0, x_1,  \ldots, x_{\frac{P}{S}-1})^T \in \mathbb{R}^N .
\end{align}

\begin{lemma}\label{S_N C C S_N^T = C Q_N conv}
    Let $S$, $k=2r+1$ and $N$ be positive integers such that $SN \geq {2k-1}$.
    Let $h,g \in \mathbb{R}^k$ and $P = \left\lfloor \frac{k-1}{S} \right\rfloor S$, we have
    \begin{align}\label{eq lemma S_N C C S_N^T = C Q_N conv}
        S_N C(P_{SN}(h)) C(P_{SN}(g))^T S_N^T 
        &= C(Q_{S,N}(\conv(h,g,\text{padding zero} = P, \text{stride} = S))) \;.
    \end{align}
\end{lemma}
\begin{proof}
Let $N$ be such that $SN \geq {2k-1}$, and $P = \left\lfloor \frac{k-1}{S} \right\rfloor S$.
Let us first detail and analyse the left-hand side of \eqref{eq lemma S_N C C S_N^T = C Q_N conv}.
Recall that by definition $P_{SN}(h)$ is $SN$-periodic: $[P_{SN}(h)]_i = [P_{SN}(h)]_{i\%SN}$ for all $i \in \mathbb{Z}$.
Using \eqref{transpose circulant matrix}, \eqref{product circulant matrices}, and \eqref{conv circulaire formule}, we have
\begin{align*}
    C(P_{SN}(h)) C(P_{SN}(g))^T &= C(P_{SN}(h)) C(\widetilde{P_{SN}(g)}) \\
    &= C\left(\left(\sum_{i=0}^{SN-1} {[P_{SN}(h)]}_i {\left[ \widetilde{P_{SN}(g)}\right]}_{j-i}\right)_{j=0..SN-1}\right) \\
    &= C\left(\left(\sum_{i=0}^{SN-1} {[P_{SN}(h)]}_{i} {[P_{SN}(g)]}_{i-j}\right)_{j=0..SN-1}\right) \;.
\end{align*}
Setting  $b^{(SN)}[{h},{g}] = \left(\sum_{i=0}^{SN-1} {[P_{SN}(h)]}_{i} {[P_{SN}(g)]}_{i-j}\right)_{j=0..SN-1} \;,$
we have
\begin{align}\label{eq verifie par c(h,g)}
    C(P_{SN}(h)) C(P_{SN}(g))^T 
    &= C(b^{(SN)}[{h},{g}]) \;. 
\end{align}
To simplify the forthcoming notation, we temporarily denote by 
\begin{align}\label{b = b sn h g}
    b:= b^{(SN)}[{h},{g}].    
\end{align}
Notice that by definition, $b$ is $SN$-periodic.
Therefore, we can restrict its study to an interval of size $ SN $.
We consider $j \in \llbracket-2r,SN-2r-1 \rrbracket$.
From the definition of $P_{SN}$ in \eqref{eq def P_n}, we have, for $i \in \llbracket -r, -r+SN-1 \rrbracket$,
\begin{align}\label{P_SN (h) _i}
[P_{SN}(h)]_i = \left\{ \begin{array}{l l l}
h_{r-i} & \text{ if } & i \in \llbracket-r,r \rrbracket \\
0 & \text{ if } & i \in \llbracket r+1,-r+SN-1 \rrbracket \;.
\end{array}\right.
\end{align}
Hence, since $P_{SN}(h)$ and $P_{SN}(g)$ are periodic, we have 
\begin{align}
    b_j &= \sum_{i=0}^{SN-1} {[P_{SN}(h)]}_{i} {[P_{SN}(g)]}_{i-j} \nonumber \\
    &= \sum_{i=-r}^{SN-1-r} {[P_{SN}(h)]}_{i} {[P_{SN}(g)]}_{i-j} \nonumber \\
    &= \sum_{i=-r}^{r} {[P_{SN}(h)]}_{i} {[P_{SN}(g)]}_{i-j}. \label{c_j intermediaire}
\end{align}
The set of indices $i \in \llbracket -r,r \rrbracket$ such that ${[P_{SN}(h)]}_{i} {[P_{SN}(g)]}_{i-j} \neq 0$ is included in $\llbracket-r,r \rrbracket \cap \{i | (i-j)\%SN \in \llbracket-r,r \rrbracket \%SN \}$. \\
Since $j \in \llbracket -2r, SN-2r-1 \rrbracket$:
We have $-r \leq i \leq r$ and $-2r \leq j \leq SN-2r-1$, then $-SN+r+1 \leq i-j \leq 3r$, but by hypothesis, $SN \geq 2k-1 = 4r+1$, hence $3r < SN-r$ and so $-SN+r < i-j < SN-r$.
Therefore, for $i \in \llbracket-r,r \rrbracket$ and $j \in \llbracket-2r,SN-2r-1 \rrbracket$
\begin{align*}
    (i-j)\%SN \in (\llbracket-r,r \rrbracket \% SN) \iff i-j \in \llbracket-r,r \rrbracket \iff i \in \llbracket -r+j,r+j \rrbracket .
\end{align*}
As a conclusion, for $j \in \llbracket -2r,SN-2r-1 \rrbracket$,
\begin{align}\label{(eq 1 )}
    \left\{i \in \llbracket-r,r \rrbracket \quad | \quad  {[P_{SN}(h)]}_{i} {[P_{SN}(g)]}_{i-j} \neq 0 \right\} \subset \llbracket-r,r \rrbracket \cap \llbracket -r+j, r+j \rrbracket.
\end{align}
Let us now analyse the right-side of \eqref{eq lemma S_N C C S_N^T = C Q_N conv}.
We start by considering zero-padding $= k-1$ and stride $= 1$, and we will arrive to the formula with zero-padding $= P$ and stride $= S$ later.
Using \eqref{conv in R^2P+1}, we denote by
\begin{align}\label{def a}
a = \conv(h,g,\text{padding zero} = k-1, \text{stride} = 1) \in \RR^{2k-1} .
\end{align}
We have from \eqref{conv h g P 1 formule}, for $j \in \llbracket0,2k-2 \rrbracket$,
\begin{align*}
    a_j = \sum_{i=0}^{k-1} h_{i} \Bar{g}_{i+j} \;.
\end{align*}
Using \eqref{def g bar} and keeping the indices $i \in \llbracket0,k-1 \rrbracket$ for which $\Bar{g}_{i+j} \neq 0$, i.e. such that $i+j \in \llbracket k-1,2k-2 \rrbracket$,
we obtain

\begin{align}\label{calcul a_j}
\left\{ \begin{array}{l l l}
         a_j= \sum_{i=k-1-j}^{k-1} h_i g_{i+j-(k-1)} & \text{ if } & j \in \llbracket 0,k-2 \rrbracket \;, \\
         a_j = \sum_{i=0}^{2k-2-j} h_i g_{i+j-(k-1)} & \text{ if } & j \in \llbracket k-1,2k-2 \rrbracket \;.
    \end{array}\right.
\end{align}
In the following, we will connect $b$ with $a$ by distinguishing several cases depending on the value of $ j $. \\
We distinguish $j \in \llbracket0,2r \rrbracket$, $j \in \llbracket-2r,-1 \rrbracket$ and $j \in \llbracket2r+1,-2r+SN-1 \rrbracket$. Recall that $k=2r+1$.

\textbf{If $j \in \llbracket 0,2r \rrbracket$: } then $ \llbracket -r,r \rrbracket \cap \llbracket -r+j,r+j \rrbracket = \llbracket -r+j,r \rrbracket $.
Using \eqref{(eq 1 )} and \eqref{P_SN (h) _i}, the equality \eqref{c_j intermediaire} becomes
\begin{align}
    b_j 
    = \sum_{i=-r+j}^{r} {[P_{SN}(h)]}_{i} {[P_{SN}(g)]}_{i-j} \nonumber 
    = \sum_{i=-r+j}^{r} h_{r-i} g_{r-i+j} \nonumber \;.
\end{align}
By changing the variable $ l = r-i $, and using $k=2r+1$, we find
\begin{align}
    b_j = \sum_{l=0}^{2r-j} h_{l} g_{l+j} \nonumber 
    = \sum_{l=0}^{k-1-j} h_{l} g_{l+j} \nonumber 
    = \sum_{l=0}^{2k-2-(k-1+j)} h_{l} g_{l+(k-1+j)-(k-1)} \;. \nonumber
\end{align}
When $j \in \llbracket0,2r \rrbracket = \llbracket0,k-1 \rrbracket$, we have $k-1+j \in \llbracket k-1,2k-2 \rrbracket$, therefore using \eqref{calcul a_j}, we obtain
\begin{align}
    b_j &= a_{k-1+j} \label{c_j 1er cas} \;.
\end{align}

\textbf{If $j \in \llbracket -2r,-1 \rrbracket$: } then $\llbracket -r,r \rrbracket \cap \llbracket -r+j,r+j \rrbracket = \llbracket -r,r+j \rrbracket $.
Using \eqref{(eq 1 )} and \eqref{P_SN (h) _i}, the equality \eqref{c_j intermediaire} becomes
\begin{align}
    b_j
    = \sum_{i=-r}^{r+j} {[P_{SN}(h)]}_{i} {[P_{SN}(g)]}_{i-j} \nonumber 
    = \sum_{i=-r}^{r+j} h_{r-i} g_{r-i+j} \;. \nonumber
\end{align}
By changing the variable $ l = r-i $, and using $k=2r+1$, we find
\begin{align*}
    b_j = \sum_{l=-j}^{2r} h_{l} g_{l+j} 
    = \sum_{l=-j}^{k-1} h_{l} g_{l+j} 
    = \sum_{l=k-1-(k-1+j)}^{k-1} h_{l} g_{l+(k-1+j)-(k-1)} \;.
\end{align*}
When $j \in \llbracket-2r,-1 \rrbracket = \llbracket-(k-1),-1 \rrbracket$, we have $k-1+j \in \llbracket 0,k-2 \rrbracket$, and using \eqref{calcul a_j}, we obtain
\begin{align}
    b_j 
    &= a_{k-1+j} \label{c_j 2e cas} \;.
\end{align}

\textbf{If $j \in \llbracket 2r+1,SN-2r-1 \rrbracket$: } then $\llbracket -r,r \rrbracket \cap \llbracket -r+j,r+j \rrbracket = \emptyset $.
The equality \eqref{c_j intermediaire} becomes 
\begin{align}\label{c_j 3e cas}
    b_j = 0.
\end{align}
Therefore, we summarize \eqref{c_j 1er cas}, \eqref{c_j 2e cas} and \eqref{c_j 3e cas}: For all $j \in \llbracket-(k-1),-(k-1)+SN-1 \rrbracket$,
\begin{align}\label{formule b final}
    b_j = \left\{
    \begin{array}{l l}
         a_{k-1+j} & \text{ if } j \in \llbracket -(k-1),k-1 \rrbracket, \\
         0 & \text{ if } j \in \llbracket k,SN-k \rrbracket .
    \end{array}
    \right.
\end{align}

Let us now introduce 'padding zero $= P$' and 'stride $= S$'. 
We will prove the equality between matrices in \eqref{eq lemma S_N C C S_N^T = C Q_N conv} using the equality between vectors in \eqref{formule b final}.

Recall that $P = \left\lfloor \frac{k-1}{S} \right\rfloor S \leq k-1$, and let $i \in \llbracket0,2P \rrbracket$.
Therefore $i-P \in \llbracket -P,P \rrbracket \subset \llbracket -(k-1),k-1 \rrbracket$, hence using \eqref{def a}, \eqref{lien conv P et P'} and
\eqref{formule b final},
we have
\begin{align*}
    \left[\conv(h,g,\text{padding zero} = P, \text{stride} = 1)\right]_i 
    = a_{k-1+i-P} 
    = b_{i-P} \;.
\end{align*}
Therefore, using \eqref{def conv S} and $ \left\lfloor 2P/S \right\rfloor + 1 = 2P/S+1$
\begin{align*} %\label{conv stridé}
    & \conv(h,g,\text{padding zero} = P, \text{stride} = S)  \\
    &= 
    \left(
    b_{-\left\lfloor \frac{k-1}{S} \right\rfloor S} , \ldots , b_{-2S} , b_{-S} , b_0 , b_S , b_{2S} , \ldots , b_{\left\lfloor \frac{k-1}{S} \right\rfloor S}
    \right)^T \in \RR^{2 P/S +1 }  \;.
\end{align*}
Using the definition of $Q_{S,N}$ in \eqref{def Q_N}, we obtain 
\begin{align*}
    & Q_{S,N}(\conv(h,g,\text{padding zero} = P, \text{stride} = S)) \\
    &= \left(
    b_0 , b_S , b_{2S} , \ldots , b_{\left\lfloor \frac{k-1}{S} \right\rfloor S} ,
    0 , \ldots , 0 ,
    b_{-\left\lfloor \frac{k-1}{S} \right\rfloor S} , \ldots , b_{-2S} , b_{-S}
    \right)^T \in \RR^N \;.
\end{align*}
But, using \eqref{c_j 3e cas}, and since $\left\lfloor \frac{k-1}{S} \right\rfloor S$ is the largest multiple of $S$ less than or equal to $k-1$ and $b$ is $SN$-periodic, we have  
\begin{align*}
    S_N b &= \left(
    b_0 , b_S , b_{2S} , \ldots , b_{\left\lfloor \frac{k-1}{S} \right\rfloor S} ,
    0 , \ldots , 0 ,
    b_{SN-\left\lfloor \frac{k-1}{S} \right\rfloor S} , \ldots , b_{SN-2S} , b_{SN-S}
    \right)^T \\
    & = \left(
    b_0 , b_S , b_{2S} , \ldots , b_{\left\lfloor \frac{k-1}{S} \right\rfloor S} ,
    0 , \ldots , 0 ,
    b_{-\left\lfloor \frac{k-1}{S} \right\rfloor S} , \ldots , b_{-2S} , b_{-S}
    \right)^T \in \RR^N \;.
\end{align*}
Finally, we have
\begin{align*}
    S_N b = Q_{S,N}(\conv(h,g,\text{padding zero} = P, \text{stride} = S)) \;.
\end{align*}
Using \eqref{b = b sn h g}, \eqref{eq verifie par c(h,g)} and Lemma \ref{S_N C(x) S_N^T = C(S_Nx)}, we conclude that
\begin{align*}
    S_N C(P_{SN}(h)) C(P_{SN}(g))^T S_N^T 
    &= S_N C(b^{(SN)}[{h},{g}]) S_N^T \\
    &= C(S_N b^{(SN)}[{h},{g}]) \\
    &= C(Q_{S,N}(\conv(h,g,\text{padding zero} = P, \text{stride} = S))) \;.
\end{align*}
\end{proof}
\begin{lemma}\label{simplif Kk Kk^T - Id_MN}
    Let $M$, $C$, $S$, $k=2r+1$ be positive integers, and let $\noyau \in \RR^{M \times C \times k}$.
    Let $N$ be such that $SN \geq {2k-1}$, and $P = \left\lfloor \frac{k-1}{S} \right\rfloor S$.
    We denote by $z_{P/S} = \begin{bmatrix}
    0_{P/S} \\ 1 \\ 0_{P/S}
    \end{bmatrix} \in \RR^{2P/S+1}$.
    We have 
    \begin{align*}
        \Kk \Kk^T - Id_{MN} &= \left(
        \begin{array}{c c c}
          C(Q_{S,N}(x_{1,1})) & \ldots & C(Q_{S,N}(x_{1,M})) \\
          \vdots & \ddots & \vdots \\
          C(Q_{S,N}(x_{M,1})) &\ldots & C\left(Q_{S,N}(x_{M,M})\right)
        \end{array}
        \right) \;,
    \end{align*}
    where for all $m,l \in \llbracket 1,M \rrbracket$,
    \begin{align}\label{x_m,l RO case}
    x_{m,l} =  \sum_{c=1}^C \conv(\noyau_{m,c},\noyau_{l,c},\text{padding zero} = P, \text{stride} = S) - \delta_{m=l} z_{P/S} \in \RR^{2P/S+1} \;.
    \end{align}
\end{lemma}
\begin{proof}
We have from \eqref{def mathcal K 1D},
\begin{align*}
\Kk = \left(
\begin{array}{c c c}
  S_NC(P_{SN}(\noyau_{1,1})) & \ldots & S_NC(P_{SN}(\noyau_{1,C})) \\
  \vdots & \vdots & \vdots \\
  S_NC(P_{SN}(\noyau_{M,1})) &\ldots & S_NC(P_{SN}(\noyau_{M,C}))
\end{array}
\right) \in \mathbb{R}^{MN \times CSN} \;.
\end{align*}
Hence, we have that the block $(m,l) \in \llbracket1,M\rrbracket^2$ of size $(N,N)$ of $\Kk \Kk^T$ is equal to :
\begin{align*}
      & \left(
\begin{array}{c c c}
  S_NC(P_{SN}(\noyau_{m,1})) & \ldots & S_NC(P_{SN}(\noyau_{m,C}))
\end{array}
\right)
\left(
\begin{array}{c}
  C(P_{SN}(\noyau_{l,1}))^T S_N^T \\ \vdots \\ C(P_{SN}(\noyau_{l,C}))^T S_N^T 
\end{array}
\right) \\
&= 
   \sum_{c=1}^{C} S_N C(P_{SN}(\noyau_{m,c})) C(P_{SN}(\noyau_{l,c}))^T S_N^T 
 \;.
\end{align*}
We denote by $A_{m,l} \in \RR^{N \times N}$ the block $(m,l) \in \llbracket1,M\rrbracket^2$ of size $(N,N)$ of $\Kk \Kk^T-Id_{MN}$.
We want to prove that $A_{m,l} = C(Q_{S,N}(x_{m,l}))$ where $x_{m,l}$ is defined in \eqref{x_m,l RO case}. 
Using \eqref{def bases canoniques}, \eqref{def mat circulante}, and \eqref{def Q_N}, we have $Id_N = C\left(\begin{bmatrix}
1 \\ 0_{N-1}
\end{bmatrix}\right) = C(Q_{S,N}(z_{P/S}))$, and therefore,
\begin{align*}
    A_{m,l} &=  \sum_{c=1}^{C} S_N C(P_{SN}(\noyau_{m,c})) C(P_{SN}(\noyau_{l,c}))^T S_N^T - \delta_{m=l} C(Q_{S,N}(z_{P/S})) \;.
\end{align*}
Using Lemma \ref{S_N C C S_N^T = C Q_N conv}, this becomes
\begin{align*}
    A_{m,l} &= \sum_{c=1}^{C} C(Q_{S,N}(\conv(\noyau_{m,c},\noyau_{l,c},\text{padding zero} = P, \text{stride} = S))) - \delta_{m=l} C(Q_{S,N}(z_{P/S})) \;.
\end{align*}
By linearity of $C$ and $Q_{S,N}$, we obtain
\begin{align*}
    A_{m,l} 
    &= C\left(Q_{S,N}\left( \sum_{c=1}^{C} \conv(\noyau_{m,c},\noyau_{l,c},\text{padding zero} = P, \text{stride} = S) - \delta_{m=l} z_{P/S}\right)\right) \\
    &= C\left(Q_{S,N}(x_{m,l})\right) \;.
\end{align*}
\end{proof}

\begin{proof}[Proof of Theorem \ref{prop norme Frobenius}]
Let $M$, $C$, $S$, $k=2r+1$ be positive integers, and let $\noyau \in \RR^{M \times C \times k}$.
Let $N$ be such that $SN \geq {2k-1}$, and $P = \left\lfloor \frac{k-1}{S} \right\rfloor S$.
For all $m,l \in \llbracket 1,M \rrbracket$, we denote by $A_{m,l} \in \mathbb{R}^{N \times N}$ the block $(m,l)$ of size $(N,N)$ of $\Kk \Kk^T - Id_{MN}$.
Using Lemma \ref{simplif Kk Kk^T - Id_MN}, we have $$A_{m,l} = C\left(Q_{S,N}\left(\sum_{c=1}^C \conv(\noyau_{m,c},\noyau_{l,c},\text{padding zero} = P, \text{stride} = S) - \delta_{m=l} z_{P/S}\right)\right).$$
Hence, from \eqref{def mat circulante} and \eqref{def Q_N}, using the fact that for all $x \in \mathbb{R}^N$, $\|C(x)\|_F^2 = N\|x\|_2^2$, and for all $x \in \mathbb{R}^{2P/S+1}$, $\|Q_{S,N}(x)\|_2^2 = \|x\|_2^2 $, we have
\begin{align*}
    & \|\Kk \Kk^T - Id_{MN}\|_F^2 \\ &= \sum_{m=1}^M \sum_{l=1}^M \|A_{m,l}\|_F^2 \\
    &= \sum_{m=1}^M \sum_{l=1}^M {\left\| C\left(Q_{S,N}\left(\sum_{c=1}^{C} \conv(\noyau_{m,c},\noyau_{l,c},\text{padding zero} = P, \text{stride} = S) - \delta_{m=l} z_{P/S}\right)\right) \right\|}_F^2 \\
    &= \sum_{m=1}^M \sum_{l=1}^M N {\left\| Q_{S,N}\left(\sum_{c=1}^{C} \conv(\noyau_{m,c},\noyau_{l,c},\text{padding zero} = P, \text{stride} = S) - \delta_{m=l} z_{P/S}\right) \right\|}_2^2 \\
    &= N \sum_{m=1}^M \sum_{l=1}^M {\left\| \sum_{c=1}^{C} \conv( \noyau_{m,c},\noyau_{l,c},\text{padding zero} = P, \text{stride} = S) - \delta_{m=l} z_{P/S} \right\|}_2^2 \;.
\end{align*}
Therefore, using \eqref{I_r0 bis} and \eqref{def CONV}, we obtain for any $M$, $C$, $S$, $k=2r+1$ and $\noyau \in \RR^{M \times C \times k}$,
\begin{align}\label{err frob projective}
    \|\Kk \Kk^T - Id_{MN}\|_F^2 &= N \| \CONV(\noyau,\noyau,\text{padding zero} = P, \text{stride} = S) - \Iro \|_F^2 \;.
\end{align}
This concludes the proof in the RO case.\\
In order to prove the theorem in the CO case we use Lemma 1 in \citet{wang2020orthogonal}.
This lemma states that
$$\|\Kk^T \Kk - Id_{CSN}\|_F^2 = \|\Kk \Kk^T - Id_{MN}\|_F^2 +CSN-MN.$$
Therefore, using that \eqref{err frob projective} holds for all $M$, $C$ and $S$, we have
\begin{align}\label{err frob embedding}
\|\Kk^T \Kk - Id_{CSN}\|_F^2 &= N \left(\| \CONV(\noyau,\noyau,\text{padding zero} = P, \text{stride} = S) - \Iro \|_F^2 -(M-CS)\right) 
\end{align}
Hence, using the definitions of $\ERR{F}{N}$ and $L_{orth}$ in Sections \ref{sec 1D L_orth} and \ref{sec 1D errors}, \eqref{err frob projective} and \eqref{err frob embedding} lead to 
$$ \left(\ERR{F}{N}(\noyau)\right)^2 = N L_{orth}(\noyau) .$$
This concludes the proof of Theorem \ref{prop norme Frobenius} in the 1D case.
\end{proof}

\subsection{Sketch of the Proof of Theorem \ref{prop norme Frobenius}, in the 2D Case}
We start by stating intermediate lemmas.
First we introduce a slight abuse of notation, for a vector $x \in \RR^{N^2}$, we denote by $\Cc(x) = \Cc(X) $, where $X \in \RR^{N \times N}$ such that $\Vect(X) = x$.
The main steps of the proof in the 2D case follow those in the 1D case and are given below.
\begin{lemma}
    Let $X \in \mathbb{R}^{SN \times SN}$.
    We have 
    $$ \Ss_N \Cc(X) \Ss_N^T = \Cc(\Ss_N \Vect(X)) . $$
\end{lemma}
Let $\Qq_{S,N}$ be the operator which associates to a matrix $ x \in \RR^{(2P/S+1) \times (2P/S+1)} $ the matrix
\begin{align*}
    \left(\begin{array}{c c c c c c c c c}
     x_{P/S,P/S} & \cdots & x_{P/S,2P/S} & 0 & \cdots & 0 & x_{P/S,0} & \cdots & x_{P/S,P/S-1}  \\
     \vdots & \vdots & \vdots & \vdots & \vdots & \vdots & \vdots & \vdots & \vdots \\
     x_{2P/S,P/S} & \cdots & x_{2P/S,2P/S} & 0 & \cdots & 0 & x_{2P/S,0} & \cdots & x_{2P/S,P/S-1} \\
     0 & \dots & 0 & 0 & \dots & 0 & 0 & \dots & 0 \\
     \vdots & \vdots & \vdots & \vdots & \vdots & \vdots & \vdots & \vdots & \vdots \\
     0 & \dots & 0 & 0 & \dots & 0 & 0 & \dots & 0 \\
     x_{0,P/S} & \cdots & x_{0,2P/S} & 0 & \cdots & 0 & x_{0,0} & \cdots & x_{0,P/S-1}  \\
     \vdots & \vdots & \vdots & \vdots & \vdots & \vdots & \vdots & \vdots & \vdots \\
     x_{P/S-1,P/S} & \cdots & x_{P/S-1,2P/S} & 0 & \cdots & 0 & x_{P/S-1,0} & \cdots & x_{P/S-1,P/S-1}  
\end{array} \right) \in \mathbb{R}^{N \times N} \;.
\end{align*}
\begin{lemma}
    Let $N$ be such that $SN \geq {2k-1}$, $h,g \in \mathbb{R}^{k \times k}$ and $P = \left\lfloor \frac{k-1}{S} \right\rfloor S$, we have
    \begin{align*}
        \Ss_N \Cc(\Pp_{SN}(h)) \Cc(\Pp_{SN}(g))^T \Ss_N^T 
        &= \Cc(\Qq_{S,N}(\conv(h,g,\text{padding zero} = P, \text{stride} = S))) \;.
    \end{align*}
\end{lemma}
\begin{lemma}\label{simplif Kk Kk^T - Id_MN 2D}
    Let $M$, $C$, $S$, $k=2r+1$ be positive integers, and let $\noyau \in \RR^{M \times C \times k \times k}$.
    Let $N$ be such that $SN \geq {2k-1}$, and $P = \left\lfloor \frac{k-1}{S} \right\rfloor S$.
    We set $z_{P/S,P/S} \in \RR^{(2P/S+1) \times (2P/S+1)}$ such that for all $i,j \in \llbracket0,2P/S \rrbracket$, ${[z_{P/S,P/S}]}_{i,j} = \delta_{i=P/S} \delta_{j=P/S}$.
    We have 
    \begin{align*}
        \Kk \Kk^T - Id_{MN^2} &= \left(
        \begin{array}{c c c}
          \Cc(\Qq_{S,N}(x_{1,1})) & \ldots & \Cc(\Qq_{S,N}(x_{1,M})) \\
          \vdots & \ddots & \vdots \\
          \Cc(\Qq_{S,N}(x_{M,1})) &\ldots & \Cc(\Qq_{S,N}(x_{M,M}))
        \end{array}
        \right) \;,
    \end{align*}
    where for all $m,l \in \llbracket 1,M \rrbracket$, $$x_{m,l} = \sum_{c=1}^C \conv(\noyau_{m,c},\noyau_{l,c},\text{padding zero} = P, \text{stride} = S) - \delta_{m=l} z_{P/S,P/S}.$$
\end{lemma}
Then we proceed as in the 1D case.

\section{Proof of Theorem \ref{Prop norme spectrale}}\label{proof Prop norme spectrale}

As in Section \ref{proof Prop existence cco} and Section \ref{proof prop norme Frobenius}, we give the full proof in the 1D case and a sketch of proof in the 2D case.

The lower-bound is a consequence of Theorem \ref{prop norme Frobenius}.

The proof of the upper-bound is different in the RO case and the CO case. In the RO case, we first express the orthogonality residual using Lemma \ref{simplif Kk Kk^T - Id_MN}. Then we conclude with calculations based on matrix norm inequalities, properties of circulant matrices and the definition of $L_{orth}(\Kbf)$.

The CO case is more difficult since, to use in place of Lemma \ref{simplif Kk Kk^T - Id_MN}, we first need to establish Lemma \ref{Ch^T S_n^T S_n Cg}. We then proceed with calculations to express that $\|\Kk^T\Kk - Id_{CSN}\|_2$ is upper-bounded by a quantity independent of $N$, as long as $N\geq 2k-1$. Then, after calculations, a key argument is to apply Theorem \ref{prop norme Frobenius} to the matrix $\Kk$ obtained for the signal size $N' = 2k-1$.

\subsection{Proof of Theorem \ref{Prop norme spectrale}, in the 1D Case}

The lower-bound of Theorem \ref{Prop norme spectrale} is an immediate consequence of \eqref{norm_equiv} and Theorem \ref{prop norme Frobenius}. We have indeed both in the RO and CO case:
\[(\ERR{s}{N}(\noyau))^2 \geq \frac{1}{\min(M,CS^2) N^2} (\ERR{F}{N}(\noyau))^2 = \frac{1}{\min(M,CS^2)} L_{orth}(\noyau).
\]

We focus from now on on the upper-bound. Let $M$, $C$, $S$, $k=2r+1$ be positive integers, and let $\noyau \in \RR^{M \times C \times k}$.
Let $N$ be such that $SN \geq {2k-1}$, and $P = \left\lfloor \frac{k-1}{S} \right\rfloor S$.
We denote by $z_{P/S} = \begin{bmatrix}
0_{P/S} \\ 1 \\ 0_{P/S}
\end{bmatrix} \in \RR^{2P/S+1}$.

\textbf{RO case ($M \leq CS$): }
From Lemma \ref{simplif Kk Kk^T - Id_MN}, we have
\begin{align}\label{kk kk^T - I_MN RO case}
    \Kk \Kk^T - Id_{MN}
&= \left(
\begin{array}{c c c}
  C(Q_{S,N}(x_{1,1})) & \ldots & C(Q_{S,N}(x_{1,M})) \\
  \vdots & \ddots & \vdots \\
  C(Q_{S,N}(x_{M,1})) &\ldots & C(Q_{S,N}(x_{M,M}))
\end{array}
\right) \;,
\end{align}
where for all $m,l \in \llbracket 1,M \rrbracket$, 
\begin{align}\label{x_m,l RO}
    x_{m,l} =  \sum_{c=1}^C \conv(\noyau_{m,c},\noyau_{l,c},\text{padding zero} = P, \text{stride} = S) - \delta_{m=l} z_{P/S} \in \RR^{2P/S+1} \;.
\end{align}
We set
\begin{align*}
    B = \Kk \Kk^T - Id_{MN} \;.
\end{align*}
Since $ B $ is symmetric and due to the well-known properties of matrix norms, we have
$\|B\|_1 = \|B\|_{\infty}$ and $\|B\|_2^2 \leq \|B\|_1 \|B\|_{\infty}$.
Hence, using the definition of $\|B\|_1$, we have
\begin{align*}
    \|B\|_2^2 \leq \|B\|_1 \|B\|_{\infty} 
     = \|B\|_1^2 
     = \left(\max_{1 \leq l \leq MN} \sum_{m=1}^{MN} |B_{m,l}|\right)^2 \;.
\end{align*}
Using \eqref{kk kk^T - I_MN RO case}, and \eqref{def mat circulante}, we obtain
\begin{align*}
    \|B\|_2^2 & \leq \max_{1 \leq l \leq M} \left(\sum_{m=1}^{M} \|Q_{S,N}(x_{m,l})\|_1 \right)^2 \;.
\end{align*}
Given the definition of $Q_{S,N}$ in \eqref{def Q_N}, we have for all $x \in \RR^{2P/S+1}$, $\|Q_{S,N}(x)\|_1 = \|x\|_1$, therefore,
\begin{align*}
    \|B\|_2^2 & \leq \max_{1 \leq l \leq M} \left(\sum_{m=1}^{M} \|x_{m,l}\|_1 \right)^2 \;.
\end{align*}
We set $l_0 \in \arg\max_{1 \leq l \leq M} \left(\sum_{m=1}^{M} \|x_{m,l}\|_1 \right)^2$.
Using that for all $x \in \RR^n$, $\|x\|_1 \leq \sqrt{n}\|x\|_2$, we have
\begin{align*}
    \|B\|_2^2  \leq \left(\sum_{m=1}^{M} \|x_{m,l_0}\|_1 \right)^2 
     \leq (2P/S+1)\left(\sum_{m=1}^{M} {\|x_{m,l_0}\|}_2 \right)^2 \;.
\end{align*}
Using Cauchy-Schwarz inequality, we obtain
\begin{align*}
    \|B\|_2^2  \leq (2P/S+1)M \sum_{m=1}^{M} {\|x_{m,l_0}\|}_2^2 
     \leq (2P/S+1)M \sum_{m=1}^{M} \sum_{l=1}^{M} {\|x_{m,l}\|}_2^2 \;.
\end{align*}
Using \eqref{x_m,l RO}, then \eqref{I_r0 bis} and \eqref{def CONV}, we obtain
\begin{align*}
    & \|B\|_2^2 \\ & \leq (2P/S+1)M \sum_{m=1}^{M} \sum_{l=1}^{M} \left\|\sum_{c=1}^{C} \conv(\noyau_{m,c},\noyau_{l,c},\text{padding zero} = P, \text{stride} = S) - \delta_{m=l} z_{P/S} \right\|_2^2 \\
    & = (2P/S+1)M \sum_{m=1}^{M} \sum_{l=1}^{M} \left\|\left[ \CONV(\noyau,\noyau,\text{padding zero} = P, \text{stride} = S) - \Iro\right]_{m,l,:} \right\|_2^2 \\
    & = (2P/S+1)M \left\| \CONV(\noyau,\noyau,\text{padding zero} = P, \text{stride} = S) - \Iro \right\|_F^2 \\
    &= (2P/S+1)M L_{orth}(\noyau) \;.
\end{align*}
This proves the inequality in the RO case.\\

\textbf{CO case ($M \geq CS$): }
First, for $n \geq 2k-1$, let $R_n$ be the operator that associates to $x \in \RR^{2k-1}$, the vector
\begin{align}\label{def R_n}
    R_n(x) = (x_{k-1}, \ldots, x_{2k-2}, 0 , \ldots, 0 , x_{0}, \ldots , x_{k-2})^T \in \RR^{n} \;.
\end{align}
Note that, when $S'=1$, $N'=SN$, we have in \eqref{def Q_N}, $P'=k-1$ and 
\begin{align}\label{R_n = Q_1,n}
    Q_{1,SN} = R_{SN}.
\end{align}
Recall from \eqref{def bases canoniques} that $(f_i)_{i=0..SN-1}$ is the canonical basis of $\RR^{SN}$.
Let 
$\Lambda_j = C(f_j) \in \mathbb{R}^{SN \times SN}$
be the permutation matrix which shifts down (cyclically) any vector by $j \in \llbracket 0,SN-1 \rrbracket$: for all $x \in \RR^{SN}$, for $i \in \llbracket0,SN-1 \rrbracket$, $(\Lambda_j x)_i = x_{(i-j)\%SN}$.
Note that, using \eqref{def mat circulante}, we have for all $x \in \RR^{SN}$, 
\begin{align}\label{[C(x)]_:,j}
    [C(x)]_{:,j} = \Lambda_j x.
\end{align}
Recall that $k=2r+1$, and for all $h \in \RR^k$,
\begin{align*}
    P_{SN}(x) = (h_r,\ldots,h_0,0,\ldots,0,h_{2r},\ldots,h_{r+1})^T \in \RR^{SN} .
\end{align*}
For $j \in \llbracket0,SN-1\rrbracket$, for $x \in \RR^{k}$, we denote by 
\begin{align}\label{def P_SN^j}
    P_{SN}^{(j)}(x)= \Lambda_j P_{SN}(x)
\end{align}
and for $x \in \RR^{2k-1}$, we denote by
\begin{align}\label{def R_SN^j}
    R_{SN}^{(j)}(x)= \Lambda_j R_{SN}(x).
\end{align}

By assumption $SN \geq 2k-1$, hence $R_{SN}(x)$ is well-defined and
we have for all $j \in \llbracket0,SN-1 \rrbracket$, for all $x \in \RR^{2k-1}$,
\begin{align}\label{norme R_SN j (x) = norme x}
\begin{cases}
     \|R_{SN}^{(j)}(x)\|_1 = \|x\|_1, \\ \|R_{SN}^{(j)}(x)\|_2 = \|x\|_2.
\end{cases}
\end{align}
We first start by introducing the following Lemma.
\begin{lemma}\label{Ch^T S_n^T S_n Cg}
    Let $h,g \in \RR^k$.
    There exist $S$ vectors $x_0,\ldots,x_{S-1} \in \RR^{2k-1}$ such that for all $N$ satisfying $SN \geq 2k-1$, we have for all $j \in \llbracket0,SN-1 \rrbracket$,
    \begin{align*}
        \left[C(P_{SN}(h))^T S_N^T S_N C(P_{SN}(g))\right]_{:,j} = R_{SN}^{(j)}(x_{j\%S}) \;.
    \end{align*}
\end{lemma}
\begin{proof}
Recall that from \eqref{def S_N} and \eqref{calcul S_N^T S_N}, we have $S_N = \sum_{i=0}^{N-1} E_{i,Si}$ and
$A_N :=S_N^T S_N = \sum_{i=0}^{N-1} F_{Si,Si} $.
When applied to a vector $x \in \RR^{SN}$, $A_N$ keeps unchanged the components of $x$ whose indices are multiples of $S$, while the other components of $A_Nx$ are equal to zero.
We know from \eqref{[C(x)]_:,j} and \eqref{def P_SN^j} that, for $j \in \llbracket 0,SN-1 \rrbracket$, the $j$-th column of $C(P_{SN}(g))$ is equal to $P_{SN}^{(j)}(g)$.
Therefore, when applying $A_N$, this becomes $A_N P_{SN}^{(j)}(g) = P_{SN}^{(j)}\left(g^{j}\right)$, where $g^{j} \in \RR^k$ is formed from $g$ by putting zeroes in the place of the elements that have been replaced by 0 when applying $A_N$.
But since $A_N$ preserves the component whose index is a multiple of $S$, we have that the $j$-th column of $A_N C(P_{SN}(g))$ has the same elements as its $j\%S$-th column, shifted down by $(j-j\%S)$ indices.
More precisely, $A_N P_{SN}^{(j)}(g) = \Lambda_{j-j\%S} A_N P_{SN}^{(j\%S)}(g) $, hence $P_{SN}^{(j)}\left(g^{j}\right) = \Lambda_{j-j\%S} P_{SN}^{(j\%S)}\left(g^{j\%S}\right) = P_{SN}^{(j)}\left(g^{j\%S}\right)$.
This implies that $g^{j} = g^{j\%S}$.
Note that, using \eqref{eq def P_n}, we can also derive the exact formula of $g^{j}$, in fact for all $i \in \llbracket0,2r \rrbracket$,
\begin{align*}
    \left[g^{j}\right]_i = \left\{ 
    \begin{array}{l l}
         g_i & \text{ if } (i-r-j)\%S = 0, \\
         0 & \text{ otherwise.} 
    \end{array}
    \right.
\end{align*}
We again can see that $g^{j} = g^{j\%S}$.
Therefore, using \eqref{[C(x)]_:,j} and \eqref{def P_SN^j}, we have $$ A_N [C(P_{SN}(g))]_{:,j} = A_N P_{SN}^{(j)}(g) = P_{SN}^{(j)}\left(g^{j}\right) = P_{SN}^{(j)}\left(g^{j\%S}\right) = \left[C\left(P_{SN}\left(g^{j\%S}\right)\right)\right]_{:,j}.$$
Therefore, we have, for all $j \in \llbracket 0,SN-1 \rrbracket$,
\begin{align*}
[C(P_{SN}(h))^T A_N C(P_{SN}(g))]_{:,j} 
    =
\left[C(P_{SN}(h))^T  C(P_{SN}(g^{j\%S}))\right]_{:,j} \;.
\end{align*}
Using the fact that the transpose of a circulant matrix is a circulant matrix and that two circulant matrices commute with each other (see Equation \ref{transpose circulant matrix} and Equation \ref{commutativity circulant matrices}), we conclude that the transpose of any circulant matrix commutes with any circulant matrix, therefore
\begin{align*}
[C(P_{SN}(h))^T A_N C(P_{SN}(g))]_{:,j}
&= \left[C(P_{SN}(g^{j\%S}))
C(P_{SN}(h))^T \right]_{:,j} \;.
\end{align*}
Using Lemma \ref{S_N C C S_N^T = C Q_N conv} with $S'=1$ and $N' = SN$, and noting that,when $S'=1$, the sampling matrix $S_{N'}$ is equal to the identity, we have
\begin{align*}
    & C(P_{SN}(g^{j\%S})) C(P_{SN}(h))^T \\
    &= Id_{N'} C(P_{N'}(g^{j\%S})) C(P_{N'}(h))^T Id_{N'}^T \\
    &= C(Q_{S',N'}(\conv(g^{j\%S},h,\text{padding zero} = \left\lfloor \frac{k-1}{S'}\right\rfloor S', \text{stride} = S'))) \\
    &= C(Q_{1,SN}(\conv(g^{j\%S},h,\text{padding zero} = k-1, \text{stride} = 1)))
\end{align*}
To simplify, we denote by $x_{j\%S} = \conv(g^{j\%S},h,\text{padding zero} = k-1, \text{stride} = 1) \in \RR^{2k-1}$.
Using \eqref{R_n = Q_1,n}, we obtain 
\begin{align*}
    C(P_{SN}(g^{j\%S})) C(P_{SN}(h))^T = C(Q_{1,SN}(x_{j\%S})) 
    = C(R_{SN}(x_{j\%S})) \;.
\end{align*}
Using \eqref{[C(x)]_:,j} and \eqref{def R_SN^j}, we obtain
\begin{align*}
[C(P_{SN}(h))^T A_N C(P_{SN}(g))]_{:,j} = 
[C(R_{SN}(x_{j\%S}))]_{:,j}
=
\Lambda_j R_{SN}(x_{j\%S}) 
= R_{SN}^{(j)}(x_{j\%S}).
\end{align*}
Therefore, we have for all $j \in \llbracket0,SN-1 \rrbracket$,
\begin{align*}
    \left[C(P_{SN}(h))^T S_N^T S_N C(P_{SN}(g))\right]_{:,j} = R_{SN}^{(j)}(x_{j\%S}) \;.
\end{align*}
This concludes the proof of the lemma.
\end{proof}
Let us go back to the main proof.\\

Using \eqref{def mathcal K 1D}, we have that the block $(c,c') \in \llbracket1,C\rrbracket^2$ of size $(SN,SN)$ of $\Kk^T \Kk$ is equal to :
\begin{align}
&
\left(
\begin{array}{c c c}
   C(P_{SN}(\noyau_{1,c}))^T S_N^T & \ldots & C(P_{SN}(\noyau_{M,c}))^T S_N^T 
\end{array}
\right)
\left(
\begin{array}{c}
  S_NC(P_{SN}(\noyau_{1,c'}))  \\
  \vdots \\
  S_NC(P_{SN}(\noyau_{M,c'}))
\end{array}
\right) \nonumber \\ 
&= 
   \sum_{m=1}^{M} C(P_{SN}(\noyau_{m,c}))^T S_N^T S_N C(P_{SN}(\noyau_{m,c'})) \;. \label{1 point}
\end{align}
For any $(m,c,c') \in \llbracket 1,M \rrbracket \times \llbracket 1,C \rrbracket^2$, we denote by ${(x_{m,c,c',s})}_{s=0..S-1}$ the $S$ vectors of $\RR^{2k-1}$ obtained when applying Lemma \ref{Ch^T S_n^T S_n Cg} with $h=\noyau_{m,c}$, and $g=\noyau_{m,c'}$.
Hence, we have, for all $j \in \llbracket0,SN-1 \rrbracket$,
\begin{align}\label{2 points}
    [C(P_{SN}(\noyau_{m,c}))^T S_N^T S_N C(P_{SN}(\noyau_{m,c'}))]_{:,j} = R_{SN}^{(j)}(x_{m,c,c',j\%S}) \;.
\end{align}
Let $\overline{f}_{k-1} = \begin{bmatrix}
0_{k-1} \\ 1 \\ 0_{k-1}
\end{bmatrix} \in \RR^{2k-1}$.
For all $s \in \llbracket0,S-1 \rrbracket$, we denote by 
\begin{align}\label{x c c' s}
    x_{c,c',s} =\sum_{m=1}^M x_{m,c,c',s} - \delta_{c=c'} \overline{f}_{k-1} \in \RR^{2k-1}.
\end{align}
Note that, from \eqref{def bases canoniques}, \eqref{def R_n}, and \eqref{def R_SN^j}, we have for all $j \in \llbracket0,SN-1 \rrbracket$, 
$f_j = R_{SN}^{(j)}(\overline{f}_{k-1})$.
Therefore, $Id_{SN} = (f_0,\ldots,f_{SN-1}) = \left( R_{SN}^{(0)}(\overline{f}_{k-1}), \ldots, R_{SN}^{(SN-1)}(\overline{f}_{k-1}) \right)$.
We set
\begin{align*}
    B_N = \Kk^T\Kk - Id_{CSN} \;.
\end{align*}
We denote by $A_{c,c'}^N \in \RR^{SN \times SN}$ the block $(c,c') \in \llbracket1,C \rrbracket^2$ of size $(SN,SN)$ of $B_N$.
Using \eqref{1 point}, \eqref{2 points}, and \eqref{x c c' s}, we have, for all $j \in \llbracket0,SN-1 \rrbracket$,
\begin{align}
    {\left[A_{c,c'}^N\right]}_{:,j}
    &= \left[\sum_{m=1}^{M} C(P_{SN}(\noyau_{m,c}))^T S_N^T S_N C(P_{SN}(\noyau_{m,c'})) - \delta_{c=c'}Id_{SN}\right]_{:,j} \nonumber \\
    &= \sum_{m=1}^{M} R_{SN}^{(j)}(x_{m,c,c',j\%S}) - \delta_{c=c'} R_{SN}^{(j)}(\overline{f}_{k-1}) \nonumber \\
    &= R_{SN}^{(j)}(x_{c,c',j\%S}) \label{3 points} \;.
\end{align}
We then proceed in the same way as in the RO case.
Since $B_N$ is clearly symmetric, we have
\begin{align*}
    {\|B_N\|}_2^2  \leq {\|B_N\|}_1 {\|B_N\|}_{\infty} 
     = {\|B_N\|}_1^2
    & = \left(\max_{1 \leq j \leq CSN} \sum_{i=1}^{CSN} |(B_N)_{i,j}|\right)^2 \\
    &= \max_{1 \leq c' \leq C \ , \ 0 \leq j \leq SN-1} \left(\sum_{c=1}^{C} {\|{\left[A_{c,c'}^N\right]}_{:,j}\|}_1 \right)^2 \;.
\end{align*}
Using \eqref{3 points} and \eqref{norme R_SN j (x) = norme x}, this becomes
\begin{align*}
    \|B_N\|_2^2  \leq \max_{\substack{1 \leq c' \leq C \\ 0 \leq j \leq SN-1}} \left(\sum_{c=1}^{C} {\|R_{SN}^{(j)}\left(x_{c,c',j\%S}\right)\|}_1 \right)^2 
     = \max_{\substack{1 \leq c' \leq C \\ 0 \leq s \leq S-1}} \left(\sum_{c=1}^{C} {\|x_{c,c',s}\|}_1 \right)^2 \;.
\end{align*}
We set $(c'_0,s_0) \in \arg\max_{\substack{1 \leq c' \leq C \\ 0 \leq s \leq S-1}} \left(\sum_{c=1}^{C} {\|x_{c,c',s}\|}_1 \right)^2$.
Using that for all $x \in \RR^n$, $\|x\|_1 \leq \sqrt{n} \|x\|_2$, we have
\begin{align*}
    \|B_N\|_2^2  \leq \left(\sum_{c=1}^{C} {\|x_{c,c'_0,s_0} \|}_1 \right)^2 
     \leq (2k-1)\left(\sum_{c=1}^{C} {\|x_{c,c'_0,s_0}\|}_2 \right)^2 \;.
\end{align*}
Using Cauchy-Schwarz inequality, we obtain
\begin{align*}
    \|B_N\|_2^2  \leq (2k-1)C \sum_{c=1}^{C} {\|x_{c,c'_0,s_0}\|}_2^2  
     \leq (2k-1)C \sum_{c=1}^{C} \sum_{c'=1}^{C}
    \sum_{s=0}^{S-1}
    {\|x_{c,c',s}\|}_2^2 \;.
\end{align*}
Using \eqref{norme R_SN j (x) = norme x} in the particular case of $N'=2k-1$, we obtain
\begin{align*}
    \|B_N\|_2^2 & \leq (2k-1)C \sum_{c=1}^{C} \sum_{c'=1}^{C}
    \sum_{s=0}^{S-1}
    {\|R_{S(2k-1)}\left(x_{c,c',s}\right)\|}_2^2  \\
    & = C \sum_{c=1}^{C} \sum_{c'=1}^{C}
    \sum_{s=0}^{S-1} (2k-1)
    {\|R_{S(2k-1)}\left(x_{c,c',s}\right)\|}_2^2 \\
    & = C \sum_{c=1}^{C} \sum_{c'=1}^{C}
    \sum_{j=0}^{S(2k-1)-1}
    {\left\|R_{S(2k-1)}^{(j)}\left(x_{c,c',j\%S}\right)\right\|}_2^2 \;.
\end{align*}
Using \eqref{3 points} for $N'=2k-1$, we obtain
\begin{align*}
    {\|B_N\|}_2^2  \leq C \sum_{c=1}^{C} \sum_{c'=1}^{C}
    \sum_{j=0}^{S(2k-1)-1}
    {\left\|{\left[A_{c,c'}^{2k-1}\right]}_{:,j}\right\|}_2^2 
     = C {\|B_{2k-1}\|}_F^2 \;.
\end{align*}
Using Theorem \ref{prop norme Frobenius}
for $N=2k-1$, we have $\|B_{2k-1}\|_F^2 = (2k-1) L_{orth}(\noyau)$ and we obtain
\begin{align*}
    {\|B_N\|}_2^2 &\leq (2k-1)C L_{orth}(\noyau) \;.
\end{align*}
Therefore, we conclude that, in the CO case 
\[\left(\ERR{s}{N}(\noyau)\right)^2 \leq (2k-1)C L_{orth}(\noyau) \;.
        \]
This concludes the proof in the 1D case.

\subsection{Sketch of the Proof of Theorem \ref{Prop norme spectrale}, for 2D Convolutional Layers}
In the RO case, we proceed as in the 1D case.\\
In the CO case, we first prove a lemma similar to Lemma \ref{Ch^T S_n^T S_n Cg}, then we proceed as in the 1D case.

\section{Proof of Proposition \ref{lemme lipchitz norme spectrale}}\label{proof lemme lipchitz norme spectrale}
Below, we prove Proposition \ref{lemme lipchitz norme spectrale} for a general matrix $A \in \RR^{a \times b}$ with $a \geq b$.
In order to obtain the statement for a convolutional layer $\Kk \in \RR^{MN \times CSN}$:\\
In the RO case ($M \leq CS$): we take $A = \Kk^T$, $a=CSN$, $b=MN$.\\
In the CO case ($M \geq CS$): we take $A = \Kk$, $a=MN$, $b=CSN$.\\

Let $A \in \RR^{a \times b}$ such that $a \geq b$.
We denote by $\varepsilon = \|A^TA-Id_b\|_2$.
Let $x \in \RR^b$, we have
\begin{align*}
    \left|\|Ax\|^2 - \|x\|^2 \right| = \left|x^TA^TAx - x^{T}x \right| = \left|x^T(A^TA-Id_b)x\right| &\leq \|x^T\| \| A^TA-Id_b \|_2 \| x \| \\ &\leq \varepsilon \|x\|^2 \;.
\end{align*}
Hence, for all $x \in \RR^b$,
\begin{align*}
    (1 - \varepsilon)\|x\|^2 \leq \|Ax\|^2 \leq (1 + \varepsilon) \|x\|^2 \;.
\end{align*}
This also implies $\sigma_{max}(A)^2 \leq 1 + \varepsilon$.
But we know that $\sigma_{max}(A^T) = \sigma_{max}(A) $, hence $\sigma_{max}(A^T)^2 \leq 1 + \varepsilon$ and therefore, for all $x \in \RR^a$,
\begin{align*}
    \|A^Tx\|^2 \leq (1 + \varepsilon) \|x\|^2 \;.
\end{align*}

Finally:
\begin{itemize}
    \item In the RO case, for $\varepsilon = \ERR{s}{N}(\noyau) = \|\Kk \Kk^T -Id_{CSN}\|_2 $, $\Kk$ is $\varepsilon$-AIP.
    \item In the CO case, for $\varepsilon = \ERR{s}{N}(\noyau) =  \|\Kk^T \Kk-Id_{MN}\|_2$, $\Kk$ is $\varepsilon$-AIP.
\end{itemize}

\section{Experiment Configurations}
In this section, we describe the details of the experiments conducted on Cifar10 an Imagenette data sets.

\subsection{Cifar10 Experiments}
\label{Cifar_appendix}
The network architecture used for Cifar10 data set is described in Table~\ref{tab:NN_archi_cifar10} (1.1 million parameters). Conv2D layer will depend on the configuration: classical $Conv2D$ for unconstrained reference configuration, $CayleyConv$ for $Cayley$ configuration 
\citep[we use][implementation]{trockman2021orthogonalizing},
%(we use ~\citep{trockman2021orthogonalizing} implementation),
and $L_{orth}$ regularization for $L_{orth}$ configuration (we use our implementation according to Definition~\ref{Lorth-def}). Weight initialization is done according to \textit{Glorot uniform}.

\begin{table}[ht]
  \label{tab:NN_archi_cifar10}
  \centering
  \begin{tabular}{lll}
    \toprule
        Layer     & Parameters $(M,C,k,k)$ & Output size $(M,H,W)$  \\
    \midrule
    Input   & $ $ & $32\times 32\times 3$     \\
    Conv2D, GS2 & $(64,3,3,3)$   &    $64\times 32\times 32$   \\
    Conv2D, GS2  & $(66,64,3,3)$   & $66\times 32\times 32$     \\
    InvDown  & $ $   & $264\times 16\times 16$     \\
    Conv2D, GS2  & $(64,264,3,3)$   & $64\times 16\times 16$     \\
    Conv2D, GS2  & $(128,64,3,3)$   & $128\times 16\times 16$     \\
    Conv2D, GS2  & $(130,128,3,3)$   & $130\times 16\times 16$     \\
    InvDown  & $ $   & $520\times 8\times 8$     \\
    Conv2D, GS2  & $(128,520,3,3)$   & $128\times 8\times 8$     \\
    Conv2D, GS2  & $(192,128,3,3)$   & $192\times 8\times 8$     \\
    Conv2D, GS2  & $(194,192,3,3)$   & $194\times 8\times 8$     \\
    InvDown  & $ $   & $776\times 4\times 4$     \\
    Conv2D, GS2  & $(192,776,3,3)$   & $192\times 4\times 4$     \\
    Flatten, Dense & $ (10,3072) $  & $10$     \\
    \bottomrule
  \end{tabular}
  \caption{Cifar10 Neural network architectures: Conv2D, GS2 is GroupSort2, InvDown is InvertibleDownsampling~\citep{trockman2021orthogonalizing}}
\end{table}

Task loss is the classical cross-entropy (CE). As described in~\citet{bethune2022PayAttention}, 1-lipschitz property of orthogonal neural network may prevent learning  with CE and require introducing a temperature, i.e. multiply the network predictions/logits by a factor $\tau$. Experiments for $Cayley$ and $L_{orth}$ are done with $\tau=20$ (and $\tau=1$ for classical $Conv2D$).

We use classical data augmentation: random translation ($\pm 10\%)$, random rotation ($\pm 15$ degree), random horizontal flipping (0.5 probability), random contrast modification ($[0.8,1.2]$). For affine transformation zero-padding is used when required. The initial learning rate is set to 0.03, and linearly decreased down to $3\times10^{-4}$.

The $E_{rob}$ and $E_{lip}$ metrics are computed using the code provided by~\citet{trockman2021orthogonalizing}. 

\subsection{Imagenette Experiments}
\label{Imagenette_appendix}
The network architecture used for Imagenette data set is described in Table~\ref{tab:NN_archi_imagenette} (1.2 million parameters). Conv2D layer will depend on the configuration: classical $Conv2D$ for unconstrained reference configuration, $CayleyConv$ for $Cayley$ configuration 
\citep[we use][implementation]{trockman2021orthogonalizing},
%(we use ~\citep{trockman2021orthogonalizing} implementation),
and $L_{orth}$ regularization for $L_{orth}$ configuration (we use our implementation according to Definition~\ref{Lorth-def}). Weight initialization is done according to \textit{Glorot uniform}.

\begin{table}[ht]
  \caption{Imagenette Neural network architectures: Conv2D, GS2 is GroupSort2, InvDown is InvertibleDownsampling~\citep{trockman2021orthogonalizing}}
  \label{tab:NN_archi_imagenette}
  \centering
  \begin{tabular}{lll}
    \toprule
        Layer     & Parameters $(M,C,k,k)$ & Output size $(M,H,W)$  \\
    \midrule
    Input   & $ $ & $3\times 160\times 160$     \\
    Conv2D, GS2 & $(32,3,3,3)$   &    $32\times 160\times 160$   \\
    Conv2D, GS2  & $(34,32,3,3)$   & $34\times 160\times 160$     \\
    InvDown  & $ $   & $136\times 80\times 80$     \\
    Conv2D, GS2  & $(32,136,3,3)$   & $32\times 80\times 80$     \\
    Conv2D, GS2  & $(64,32,3,3)$   & $64\times 80\times 80$     \\
    Conv2D, GS2  & $(66,64,3,3)$   & $66\times 80\times 80$     \\
    InvDown  & $ $   & $264\times 40\times 40$     \\
    Conv2D, GS2  & $(64,264,3,3)$   & $64\times 40\times 40$     \\
    Conv2D, GS2  & $(96,64,3,3)$   & $96\times 40\times 40$     \\
    Conv2D, GS2  & $(98,96,3,3)$   & $98\times 40\times 40$     \\
    InvDown  & $ $   & $392\times 20\times 20$     \\
    Conv2D, GS2  & $(96,392,3,3)$   & $96\times 20\times 20$     \\
    Conv2D, GS2  & $(128,96,3,3)$   & $128\times 20\times 20$     \\
    Conv2D, GS2  & $(130,128,3,3)$   & $130\times 20\times 20$     \\
    InvDown  & $ $   & $520\times 10\times 10$     \\
    Conv2D, GS2  & $(128,520,3,3)$   & $128\times 10\times 10$     \\
    Conv2D, GS2  & $(160,128,3,3)$   & $160\times 10\times 10$     \\
    Conv2D, GS2  & $(162,160,3,3)$   & $162\times 10\times 10$     \\
    InvDown  & $ $   & $648\times 5\times 5$     \\
    Conv2D, GS2  & $(160,648,3,3)$   & $160\times 5\times 5$     \\
    Flatten, Dense & $ (10,4000) $  & $10$     \\
    \bottomrule
  \end{tabular}
\end{table}

Task loss is the classical cross-entropy (CE). As described in~\citet{bethune2022PayAttention}, 1-lipschitz property of orthogonal neural network may prevent learning  with CE and require introducing a temperature, i.e. multiply the network predictions/logits by a factor $\tau$. Experiments for $Cayley$ and $L_{orth}$ are done with $\tau=20$ (and $\tau=1$ for classical $Conv2D$).The initial learning rate is set to $5\times10^{-4}$, and linearly decreased down to $5\times10^{-6}$.

Input images are normalized per channel using the recommended mean and std ($[0.485, 0.456, 0.406]$, $[0.229, 0.224, 0.225]$). The only data augmentation used is random horizontal flipping (0.5 probability).

\section{Computing the Singular Values of \texorpdfstring{$\Kcal$}{Kcal}}
\label{sec:poweriter}

In this appendix, we describe methods for computing singular values of a 2D \Kcalname, with or without stride. The codes are provided in the \Deellip\footnote{\url{https://github.com/deel-ai/deel-lip}} library.

\subsection{Computing the Singular Values of \texorpdfstring{$\Kcal$}{} when \texorpdfstring{$S=1$}{}}

For convolutional layers without stride, $S=1$, we use the algorithm described in \citet{sedghi2018singular}. We describe the algorithm for 2D convolutional layers in Algorithm \ref{S=1_algo}. The algorithm provides the full list of singular values.

\begin{algorithm}
\caption{Computing the list of singular values of $\Kcal$, when $S=1$, \citep{sedghi2018singular}.}\label{S=1_algo}
\begin{algorithmic}[1]
\Require \Kbfname: $\Kbf\in\KbfS$, channel size: $N\geq k$
\Ensure list of the singular values of $\Kcal$:  $\sigma$
\Procedure{computeSingularValues}{$\Kbf$,$N$}
\State transforms = np.fft.fft2($\Kbf$, ($N$,$N$), axes=[0, 1])\Comment{np stands for numpy}
\State sigma = np.linalg.svd(transforms, compute\_uv=False)
\State \textbf{return} sigma
\EndProcedure
\end{algorithmic}
\end{algorithm}

\subsection{Computing the Smallest and the Largest Singular Value of \texorpdfstring{$\Kcal$}{Kcal} for any Stride \texorpdfstring{$S$}{S}}
For convolutions with stride, $S>1$, there is no known practical algorithm to compute the list of singular values $\sigma$. In this configuration, we use the well-known power iteration algorithm associated with
a spectral shift to compute the smallest and the largest singular value ($\sigma_{min},\sigma_{max}$) of $\Kcal$. We give the principle of the algorithm in Algorithm \ref{anyS_algo}. For clarity, we assume a function '$\lambda$ = power\_iteration($M$,$u_{init}$)', that applies the power iteration algorithm to a symetric matrix $M$ starting from a random vector $u_{init}$, and returns its largest eigenvalue $\lambda$. In practice, of course, we cannot construct $M$ and the implementation must use the usual functions that apply $\Kcal$ and $\Kcal^T$. A detailed python implementation is provided in the \Deellip library.

\begin{algorithm}        
\begin{algorithmic}              
\Require \Kbfname: $\Kbf\in\KbfS$, channel size: $N\geq k$, stride parameter: $S\geq 1$
\Ensure the smallest and the largest singular value of $\Kcal$:  $(\sigma_{min},\sigma_{max})$
\Procedure{computeMinAndMaxSingularValues}{$\Kbf$,$N$,$S$}
%\State compute $\Kcal$ with $\Kbf$,$N$,$S$ 
\State \If{$CS^2 \geq M$} \Comment{RO case}
\State u = np.random.randn($M$,$N$,$N$)
\State lambda\_1 = power\_iteration($\Kcal\Kcal^T$, u)
\State bigCste = 1.1* lambda\_1
\State u = np.random.randn($M$,$N$,$N$)
\State lambda\_2 = power\_iteration(bigCste$.\Id_{MN^2} - \Kcal\Kcal^T$, u)
\Else\Comment{CO case}
\State u = np.random.randn($C$,$SN$,$SN$)
\State  lambda\_1 = power\_iteration( $\Kcal^T\Kcal$, u )
\State  bigCste = 1.1* lambda\_1
\State u = np.random.randn($C$,$SN$,$SN$)
\State  lambda\_2 = power\_iteration(bigCste$.\Id_{CS^2N^2} - \Kcal^T\Kcal$, u)
\State \EndIf
\State sigma\_max = np.sqrt(lambda\_1)
\State sigma\_min = np.sqrt(bigCste-lambda\_2)

\State \textbf{return} (sigma\_min,sigma\_max)
\EndProcedure
\end{algorithmic}
\caption{Computing $(\sigma_{min},\sigma_{max})$, for any $S\geq 1$.\label{anyS_algo}}     
\end{algorithm}

\bibliography{ref}

\fi %keepsupplementary
\end{document}